\documentclass[11pt,a4paper]{article}

\usepackage{geometry}                % See geometry.pdf to learn the layout options. There are lots.
\usepackage{amsmath}
\usepackage{amsthm}
\usepackage{amssymb}
\usepackage{hyperref}
\usepackage[usenames,dvipsnames]{xcolor}
\hypersetup{
	colorlinks=true,
	linkcolor=blue,
	citecolor=Green,
	filecolor=magenta,      
	urlcolor=olive,
	linktoc=page
}

\usepackage{placeins}
\usepackage{enumitem}
\usepackage{graphicx}
\usepackage{ amssymb }
\usepackage{epstopdf}
\usepackage[utf8]{inputenc}
\usepackage[english]{babel}
\usepackage{amsthm}
\usepackage{ dsfont }
\usepackage[square,sort,comma,numbers]{natbib}
\usepackage{amsmath}
\usepackage{amsfonts}
\usepackage{amssymb}
\usepackage{scalerel}[2016/12/29]
\usepackage{bbm}
\usepackage{amsmath,amsfonts,amsthm,amssymb}
\usepackage{fancyhdr}
\usepackage{float}
\usepackage{ amssymb }

\theoremstyle{plain}
\newtheorem{theorem}{Theorem}[section]
\newtheorem{corollary}[theorem]{Corollary}
\newtheorem{prop}[theorem]{Proposition}% reset theorem numbering for each chapter
\newtheorem{lemma}[theorem]{Lemma}

\theoremstyle{definition}
\newtheorem{remark}[theorem]{Remark}
\theoremstyle{plain}
\newtheorem{innercustomthm}{Theorem}
\newenvironment{customthm}[1]
  {\renewcommand\theinnercustomthm{#1}\innercustomthm}
  {\endinnercustomthm}

\newcommand{\inte}{\mathbb{Z}}
\newcommand{\rat}{\mathbb{Q}}
\newcommand{\nat}{\mathbb{N}}
\newcommand{\real}{\mathbb{R}}

\newcommand{\prob}{\mathbb{P}}
\newcommand{\expt}{\mathbb{E}}
\newcommand{\indic}{\mathbbm{1}}
\newcommand{\Mod}[1]{\ (\text{mod}\ #1)}
\newcommand{\floor}[1]{{\left\lfloor #1 \right\rfloor}}
\newcommand{\ceil}[1]{{\left\lceil #1 \right\rceil}}

\newcommand{\sset}{\subset}

\newcommand{\al}{\alpha}
\newcommand{\Om}{\Omega}
\newcommand{\mathforall}{\text{ for all }}

\newcommand{\mathand}{\;\text{and}\;}

\newcommand{\mathor}{\;\text{or}\;}

\newcommand{\mathas}{\;\text{as}\;}

\newcommand{\ga}{\gamma}
\newcommand{\Ga}{\Gamma}
\newcommand{\ep}{\epsilon}

\newcommand{\de}{\delta}

\newcommand{\sig}{\sigma}

\newcommand{\del}{\partial}

\newcommand{\scrA}{\mathcal{A}}

\newcommand{\scrP}{\mathcal{P}}
\newcommand{\scrD}{\mathcal{D}}

\newcommand{\scrM}{\mathcal{M}}

\newcommand{\scrL}{\mathcal{L}}

\newcommand{\scrX}{\mathcal{X}}

\newcommand{\card}[1]{\left\vert #1 \right\vert}

\newcommand{\close}[1]{\mkern 1.5mu\overline{\mkern-1.5mu#1\mkern-1.5mu}\mkern 1.5mu}
\newcommand{\supp}{\text{supp}}
\newcommand{\Z}{\mathds{Z}}
\newcommand{\ddd}{\mathellipsis}
\usepackage{MnSymbol}

\newcommand{\Lip}{\text{\fontfamily{ppl}\selectfont Lip}}

\usepackage{yfonts}
\usepackage{ wasysym }

\newcommand{\eqd}{\stackrel{d}{=}}

\newcommand{\cvgd}{\stackrel{d}{\to}}

\newcommand{\X}{\times}

\newcommand{\rev}{\text{rev}}
\newcommand{\id}{\text{id}}
\newcommand{\as}{\text{almost surely}}

\newcommand{\lf}{\left}
\newcommand{\rg}{\right}
\usepackage{titling}

\newcommand\blfootnote[1]{%
	\begingroup
	\renewcommand\thefootnote{}\footnote{#1}%
	\addtocounter{footnote}{-1}%
	\endgroup
}

\author{Duncan Dauvergne}                
\title{The Archimedean limit of random sorting networks }

\begin{document}
\maketitle

\begin{abstract}
	
	A sorting network\blfootnote{MSC classes -- Primary: 60C05, Secondary: 05E15, 05A05, 68P10} (also known as a reduced decomposition of the reverse permutation), is a shortest path from $12 \cdots n$ to $n \cdots 21$ in the Cayley graph of the symmetric group $S_n$ generated by adjacent transpositions. We prove that in a uniform random $n$-element sorting network $\sig^n$, all particle trajectories are close to sine curves with high probability. We also find the weak limit of the time-$t$ permutation matrix measures of $\sig^n$. As a corollary of these results, we show that if $S_n$ is embedded into $\real^n$ via the map $\tau \mapsto (\tau(1), \tau(2), \dots \tau(n))$, then with high probability, the path $\sig^n$ is close to a great circle on a particular $(n-2)$-dimensional sphere in $\real^n$. These results prove conjectures of Angel, Holroyd, Romik, and Vir\'ag. \end{abstract}

\begin{figure}[H]
	\centering
	\includegraphics[scale= 1]{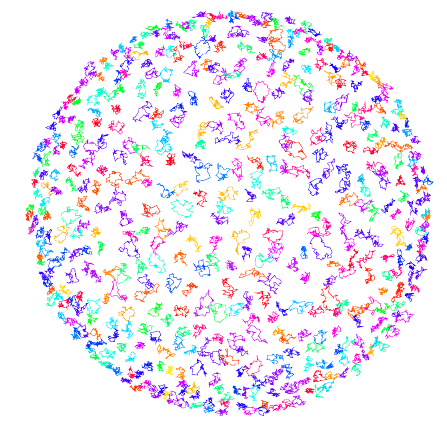}
	\caption{The half-way permutation matrix evolution for a uniform $500$-element sorting network. As a corollary of our main results, we prove that all of the paths in this evolution are localized, and that the distribution of these localized paths is the projected surface area measure of the $2$-sphere onto the unit disk. This figure is from \cite{angel2007random}.}
	\label{fig:window}
\end{figure}

\pagebreak

\tableofcontents

\section{Introduction}
\label{S:intro}

Consider $n$ numbers $1, \dots, n$ listed in increasing order. An $n$-element \textbf{sorting network} is a way of reversing this list from increasing to decreasing order by using a minimal number of adjacent swaps (see Figure \ref{fig:wiring} for an example with $n = 4$). This minimal number is $N = {n \choose 2}$. For $i \in \{1, \dots, N\}$, let $k_i \in \{1, \dots, n-1\}$ be the location of the $i$th swap: that is, the particles at locations $k_i$ and $k_{i+1}$ get swapped at stage $i$.

\medskip

We can equivalently define sorting networks in the following way. Let $\Ga(S_n)$ be the Cayley graph of the symmetric group  $S_n$ with generators given by the adjacent transpositions $(i, i + 1)$ for $i = 1, \ddd, n -1$. Then a sorting network is a shortest path from the identity permutation $\id_n = 12 \cdots n$ to the reverse permutation $\rev_n = n \cdots 21$ in  $\Ga(S_n)$.

\medskip

The name sorting network comes from computer science, where sorting networks are viewed as $N$-step algorithms for sorting a list of $n$ numbers. At step $i$,  the sorting network algorithm sorts the numbers at positions $k_i$ and $k_i + 1$ into increasing order. This process sorts any list in $N$ steps.

\medskip

In combinatorics, sorting networks are known as {\bf reduced decompositions} or \textbf{reduced words} for the reverse permutation, as any sorting network can be represented as a minimal length decomposition of the reverse permutation as a product of adjacent transpositions: $\rev_n = (k_N, k_N + 1) \cdots (k_1, k_1 + 1)$. 

\medskip

Stanley \cite{stanley1984number} showed that the number of $n$-element sorting networks is equal to
\begin{equation}
\label{E:stanley}
\frac{{n \choose 2}!}{(2n-3)^1(2n-5)^2 \cdots 3^{n-2}1^{n-1}},
\end{equation}
and he observed that this is the same as the number of standard Young tableaux of staircase shape $(n-1, n-2, \dots, 1)$. Stanley's argument was based on properties of symmetric functions and did not yield a bijective proof of the connection with Young tableaux. A few years later, Edelman and Greene found an explicit bijection \cite{edelman1987balanced}. 

\medskip

Since these seminal works, the combinatorics of sorting networks and reduced decompositions of other permutations and of elements of other Coxeter groups has been studied in great detail, revealing interesting connections with many other areas. We highlight a few of these here. For more connections and background, see Bj\"orner and Brenti \cite{bjorner2006combinatorics} and Garsia \cite{garsia2002saga}.

\medskip

Reduced decompositions can be used to give a combinatorial interpretation of Schubert polynomials, see Billey, Jockusch, and Stanley \cite{billey1993some}, Billey and Haiman \cite{billey1995schubert} and Manivel \cite{manivel2001symmetric}. They arise in the study of Bott-Samelson resolutions of Schubert varieties, see Magyar \cite{magyar1998schubert}, and are a useful tool in studying matrix factorizations, total positivity, and canonical bases, e.g. see Berenstein, Fomin, and Zelevinsky \cite{berenstein1996parametrizations} and Leclerc and Zelevinsky \cite{leclerc1998quasicommuting}.

\medskip

Equivalence classes of reduced decompositions of permutations are in bijection with rhombic tilings of certain polygons, see Elnitsky \cite{elnitsky1997rhombic} and Tenner \cite{tenner2006reduced}.  Permutations of Coxeter group elements that avoid certain patterns can also be characterized by properties of their reduced decompositions, see Lascoux and Schutzenberger \cite{lascoux1985schubert}, Stembridge \cite{stembridge1997some}, and Tenner \cite{tenner2006reduced}. Reduced decompositions are also connected to the study of point configurations in the plane and pseudoline arrangements, e.g. see the work of Goodman and Pollack \cite{goodman1980combinatorial} and subsequent papers.  

\medskip

A natural direction of inquiry regarding sorting networks is to try to understand their asymptotic properties as the number of elements $n$ approaches infinity. By applying Stirling's formula to Equation \eqref{E:stanley}, it is fairly straightforward to find asymptotics for the number of sorting networks. Beyond this, the next logical question to ask is what a typical (i.e. uniformly chosen) sorting network looks like as $n$ approaches infinity. This direction of inquiry has been incredibly fruitful for understanding other objects in algebraic combinatorics, e.g.\ Young tableaux and domino/lozenge tilings.

\medskip

In \cite{angel2007random} Angel, Holroyd, Romik, and Vir\'ag initiated the study of uniform random $n$-element sorting networks. They studied the global limiting behaviour of the space-time swap distribution, rescaled particle trajectories, time-$t$ permutation matrices, and the Cayley graph path itself.

\medskip

They proved a law of large numbers for the space-time swap distribution, and based on strong numerical evidence, made conjectures about the limiting behaviour of the other three objects. One of the main difficulties they faced in proving these conjectures is that unlike with many probabilistic models where global limiting results have been proven (e.g.\ random walks, classical interacting particle systems, random tilings, random graphs), the combinatorics of sorting networks cannot be easily reduced to a series of local rules. In this paper, we overcome this difficulty and prove their conjectures.

\begin{figure}
	\centering
	\includegraphics[scale=0.7]{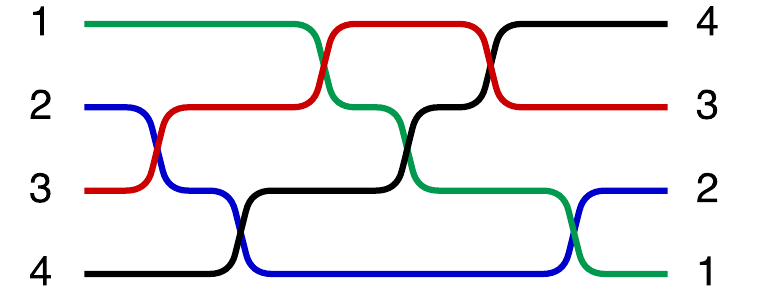}
	\caption{A ``wiring diagram" for a sorting network with $n = 4$. In
		this diagram, trajectories are drawn as continuous curves for
		clarity.}
	\label{fig:wiring}
\end{figure}

\subsection{Main limit theorems}
\label{SS:main-thms}

We will think of the elements $\{1, \ddd, n \}$ as particles being sorted in time (see Figure \ref{fig:wiring}). We use the notation $\sig(x, t) = (k_\floor{t}, k_\floor{t} + 1) \ddd (k_1, k_1 + 1)(x)$ for the position of particle $x$ at time $t\in [0, N]$ in a sorting network $\sig$. We call $(k_1, \dots, k_N)$ the \textbf{swap sequence} for $\sig$.

\bigskip

\noindent {\bf I. Rescaled particle trajectories.} \qquad  For a sorting network $\sig$, define the global trajectory
$$
\sig_G(x, t) = \frac{2\sig(x, Nt)}n - 1.
$$
The function $\sig_G(x, \cdot):[0, 1] \to [-1, 1]$ is the trajectory of particle $x$, with time rescaled so that the sorting process finishes at time $1$, and space rescaled so that the trajectory stays in the interval $[-1, 1]$.
In \cite{angel2007random}, Angel et al.\ conjectured that with high probability, all particle trajectories in a uniform random sorting network are close to sine curves  (see Figure \ref{fig:sinecurves}). They proved that with high probability, all global trajectories in a random sorting network are close in uniform norm to some H\"older-$1/2$ curve with H\"older constant $\sqrt{8}$.

\medskip

Our first theorem proves the sine curve conjecture from \cite{angel2007random}. Here and throughout the paper we use the notation $\sig^n$ for a uniform random $n$-element sorting network.

\begin{customthm}{1}[Sine curve limit]
	\label{T:sine-curves}
	For each $n$ there exist random variables $\{(A^n_i, \Theta^n_i) \in [0, 1] \X [0, 2\pi]\}_{i \in \{1, \ddd, n\}}$ such that for any $\ep > 0$, we have that
	$$
	\prob \lf( \max_{i \in [1, n] } \max_{t \in [0, 1]} \card{ \sig^n_G(i, t) - A_i^n \sin(\pi t + \Theta_i^n) } > \ep \rg) \to 0 \qquad \mathas \;\; n \to \infty.
	$$
\end{customthm}

\begin{figure}
	\centering
	\includegraphics[scale=0.8]{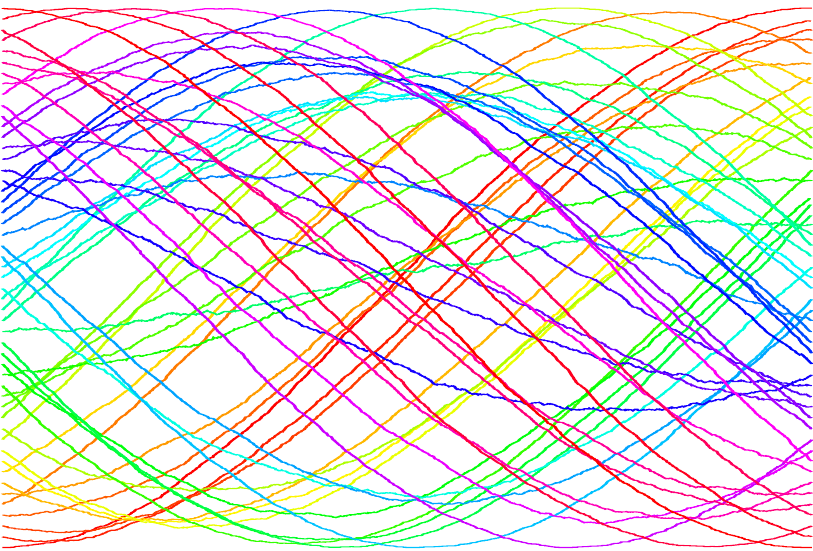}
	\caption{This is a diagram of selected particle trajectories in a 2000 element sorting network. This image is taken from \cite{angel2007random}.}
	\label{fig:sinecurves}
\end{figure}

\bigskip

\noindent {\bf II. Permutation Matrices.} \qquad For a uniform $n$-element sorting network $\sig^n$, define the random measure
$$
\rho^n_t= \frac{1}n \sum_{i=1}^n \delta \lf( \sig^n_G(i, 0), \sig^n_G(i, t) \rg).
$$
Here $\de(x, y)$ is a $\de$-mass at the point $(x, y)$.
The measure $\rho^n_t$ rescales the time-$t$ permutation matrix of $\sig^n$, placing atoms of weight $1/n$ at the positions of the ones. Define the {\bf Archimedean measure} $\mathfrak{Arch}$ on the square $[-1, 1]^2$ to be the measure with Lebesgue density 
$$
f(x, y) = \frac{1}{2\pi\sqrt{1 - x^2 - y^2}}
$$
on the unit ball $B(0, 1)$, and $0$ elsewhere. The Archimedean measure is the unique rotationally symmetric measure on $[-1, 1]^2$ whose linear projections are all uniform. It is obtained by projecting the surface area measure of the $2$-sphere onto the unit disk. The fact that this measure has uniform linear projections follows from Archimedes' theorem that if two parallel planes slice through a $2$-sphere, then the surface area cut out is simply proportional to the distance between the planes.
 Define the {\bf time-$t$ Archimedean measure} by
$$
(X, X \cos (\pi t) + Y \sin (\pi t)) \eqd \mathfrak{Arch}_{t}, \qquad \text{ where } \;\; (X, Y) \eqd \mathfrak{Arch}.
$$
Note that $\mathfrak{Arch} = \mathfrak{Arch}_{1/2}$.
In \cite{angel2007random}, Angel et al.\ conjectured that $\rho^n_t$ converges weakly to $\mathfrak{Arch}_{t}$ for every $t$ (see Figure \ref{fig:circles}). They proved that for any $t$, the support of $\rho^n_t$ lies in a particular octagon with high probability. In \cite{dauvergne1}, Dauvergne and Vir\'ag showed that the support of any limit of $\rho^n_t$ lies within the elliptical support of $\mathfrak{Arch}_t$. 
\begin{figure}
	\centering
	\includegraphics[scale=0.9]{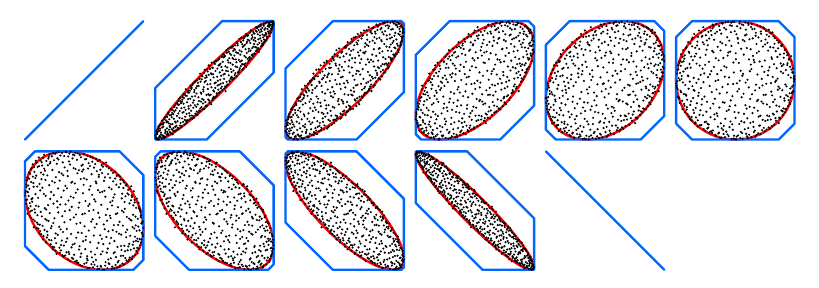}
	\caption{This is a diagram of the measures $\{\rho^n_t : t \in \{0, 1/10, 2/10, \ddd 1\} \}$ in a $500$-element sorting network. The blue octagons are the octagonal support bounds proved in \cite{angel2007random}, and the red ellipses are the supports of the measures $\mathfrak{Arch}_t$. This figure is from \cite{angel2007random}.}
	\label{fig:circles}
\end{figure}

\medskip

Our second theorem proves the weak convergence of the random measures $\rho^n_t$. We also show that the support of $\rho^n_t$ and $\mathfrak{Arch}_t$ are close. To state this theorem, recall that the Hausdorff distance between two sets $A, B \sset \real^2$ is
$$
d_H(A, B) = \max \lf\{\sup_{a \in A} \inf_{b \in B} d(a, b), \;\sup_{b \in B} \inf_{a \in A} d(a, b) \rg\}.
$$
\begin{customthm}{2}[Permutation matrix limit]
	\label{T:matrices}
	For any $t \in [0, 1]$,
	$
	\rho^n_t \to \mathfrak{Arch}_t
	$ 
	weakly 
	in probability
	as $n \to \infty$. That is, for any weakly open set $U$ in the space of probability measures on $[-1, 1]^2$ containing $\mathfrak{Arch}_t$, we have that
	$$
	\prob(\rho^n_t \in U) \to 1 \qquad \mathas \quad n \to \infty.
	$$
	Moreover, 
	$$
	d_H(\supp(\rho^n_t), \supp(\mathfrak{Arch}_t)) \to 0 \qquad \text{ in probability } \mathas n \to \infty.
	$$
\end{customthm}

Angel, Holroyd, Romik, and Vir\'ag \cite{angel2007random} also considered a permutation matrix evolution for $\sig^n$. Let $j \in \{1, \ddd, n\}$, and consider the random complex-valued function
$$
Z^n_j(t) = e^{\pi i t} \lf[\sig^n_G(j, t) + i\sig^n_G(j, t + 1/2) \rg], \qquad t \in [0, 1/2].
$$
For a fixed $t$, $(Z^n_1(t), \ddd, Z^n_n(t))$ is the set of points in the scaled permutation matrix for 
\[
\sig^n(\cdot, N(t + 1/2))(\sig^n)^{-1}(\cdot, Nt)
\]
after a counterclockwise rotation by $\pi t$ (see Figure \ref{fig:window}). Theorem \ref{T:sine-curves} guarantees that each of the paths $Z^n_j$ localize.
\begin{customthm}{3}[Path Localization]
	\label{T:unif-rotation}
	Let $\sig^n$ be a uniform random $n$-element sorting network.
	Then 
	$$
	\max_{j \in [1, n]} \max_{s, t \in [0, 1]} |Z^n_j(t) - Z^n_j(s)| \to 0 \qquad \text{in probability as } \;\; n \to \infty.
	$$
\end{customthm}

\bigskip

\noindent {\bf III. Great Circles.} \qquad We can embed the vertices of $\Ga(S_n)$ into $\real^n$ by sending the permutation $\tau \in S_n$ to the point $\close{\tau} = (\tau(1), \ddd, \tau(n)) \in \real^n$. For any $\tau \in S_n$, the point $\close{\tau}$ lies on the $(n-2)$-sphere $\mathbb{S}^{n-2} = \mathbb{L}_n \cap \mathbb{K}_n$, where
\begin{align*}
\mathbb{L}_n =& \lf\{ (x_1, \ddd, x_n) \in \real^n : \sum_{i=1}^n x_i = \frac{n(n+1)}2 \rg\} \qquad \mathand \\
\mathbb{K}_n =& \lf\{ (x_1, \ddd, x_n) \in \real^n : \sum_{i=1}^n x_i^2 = \frac{n(n+1)(2n+1)}6\rg\}.
\end{align*}
If we also embed the edges of $\Ga(S_n)$ as straight lines between the embedded vertices, we get an object called the {\bf permutahedron} (see Figure \ref{fig:permutahedron}). 

\medskip

\begin{figure}
	\centering
	\includegraphics[scale=0.073]{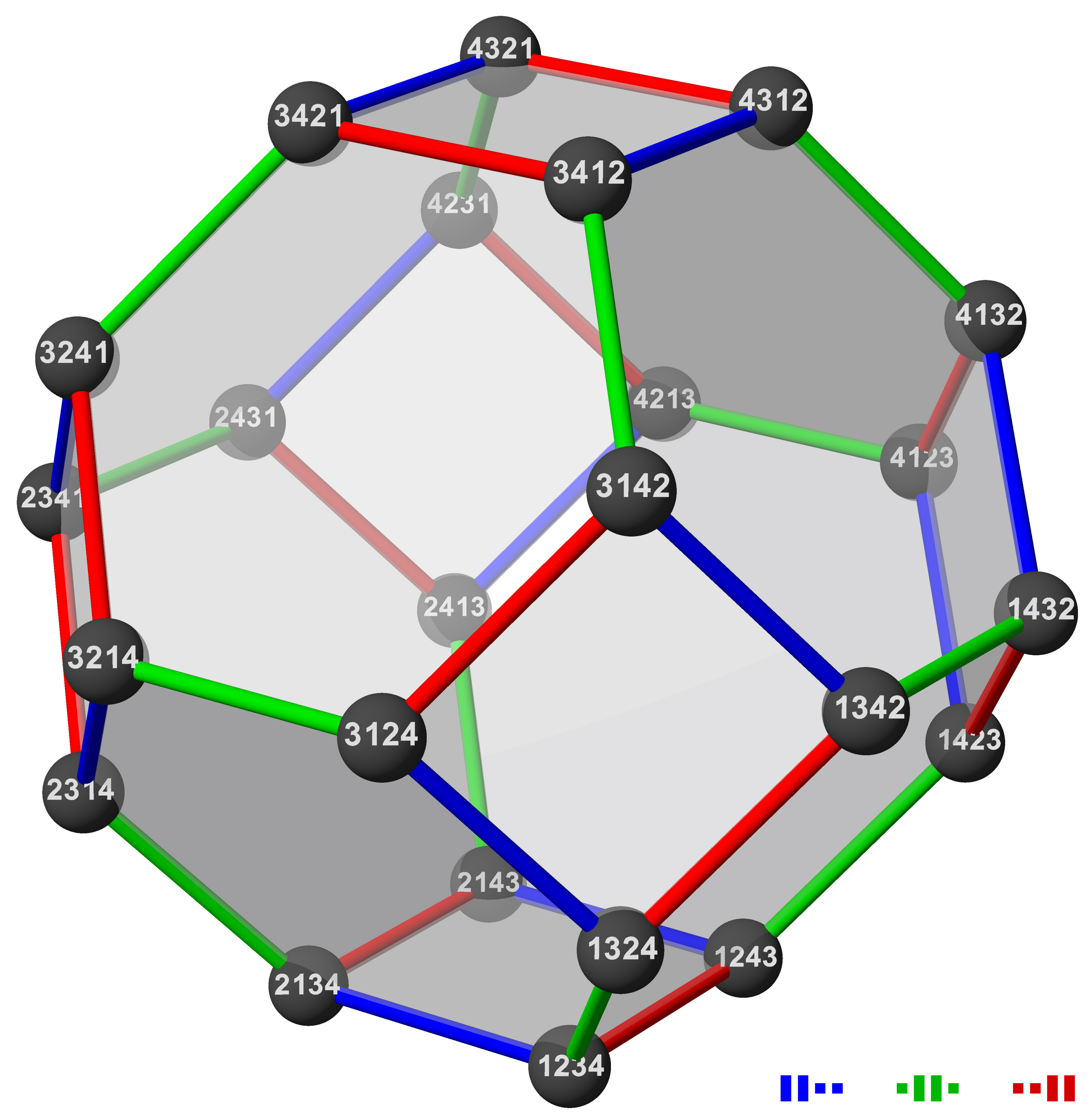}
	\caption{The permutahedron for $S_4$.}
	\label{fig:permutahedron}
\end{figure}
In \cite{angel2007random}, Angel et al.\ conjectured that with high probability, a uniform random sorting network is close to a great circle in $\mathbb{S}^{n-2}$ under this embedding. They showed that this conjecture implies the other global limiting results for uniform random sorting networks. Our strongest theorem proves this conjecture.

\medskip

For two functions $f, g: [0, 1] \to \real^n$,  define the distance function
$$
d_\infty(f, g) = \sup_{t \in [0, 1]} ||f(t) - g(t)||_\infty.
$$
This is the uniform norm on functions, where the pointwise distance is the $L^\infty$ distance. Also, for an $n$-element sorting network $\sig$, we define the embedded and time-rescaled path
$
\bar{\sig}:[0, 1] \to \mathbb{S}^{n-2}
$
by letting 
$$
\bar{\sig}(t) = (\sig(1, Nt), \sig(2, Nt), \dots, \sig(n, Nt)).
$$

\begin{customthm}{4}[Great circles]
	\label{T:geom-limit}
	Let $\bar{\sig}^n$ be the embedding of $\sig^n$ into $\mathbb{S}^{n-2}$. For every $n$, there exists a random path $C_n:[0, 1] \to \mathbb{S}^{n-2}$ such that $C_n$ is a constant-velocity parametrization of an arc of a great circle in $\mathbb{S}^{n-2}$ starting at $(1, \ddd, n)$ and finishing at $(n, \ddd, 1)$, and such that 
	$$
	\frac{d_\infty(\bar{\sig}^n, C_n)}{n} \to 0
	\qquad \text{in probability as} \;\;  n \to \infty.
	$$
\end{customthm}

Note that it is easy to find sorting networks that aren't close to a great circle in $\mathbb{S}^{n-2}$. For example, the ``bubble sort" sorting network given by the swap sequence 
$$
(1, 2, \ddd, n-1, 1, 2, \ddd, n-2, \ddd, 1, 2, 1)
$$
is $d_\infty$-distance $n - 1 - o(1)$ from any great circle.

\subsection{Random trajectory limits}

To prove the main theorems of Section \ref{SS:main-thms}, we first analyze the limit of a random particle trajectory. This approach was first considered by Rahman, Vir\'ag, and Vizer \cite{rahman2016geometry}.

\medskip

Let $\scrD$ be the closure in the uniform norm $||\cdot||_u$ of the space of all possible sorting network trajectories $\sig_G(x, \cdot):[0, 1] \to [-1, 1]$. The space $(\scrD, ||\cdot||_u)$ is a complete separable metric space. The only functions in $\scrD$ are continuous functions and the sorting network trajectories themselves.

\medskip
Let $Y_n \in \scrD$ be a uniform $n$-element sorting network trajectory. That is, if $\sig^n$ is a uniform $n$-element sorting network, and $I_n$ is an independent uniform random variable on $\{1, \ddd, n\}$, then
$$
Y_n = \sig^n_G(I_n, \cdot).
$$
We refer to $Y_n$ as the {\bf trajectory random variable} of $\sig^n$.
In \cite{dauvergne1}, Dauvergne and Vir\'ag proved that the sequence $\{Y_n\}_{n \in \nat}$ is precompact in distribution, and that any subsequential limit is almost surely Lipschitz. In this paper, we show that $\{Y_n\}_{n \in \nat}$ converges in distribution, and identify its limit as the Archimedean path, a random element of $\scrD$ first described in \cite{rahman2016geometry}, see Conjecture 1 from that paper and surrounding discussion.

\begin{customthm}{5}[The weak trajectory limit]
	\label{T:weak-limit}
	Let $(X, Z) \sim \mathfrak{Arch}$, and define the Archimedean path $\scrA \in \scrD$ by $\scrA(t) = X\cos(\pi t) + Z\sin(\pi t)$. Then
	$$
	Y_n \cvgd \scrA \qquad \mathas \; n \to \infty.
	$$
\end{customthm}

Theorem \ref{T:weak-limit} will be used in the proof of all our main theorems from Section \ref{SS:main-thms}. Most of the paper is devoted to its proof. We note here that we can equivalently write 
$$
\scrA(t) = \sqrt{1 - V^2}\sin(\pi t + 2\pi U),
$$
where $V$ and $U$ are independent uniform random variables on $[0, 1]$.

\bigskip

\noindent \textbf{Random $m$-out-of-$n$ sorting networks.} \qquad 
We will also use Theorem \ref{T:weak-limit} to identify the limit of random $m$-out-of-$n$ subnetworks. This answers a question of Angel and Holroyd \cite{angel2010random}. This limit can also be by found by using the stronger great circle theorem (Theorem \ref{T:geom-limit}), as was noted in \cite{angel2010random}.

\medskip

Let $\sig$ be an $n$-element sorting network. For $A \sset \{1, \dots, n\}$, let $\sig|_A$ be the $|A|$-element sorting network given by restricting $\sig$ to the set $A$. Specifically, for $i \in \{1, \dots, N\}$, let $\sig^*_A(\cdot, i)$ be the relative ordering of the particles in $A$ in the permutation $\sig(\cdot, i)$. This gives a sequence of $N$ permutations $\{\sig^*_A(\cdot, i) \in S_{|A|} : i \in \{1, \dots N\} \}$. Removing duplicates gives the permutation sequence for the sorting network $\sig|_A$. For $m < n$, let the \textbf{random $m$-out-of-$n$ subnetwork $\tau^n_m$} be the restriction of $\sig^n$ to a uniform $m$-element subset of $\{1, \ddd, n\}$, chosen independently from $\sig^n$. 

\medskip

Let $\{x_1, \ddd, x_n\}$ be a set of points in $\real^2$ in general position, and such that no two pairs of points determine parallel lines. Label the points in order of increasing $x$-coordinate. For all but finitely many angles $\theta$, listing the labels of the points $\{x_1, x_2, \ddd, x_n\}$ in increasing order of their horizontal projections after rotation by $\pi \theta$  gives a permutation $\tau_\theta$. The \textbf{geometric sorting network} associated to $\{x_1, \ddd, x_n\}$ is simply the sequence of permutations $\{\tau_\theta, \theta \in [0, 1] \}$ listed in order of increasing $\theta$. 

\begin{customthm}{6}[The subnetwork limit]
	\label{T:subnetwork}
	Let $\{X_1, X_2, \ddd, X_m \}$ be random points in the unit ball $B(0, 1)$ sampled from the Archimedean distribution $\mathfrak{Arch}$, and let $\tau_m$ be the associated geometric sorting network. Then
	$$
	\tau^n_m \cvgd \tau_m \qquad \mathas \; n \to \infty.
	$$ 
\end{customthm}

Note that while random $m$-out-of-$n$ subnetworks are geometric with high probability for fixed $m$, the same is not true of entire sorting networks. Indeed, Angel, Gorin, and Holroyd \cite{angel2012pattern} showed that a uniform sorting network is, with high probability, not geometric.

\subsection{The local speed distribution} 

As a by-product of the proof of Theorem \ref{T:weak-limit}, we find the distribution of speeds in the local limit of random sorting networks. To state this result, we first give an informal description of this limit (a precise description is given in Section \ref{S:local}). The existence of this limit was established independently by Angel, Dauvergne, Holroyd, and Vir\'ag \cite{angel2017local}, and by Gorin and Rahman \cite{gorin2017}. Define 
$$
U_n (x, t) = \sig^n(\floor{n/2} + x, nt) - \floor{n/2}.
$$
Each path $U_n(x, \cdot)$ is a locally scaled particle trajectory.
With an appropriate notion of convergence, we have that 
$$
U_n \cvgd U,
$$
where $U$ is a random function from $\Z \X [0, \infty) \to \Z$. The process $U$ is the local limit at the centre of the sorting network. We can also take a local limit centred at particle $\floor{\al n}$ for any $\al \in (0, 1)$. The result is the process $U$ with time rescaled by a semicircle factor $2\sqrt{\al(1 - \al)}$. 

\medskip

In \cite{dauvergne1}, Dauvergne and Vir\'ag used a stationarity argument to conclude that particles in $U$ have asymptotic speeds. Specifically, they showed that for every $x \in \Z$, the limit
$$
S(x) = \lim_{t \to \infty} \frac{U(x, t) - U(x, 0)}{t} \qquad \text{exists }\as.
$$ 
By spatial stationarity of $U$, the speed $S(x)$ has distribution $\mu$ independent of $x$. In this paper, we identify $\mu$.

\begin{customthm}{7}
	\label{T:main-2}
	The measure $\mu$ is the arcsine distribution on $[-\pi, \pi]$ given by the Lebesgue density
	$$
	f(x) = \frac{1}{\pi\sqrt{\pi^2 - x^2}}.
	$$
\end{customthm}

\subsection{Related work and some random sorting variants}

All of the prior work discussed so far on random sorting networks relies crucially on Edelman-Greene bijection. This bijection has also been used by Reiner \cite{reiner2005note} and Tenner \cite{tenner2014expected} to study the frequency of particular substrings in the swap sequence of a random sorting network. Little \cite{little2003combinatorial} found another bijection between the set of sorting networks and Young tableaux of staircase shape, and Hamaker and Young \cite{HY} proved that these bijections coincide. 

\medskip

Problems involving limits of sorting networks under different measures have been considered by Angel, Holroyd, and Romik \cite{angel2009oriented}, and also by Young \cite{young2014markov}. Uniform ``relaxed" random sorting networks have been analyzed by Kotowski and Vir\'ag \cite{kotowski2016limits} (see also \cite{rahman2016geometry}). Uniform random sorting networks that avoid intermediate permutations with a $132$-pattern have been analyzed by Linusson, Potka, and Sulzgruber \cite{linusson2018random}.

\medskip

A natural follow-up to studying random sorting networks would be to try to prove analogous theorems for random reduced decompositions of other permutations.
When the target permutation is vexillary (i.e. it avoids the pattern $2143$) then the Edelman-Greene bijection allows for efficient sampling of random reduced decompositions so the limit object can at least be guessed at, see Gross and Marsaglia \cite{gross2012vexillary} for simulations. Much of the combinatorics that is necessary for analyzing random reduced decompositions of the reverse permutation breaks down for general vexillary permutations, so the proofs from this paper and previous random sorting network papers cannot be adapted directly. Reduced decompositions of random $321$-avoiding permutations are in correspondence with skew standard Young tableaux of a shape determined by the permutation, see Section 2 in \cite{billey1993some}, so a general theory of random reduced decompositions would also subsume the theory of random Young tableaux.

\medskip

There are also closely related sorting contexts which exhibit similar Archimedean limiting behaviour. For the Coxeter groups $\{B_n : n \in \mathbb{N}\}$, reduced decompositions of the long element can again be efficiently sampled by a variant of the Edelman-Greene bijection, see Haiman \cite{haiman1992dual}. It seems possible that the same broad strategy used to show the Archimedean limit in this paper could work for type $B$ reduced decompositions.

\medskip

Another variant of the Edelman-Greene bijection allows for sampling of random involution words for the reverse permutation in $S_n$, see Marberg  \cite{marberg2020symplectic} for relevant definitions and the description of the algorithm. In this case, while simulations still suggest an Archimedean limit, a proof along the lines of the current paper seems further away. Indeed, random involution words lack a basic time stationarity property used heavily in this paper and in previous work on random sorting networks.

\medskip

\subsection{A connection with fluid mechanics}

It turns out that the Archimedean path appeared in the literature on fluid mechanics long before the first papers on random sorting networks. \footnote{And as far as I know, no one studying either sorting networks or fluid mechanics was aware of this connection until Laurent Miclo pointed this out at one of my talks after the first version of this paper appeared online.} We give a brief description of the context here. For more details and motivation, see Brenier \cite{brenier2008generalized} and references therein.

\medskip

Incompressible fluid flow in a compact connected subset $D \sset \mathbb{R}^d$ can be modelled by a system of PDEs: the Euler equations. A solution to these equations from time $0$ to $1$ is a function $g:[0, 1] \X D \to D$ with the property that $g(t, \cdot)$ is a volume and orientation preserving diffeomorphism on $D$ for all times $t$. When $g(0, a) = a$, we can think of the value $g(t, a)$ as representing the time-$t$ location of the parcel of fluid that started at location $a$. One way of finding solutions to the Euler equations is by looking for minimizers of the Dirichlet energy
\begin{equation}
\label{E:energy}
E(g) = \int_0^1 \int_D \frac{1}2 |\del_t g(t, a)|^2 da dt
\end{equation}
among functions with the same values at time $0$ and $1$. For a given admissible end state (i.e. a volume and orientation preserving diffeomeorphism) $g_1:D \to  D$, there may not exist a minimizer of \eqref{E:energy} with $g(0, a) = a$ and $g(1, \cdot) = g_1$, see \cite{shnirelman56geometry}. Because of this difficulty, Brenier \cite{brenier1989least} introduced a notion of generalized solutions of the Euler equations. These solutions allow particles to `split' and `cross each other'. Formally, solutions are now random functions $G:[0, 1] \X D \to D$ that minimize expected Dirichlet energy $\expt E(G)$ subject to the constraint that $G(t, \cdot)$ is uniformly distributed on $D$ for all $t \in [0, 1]$ (this is incompressibility), and subject to the initial and final conditions $G(0, a) = a$ and $G(1, a) = g_1(a)$ for all $a \in D$. The end state $g_1$ no longer needs to be either a diffeomorphism or orientation-preserving; it just needs to preserve volume and it can even be taken to be random. For a given $g_1$, a generalized solution of the Euler equations always exists and solves a set of `generalized' Euler equations.

\medskip

The case when $D = [-1, 1]$ and $g_1(a) = (-a)$ is one of the only scenarios where a non-deterministic generalized solution can be computed. In this case, the unique generalized solution to the Euler equations is $G(\cdot, a) = \scrA_a(\cdot)$, where $\scrA_a$ is the Archimedean path conditioned to start at $a$ (see \cite{brenier1989least}, Proposition 6.3 or \cite{brenier2008generalized}, Figure 3,6). In other words, if $U$ is an independent uniform random variable, then for every $a, b$, the random function $G(\cdot, U) \eqd \scrA$. 

\medskip

The fact that the Archimedean path minimizes Dirichlet energy was later also observed by Rahman, Vir\'ag and Vizer \cite{rahman2016geometry} and used to study random `relaxed' sorting networks by Kotowski and Vir\'ag \cite{kotowski2016limits}. 

\medskip

In one sense, it is not all that surprising that the limit of random sorting networks is an energy minimizer; hydrodynamic limits of interacting particle systems often minimize energy. On the other hand, it is the limiting version of the surprising and beautiful fact that random sorting networks typically follow great circles on a sphere. Moreover, unlike in many interacting particle systems, there does not seem to be an obvious way to prove that random sorting networks tend to minimize energy. In fact, in this paper we take an orthogonal approach to proving the main limit theorems, though still one based around minimizing a particular functional.

\medskip

It is interesting to note that for a general target permutation, the trajectory limit will not necessarily minimize Dirichlet energy. For example, the limit of random reduced decompositions of the permutation $(\floor{n/2}, \dots n, 1, \dots, \floor{n/2} -1)$ is not an energy minimizer (this can be deduced from Romik \cite{romik2015surprising}, Theorem 3.24). However, it is possible that a time-changed version of the limiting process will minimize energy and therefore align with a generalized solution of the Euler equations.

\subsection{Outline of the proofs and the paper}
Most of the paper is devoted to proving Theorem \ref{T:weak-limit}. The key to proving this theorem is the notion of particle flux across a curve $h$. Heuristically, particle flux measures the number of times that particles cross $h$ in a typical large-$n$ sorting network. It is defined in terms of the a priori unknown local speed distribution $\mu$.

\medskip

Particle flux is a useful quantity because any limit $Y$ of the trajectory random variables $Y_n$ must minimize this quantity among curves $h$ with $h(0) = -h(1)$. This is due to the fact that in any sorting network, every pair of particles swaps \textit{exactly} once, whereas any particle must cross a line with $h(0) = -h(1)$ \textit{at least} once. We define the particle flux functional and establish that it does indeed measure what it intends to in Section \ref{S:flux-intro}.

\medskip

The next three sections of the paper (Section \ref{S:unique}, Section \ref{S:path-speed}, and Section \ref{S:integral-formula}) are devoted to understanding properties of minimal flux paths and how they interact. The goal of these sections is to work towards a classification of minimal flux paths. First, by stationarity and symmetry arguments it will be enough to classify minimal flux paths with $h(0) = h(1) = 0$ satisfying $h \ge 0$. Our end goal will be to show that the only minimal flux paths in this class are the sine curves $a \sin (\pi t)$ for $a \in [0, 1]$. When going through these sections, we encourage the reader to keep in mind this goal.

\medskip

To understand minimal flux paths, we use a combination of analytic reasoning about how minimizers must behave with intuition based on both the combinatorics of sorting and previously established results about sorting networks. Two of the most useful observations guiding the proofs are: 
\begin{enumerate}[nosep, label=(\roman*)]
	\item  the number of particles that get within distance $r$ of the edge can be understood by using the semicircular law of large numbers for the swap distribution,
	\item since sorting network trajectories can only cross each other once, if they start at the same location then their initial speeds and their maximum heights must be in the same relative order.
\end{enumerate}
In Section \ref{S:unique}, we show that for every height $a \in [0, 1]$, there is exactly one minimal flux path $h_a$ with $h_a(0) = -h_a(1) = 0$, $h \ge 0$, and maximum height $a$. Moreover, paths in this family are ordered by their maximum height, and all these paths are unimodal and concave. 

\medskip

While these ordering, uniqueness, and regularity properties do not yet allow us to conclude that $h_a(t) = a \sin (\pi t)$, they do give enough structure to allow us to identify certain features of any limit $Y$ of $Y_n$. First, they allow us show that the local speed distribution $\mu$ matches the initial derivative distribution of any limit $Y$. Intuitively, this means that trajectories in a uniform sorting network typically follow straight lines on all time scales $t$ with $n \ll t \ll n^2$. Second, these properties allow us to explicitly identify the maximum height distribution of $Y$. This also uses the observation (i) above. All this is done in Section \ref{S:path-speed}.

\medskip

Now, at this point we have an explicit formula for the probability that $Y$ reaches a certain maximum height. Observation (ii) above relates this to the initial derivative distribution of $Y$, and hence to the local speed distribution $\mu$. Exploring this connection yields an integral transform formula involving $\mu$, which we obtain in Section \ref{S:integral-formula}. 

\medskip

By proving injectivity of this integral transform, we are able to identify $\mu$ as the arcsine distribution and then use this to prove that the only minimal flux paths are sine curves. This is done in Section \ref{S:transform} and completes the proof of Theorem \ref{T:weak-limit}.

\medskip

In the last two sections of the paper, we use Theorem \ref{T:weak-limit} to establish our stronger limit theorems. The basic idea behind the stronger limit theorems is that if most particles in a sorting network move along sine curves, then the fact that every pair of particles swaps exactly once forces all other particles to follow sine curves as well: the movements of the many control the movements of the few. In Section \ref{S:strong-limit}, we combine Theorem \ref{T:weak-limit} with bounds from \cite{angel2007random} to prove Theorems \ref{T:sine-curves}, \ref{T:matrices}, and \ref{T:unif-rotation}. Finally, in Section \ref{S:geom-limit}, we combine Theorem \ref{T:weak-limit} with Theorem \ref{T:sine-curves} to prove Theorem \ref{T:geom-limit}.

\section{Preliminaries}
\label{S:local} 
In this section, we introduce the necessary background about sorting networks. The most basic observation about uniform $n$-element sorting networks is that they exhibit a type of time stationarity. This was first observed in \cite{angel2007random}.

\begin{theorem}
\label{T:time-stat}
Let $(K^n_1, K^n_2, \ddd, K^n_N)$ be the swap sequence of a random $n$-element sorting network $\sig^n$. Then 
$$
\lf(K^n_1, K^n_2, \ddd, K^n_N\rg) \eqd \lf(n - K^n_N, K^n_1, K^n_2, \ddd, K^n_{N - 1}\rg).
$$
\end{theorem}
We will repeatedly use time stationarity of sorting networks to reduce proofs to statements about the beginning of a sorting network. Using the Edelman-Greene bijection between sorting networks and Young tableaux, Angel, Holroyd, Romik, and Vir\'ag \cite{angel2007random} also found an explicit formula for the distribution of $K^n_1$.

\begin{theorem}
\label{T:dist-k1}
Let $K^n_1$ be the location of the first swap of $\sig^n$. For any $i \in \{1, \dots, n-1\}$, we have that
$$
\prob(K^n_1 = i) = \frac{1}N \frac{[3\cdot5\cdot7\cdots(2i - 1)][3\cdot 5 \cdots (2(n-i) - 1)]}{[2\cdot4\cdot6\cdots(2i - 2)][2\cdot 4 \cdots (2(n-i) - 2)]} \le \frac{3}n.
$$
Moreover, if $\{i_n\}_{n \in \nat}$ is any sequence of integers such that $i_n/n\to (\al + 1)/2$ for some $\al \in (-1, 1)$, then
$$
n\prob(K^n_1 = i_n) \to \frac{4\sqrt{1 - \al^2}}{\pi} \qquad \mathas n \to \infty.
$$
\end{theorem}

\bigskip

\noindent \textbf{The local limit.} \qquad  Define a {\bf swap function} as a map $V:\inte \times [0, \infty) \to \Z$ with the following properties:

\smallskip

 \begin{enumerate}[nosep,label=(\roman*)]
  \item For each $x$, $V(x,\cdot)$ is cadlag with nearest neighbour jumps.
  \item For each $t$,  $V(\cdot,t)$ is a bijection from $\Z$  to $\Z$ (i.e. a permutation) and $V(\cdot, 0)$ is the identity.
  \item Define $V^{-1}(x, t)$ by $V(V^{-1}(x, t),t) = x$.  Then for each $x$, $V^{-1}(x, \cdot)$
    is a cadlag path.
  \item For any time $t \in (0, \infty)$ and any $x \in \Z$, 
  $$
  \lim_{s \to t^-} V^{-1}(x, s) = V^{-1}(x + 1, t) \qquad \text{if and only if} \qquad  \lim_{s \to t^-} V^{-1}(x +1, s) = V^{-1}(x, t).
  $$
  \end{enumerate}
  We think of a swap function as a collection of particle trajectories $\{V(x, \cdot) : x \in \Z\}$.
Condition (iv) guarantees that the only way that a particle at position $x$ can move up at time $t$ is if the particle at position $x+1$ moves down. That is, particles move by swapping with their neighbours. 

\medskip
  
  Let $\scrA$ be the space of swap functions endowed with the following topology. A sequence of swap functions $V_n \to V$ if each of the cadlag paths $V_n(x, \cdot) \to V(x, \cdot)$ and $V^{-1}_n(x, \cdot) \to V^{-1}(x, \cdot)$. Convergence of cadlag paths is convergence in the Skorokhod topology. We refer to a random swap function as a \textbf{swap process}.
  
  \medskip
  
For a swap function $V$ and a time $t \in (0, \infty)$, define
$$
V(\cdot, t, s) = V(V^{-1}(t, \cdot), t + s).
$$
The function $V(\cdot, t, s)$ is the increment of $V$ from time $t$ to time $t + s$. 

\medskip

Now for $i \in \{1, \dots n-1\}$, define the semicircle scaling factor 
$$
r_n(i) = \sqrt{1 - (2i/n - 1)^2},
$$
and consider the shifted, time-scaled swap process 
$$
U_{n}^{i}(x, s) = \sig^n \lf(i + x, \frac{ns}{r_n(i)} \rg) - i.
$$
Here recall that $\sig^n(i, t)$ is the location of particle $i$ at time $t$ in a uniform $n$-element sorting network $\sig^n$.
To ensure that $U_{n}^{i}$ fits the definition of a swap process, we can extend it to a random function from $\Z \X [0, \infty) \to \Z$ by letting $U_{n}^{i}$  be constant after time $(n-1)r_n(i)/2$, and with the convention that  $U_{n}^{i}(x, s)= x$ whenever $x \notin \{1 -i, \ddd, n - i\}$. In the swap processes $U_{n}^{i}$, all particles are labelled by their initial positions.
The following is shown in \cite{angel2017local}, and also essentially in \cite{gorin2017}. 

\begin{theorem}
\label{T:local}
There exists a swap process $U$ such that the following holds. For any $\al \in (-1, 1)$, and any sequence of integers $\{i_n\}_{n \in \nat}$ such that $i_n/n \to (\al + 1)/2$, we have that
  $$
U_{n}^{i_n}  \cvgd U \qquad \mathas \; n \to \infty.
  $$
The swap process $U$ has the following properties:
\begin{enumerate}[nosep,label=(\roman*)]
 \item $U$ is stationary and mixing of all orders with respect to the spatial shift $\tau U(x, t) = U(x + 1, t) - 1$.
  \item $U$ has stationary increments in time: for any $t \ge 0$, the
  process $U(\cdot, t, s)_{s\ge 0}$ has the same law
  as $U(\cdot,s)_{s\geq 0}$.
  \item $U$ is symmetric: $U(\cdot, \cdot) \eqd - \; U(- \; \cdot, \cdot)$.
  \item For any $t \in [0, \infty)$, $\prob($There exists $x \in \Z$ such that $U(x, t) \neq \lim_{s \to t^-} U(x, t)) = 0$.
  \end{enumerate}
  \smallskip
  
  \noindent Moreover, for any sequence of times $\{t_n : n \in \nat\}$ such that $(n-1)r_n(i)/2 - t_n \to \infty$ as $n \to \infty$,
  $$
U^{i_n}_n (\cdot, t_n, \cdot) \cvgd U \qquad \mathas n \to \infty.
$$ 
\end{theorem}

The main takeaway from Theorem \ref{T:local} is that at every location in a large random sorting network, the local swap process converges to a universal limiting object. The only difference between the limits at different locations comes from a rescaling of time by the semicircle factor $\sqrt{1- \al^2}$ anticipated by Theorem \ref{T:dist-k1}.

\medskip

Now, for a swap function $V$, let $W(V, t)$ be the number of times that the particles at locations $0$ and $1$ swap in the interval $[0, t]$. That is,
$$
W(V, t) = \card {\lf\{s \in (0, t] : 
  \lim_{r \to s^-} V^{-1}(0, r) = V^{-1}(1, s)  \rg\}}.
 $$

As a by-product of the proof of convergence of $U_n^{i_n}$ to $U$,  Angel et al.\ \cite{angel2017local} also found the expectation of $W(U, t)$.

\begin{theorem}[Proposition 7.10, \cite{angel2017local}]
\label{T:swaps}
Let $\al \in (-1, 1)$, and let $\{i_n\}_{n \in \nat}$ be any sequence of integers converging to $(\al + 1)/2.$ Then for any $t \in [0, \infty)$, we have that
$$
\expt W(U_n^{i_n}, t) \to \expt W(U, t) = \frac{4t}{\pi}.
$$
\end{theorem}
Note that the expected value $4t/\pi$ in Theorem \ref{T:swaps} is immediate from Theorem \ref{T:dist-k1}. Only the convergence in Theorem \ref{T:swaps} is non-trivial.
Now let $Q(V, t)$ be the number of swaps that particle $0$ makes by time $t$ in a swap function $V$. That is,
$$
Q(V, t) = \card {\Big\{ s \in (0, t] : \lim_{r \to s^-} V(0, r) \neq V(0, s) \Big\}}.
$$
Dauvergne and Vir\'ag \cite{dauvergne1} used a straightforward stationarity argument to prove an analogous result to Theorem \ref{T:swaps} for $Q(U, t)$. 
\begin{lemma}[Lemma 3.2, \cite{dauvergne1}]
\label{L:swaps-2}
For any $t \in [0, \infty)$, we have that
$$
\expt Q(U, t) = \frac{8t}{\pi}.
$$
\end{lemma}

The fact that $U$ is stationary in both time and space implies that the point process of swaps of a given particle $x$ in $U$ is also stationary. This realization combined with the ergodic theorem allows us to conclude that all particles in $U$ have asymptotic speeds. This observation was used in \cite{dauvergne1} to prove results about the relationship between the local and global limit. We use their results as a starting point in our proofs.

\begin{theorem}[Theorem 1.7, \cite{dauvergne1}]
 \label{T:local-speeds}
 For every $x \in \Z$, the limit
$$
S(x) = \lim_{t \to \infty} \frac{U(x, t) - U(x, 0)}{t}  \qquad \text{exists } \as.
$$
The random variables $S(x), x \in \Z$ satisfy $S(x) \eqd -S(x) \eqd S(0)$, and $S(0) \in [-\pi, \pi]$ almost surely.
\end{theorem}
Throughout the paper we let $\mu$ be the law of $S(0)$ and refer to $\mu$ as the {\bf local speed distribution}. One of the main difficulties overcome in this paper is in finding the local speed distribution $\mu$. We will slowly learn more information about $\mu$ throughout the paper, culminating in the proof that $\mu$ is the arcsine distribution on $[-\pi, \pi]$. 

\medskip

Dauvergne and Vir\'ag \cite{dauvergne1} were able to express limiting swap rates in $U$ in terms of $\mu$. Define 
\begin{equation*}
D_\mu^+(c) = \int(y - c)^+d \mu(y), \;\;\;\; D_\mu^-(c) = \int(y - c)^-d \mu(y) \;\;\;\;\mathand \;\;\;\; D_\mu(c) = \int|y - c|d \mu(y).
\end{equation*}

\begin{theorem}[Theorem 1.8, \cite{dauvergne1}]
\label{T:particle-swaps} Let $S(0)$ be the asymptotic speed of particle $0$ in $U$. For any $x \in \Z$, 
$$
\frac{Q(U, t)}t \to D_\mu(S(0)) \qquad \text{ almost surely and in } L^1.
$$
In particular, the random variables $Q(U, t)/t$ are uniformly integrable and $\expt D_\mu(S(0)) = 8/\pi$.
\end{theorem}
We also need an analogous result regarding crossings of lines in the local limit. Let $L(t) = ct + d$ be a line of constant slope $c$. For a swap function $V$, define
$$
C^+(V, L, t) = \card{\Big\{x \in \Z: V(x, 0) \le L(0), V(x, t) > L(t)\Big\}}.
$$
The quantity $C^+(V, L, t)$ is the total number of particles that are below $L$ at time $0$ and above $L$ at time $t$. We symmetrically define $C^-(V, L, t)$ as the total number of particles that are above $L$ at time $0$ and below $L$ at time $t$, and let $C(V, L, t) = C^-(V, L, t) + C^+(V, L, t)$.

\begin{theorem}[Theorem 5.7, \cite{dauvergne1}]
 \label{T:line-rate}
Let $L(t) = ct + d$. Then almost surely and in $L^1$, we have that
\begin{align*}
\lim_{t \to \infty} \frac{C^+(U, L, t)}t = D_\mu^+(c), \qquad \lim_{t \to \infty} \frac{C^-(U, L, t)}t = D_\mu^-(c), \qquad \lim_{t \to \infty}  \frac{C(U, L, t)}t =  D_\mu(c).
\end{align*}
\end{theorem}

We also record here a few basic facts about the functions $D_\mu$ and $D_\mu^+$ that will be used throughout the paper. These properties can be proven using basic facts about integrals, and the fact that $\mu$ is symmetric and supported in $[-\pi, \pi]$.

\begin{lemma}
\label{L:convex}
\begin{enumerate}[nosep, label=(\roman*)]
\item Both $D_\mu$ and $D_\mu^+$ are convex, $1$-Lipschitz functions.
\item For all $x$, we have that $D_\mu^+(x) \le D_\mu(x) \le |x| \vee \pi$.
\item Suppose that $L(t) = at + b$ is tangent to either $D_\mu$ or $D_\mu^+$. Then $b \in [0, \pi]$.
\item $D_\mu$ is a symmetric function, and hence minimized at $0$.
\item $D_\mu^+(\pi) = 0$, and $D_\mu(\pm \pi) = \pi$.
\end{enumerate}

\end{lemma}

\bigskip

\noindent \textbf{Subsequential limits of $Y_n$.} \qquad Recall that $Y_n$ is the trajectory random variable of $\sig^n$. We record here the main result of \cite{dauvergne1} regarding subsequential limits of $Y_n$. Here and throughout the paper, the phrase ``subsequential limit of $Y_n$'' always refers to a subsequential limit of $Y_n$ in distribution. 

\medskip

We say that a path $y \in \scrD$ is $g(y)$-Lipschitz if $y$ is absolutely continuous and if for almost every $t$, $|y'(t)| \le g(y(t))$. 
\begin{theorem}[Theorem 1.4, \cite{dauvergne1}]
\label{T:bounded-speed}
\begin{enumerate}[nosep, label=(\roman*)]
\item The sequence $\{Y_n\}$ is precompact in distribution.
\item Suppose that $Y$ is a subsequential limit of $Y_n$ (in distribution). Then
$$
\prob\bigg(Y \text{ is } \pi\sqrt{1-y^2}\text{-Lipschitz},\;Y(0) = -Y(1) \bigg) = 1.
$$
Moreover, $Y(t)$ is uniformly distributed on $[-1, 1]$ for every $t$.
\end{enumerate}
\end{theorem}

In addition, we observe here that any subsequential limit $Y$ of $Y_n$ inherits certain symmetries from $\sig^n$.

\begin{prop}
\label{P:Y-symmetries} Let $Y$ be any subsequential limit of $Y_n$.
\smallskip
\begin{enumerate}[nosep, label=(\roman*)]
\item Define $Y_t \in \scrD$ by 
$$
Y_t(s) = \begin{cases}
\;\;Y(s + t), \quad &s \le t.\\
\;\; -Y(s + t - 1), \quad &s > t.
\end{cases}
$$
For any $t \in [0, 1]$, we have that $Y_t \eqd Y$.
\item $Y \eqd - Y$.
\item Define $Z \in \scrD$ by $Z(t) = Y(1 - t)$. Then $Z \eqd Y$.
\end{enumerate}
\end{prop}

\begin{proof}
Property (i) follows from time stationarity of random sorting networks (Theorem \ref{T:time-stat}). Properties (ii) and (iii) follow from the corresponding properties of the swap sequence $(K^n_1, \dots, K^n_N)$ of $\sig^n$:
\[
\lf(K^n_1, K^n_2, \ddd, K^n_N\rg) \eqd \lf(n - K^n_1, n - K^n_2,  \ddd, n - K^n_{N}\rg) \eqd \lf(K^n_N, \ddd, K^n_2, K^n_1\rg). \qquad \qedhere
\]
\end{proof}

\section{Particle flux for Lipschitz paths}
\label{S:flux-intro}

In this section, we introduce particle flux and prove that it measures the amount of particles that cross a line in a typical sorting network. Let $\Lip$ be the set of Lipschitz paths from $[0,1] \to [-1, 1]$ (we will use this notation throughout the paper).
Define the \textbf{local speed} of a function $h \in \Lip$ at time $t$ by
$$
s(t) = \frac{d(\arcsin(h))}{dt} = \frac{h'(t)}{\sqrt{1 - h^2(t)}}.
$$
The local speed is not well-defined by the above formula at points when $h(t) = \pm 1$. In this case we use the convention that $s(t) = 0$.
The local speed $s(t)$ exists for almost every time $t$ for any $h \in \Lip$. Recalling the definition of $D_\mu$ from the previous section, we then define the {\bf particle flux} of $h$ over a set $A$ by
\begin{equation}
\label{E:path-flux}
J(h; A) = \frac{1}2 \int_A D_\mu(s(t)) \sqrt{1-h^2(t)}dt.
\end{equation}
We define $J(h) = J(h; [0, 1])$. Note that $J(h) < \infty$ for any Lipschitz function $h$. This follows from Lemma \ref{L:convex} (ii), which implies that 
$$
D_\mu(s(t)) \sqrt{1-h^2(t)} \le [\pi \vee |s(t)|] \sqrt{1-h^2(t)} \le \pi \vee |h'(t)|.
$$
We will also consider {\bf positive particle flux} $J^+(h; A)$ and {\bf negative particle flux} $J^-(h; A)$ defined by
$$
J^+(h; A) = \frac{1}2 \int_A D^+_\mu(s(t)) \sqrt{1-h^2(t)}dt, \qquad J^-(h; A) = \frac{1}2 \int_A D^-_\mu(s(t)) \sqrt{1-h^2(t)}dt.
$$ 
Again, we let $J^+(h) = J^+(h; [0, 1])$ and $J^-(h) = J^-(h; [0, 1])$.

\medskip

We now connect flux to random sorting networks.
If a random sorting network resembles the local limit in a local window around the global space-time position $(t, h(t))$, then by Theorem \ref{T:line-rate}, the number of distinct particles that cross $h$ in this window should be proportional to 
$$
D_\mu\lf(\frac{h'(t)}{\sqrt{1 -h^2(t)}}\rg) \sqrt{1-h^2(t)}.
$$
The scaling factors of $\sqrt{1-h^2}$ come from the semicircle rescaling of time in the local limit away from the center.

\medskip

Therefore in a typical large-$n$ sorting network, where most local windows resemble the local limit, $J(h)$ should be proportional to the number of particles that cross the line $h$, counting multiple crossings for a given particle if and only if the crossings happen at globally distinguishable locations. 

\medskip

The factor of $1/2$ is to account for the difference between the ${n \choose 2}^{-1}$ scaling in the global limit and the $n^{-1}$ scaling in the local limit. Similarly, $J^+(h)$ and $J^-(h)$ should be proportional to the number of upcrossings and downcrossings of the line $h$ in a large $n$ sorting network. 

\medskip

Now let $\Lip^*$ be the set of paths $h \in \Lip$ with $h(0) = -h(1)$. If $h \in \Lip_{r}$, then in any sorting network, every particle must cross $h$ at least once unless the particle starts at $h(0)$. Therefore $J(h)$ should be bounded below for such paths. Since any two particles in a sorting network cross each other exactly once, $J(h)$ should achieve this lower bound when $h$ is a trajectory limit. 

\medskip
 
The next theorem makes rigorous this intuition behind the definition of particle flux. To state the theorem, for a random variable $Y \in \scrD$ and a path $h \in \Lip$, we define
$$
P^+_Y(h; [a, b]) = \prob \Big( \exists s < t \in [a, b] \text{ such that } Y(s) < h(s) \mathand Y(t) > h(t) \Big).
$$
In other words, $P^+_Y(h; [a, b])$ is the probability that $Y$ upcrosses $h$ in the interval $[a, b]$. We similarly define the downcrossing probability
$$
P^-_Y(h; [a, b]) = \prob \Big( \exists s < t \in [a, b] \text{ such that } Y(s) > h(s) \mathand Y(t) < h(t) \Big).
$$
\begin{theorem} 
	\label{T:particle-crossings} Let $Y$ be any subsequential limit of $Y_n$.
	\begin{enumerate}[label=(\roman*)]
		\item Let $h \in \Lip$ and $[a, b] \sset [0, 1]$. Then $P^+_Y(h; [a, b]) \le J^+(h; \; [a, b])$ and $P^-_Y(h; [a, b]) \le J^-(h; [a, b])$.
		\item Let $h \in \Lip^*$. Then $J(h) \ge 1$.
		\item $\prob(J(Y) = 1) = 1$.
	\end{enumerate}
\end{theorem}

To prove this theorem, we first need two key lemmas about convergence to the local limit.
Recall that $C^+(V, L, t)$ is the number of upcrossings of a line $L(s) = cs + d$ in the interval $[0, t]$ in a swap function $V$. Recall also the definition of $U_n^{i}$ from Section \ref{S:local}.\begin{lemma}
	\label{L:L1-swap-rate}
	Let $\al \in (-1, 1)$, and suppose that $\{i_n\}_{n \in \nat}$ is a sequence of integers such that $i_n/n \to (1 + \al)/2$. Let $X$ be a uniform random variable on $[0, 1]$ that is independent of $U, U_n^{i_n}$, let $\{c_n\}_{n \in \nat}$ be a sequence of real numbers converging to $c \in \real$, and let $\{d_n\}_{n \in \nat}$ be a sequence of real numbers in $[0, 1]$. Define 
	$$
	L_n(s) = c_ns + d_n + X, \quad \text{and} \quad L(s) = cs + X.
	$$
	Then for any time $t \in (0, \infty)$, we have that 
	\begin{enumerate}[label=(\roman*)]
		\item $C^+(U_n^{i_n}, L_n, t) \cvgd C^+(U, L, t)$ as $n \to \infty$.
		\item $\expt C^+(U_n^{i_n}, L_n, t) \to \expt C^+(U, L, t)$ as $n \to \infty$.
		\item $\expt C^+(U_n^{i_n}, L_n, t) \le 3t + |c_n t| + 2$ for all $n$.
	\end{enumerate}
\end{lemma}

\begin{proof} 
First note that $d_n + X \Mod 1 \eqd X$ for all $n$. Therefore by the stationarity of $U$ with respect to integer-valued spatial shifts (Theorem \ref{T:local} (i)), we have that
\begin{equation}
\label{E:dn-shift}
C^+(U^{i_n}_n, L_n, t) \cvgd C^+(U, L, t) \quad \text{if and only if} \quad C^+(U^{i_n}_n, c_ns + X, t) \cvgd C^+(U, L, t).
\end{equation}
Now if $V_n \in \scrA$ is a sequence of swap functions converging to a swap function $V$ (in the swap function topology introduced in Section \ref{S:local}), then
$$
C^+(V_n, c_ns + X, t) \to C^+(V, L, t)
$$ 
unless $V$ either has a swap at time $t$, or either $X$ or $ct + X \in \Z$. By Theorem \ref{T:local} (iv), for any time $t$, the event where $U$ has a swap at time $t$ has probability $0$. Moreover, the probability that either $X$ or $ct + X \in \Z$ is also $0$. Therefore 
$$
C^+(U^{i_n}_n, c_ns + X, t) \cvgd C^+(U, L, t),
$$
and hence (i) follows by statement \eqref{E:dn-shift}. Now recall that $W(V, t)$ is the number of swaps at location $0$ in a swap function $V$ in the interval $[0, t]$. For any swap function $V$ and any line $H(s) = as+b$ with $b \in [0, 1)$, we have that
\begin{equation}
\label{E:C-W}
C^+(V, H, t) \le W(V, t) + |a t| + 2.
\end{equation}
To see why this is true, observe that every particle $x$ with $x \le 0$ and $V(x, t) > 1$ must move from position $0$ to position $1$ at least once in the interval $[0, t]$, therefore contributing to $W(V, t)$. Every particle $x$ that upcrosses $H$ in the interval $[0, t]$ fits this description, unless
$$
at \le V(x, t) \le 1.
$$
There are at most $|a t| + 2$ such values of $x$, proving \eqref{E:C-W}.

\medskip

 Now again since $U$ has no swap at time $t$ almost surely (Theorem \ref{T:local} (iv)), $W(U_n^{i_n}, t) \cvgd W(U, t)$. Also, by Theorem \ref{T:swaps}, we have that
$$
\expt W(U_n^{i_n}, t) \to \expt W(U, t).
$$
Hence, the random variables $W(U_n^{i_n}, t)$ are uniformly integrable (see Proposition 3.12, \cite{kallenberg2006foundations}). Therefore by \eqref{E:C-W} applied to the swap processes $U_n^{i_n}$ and the lines $L_n$, the random variables $C^+(U_n^{i_n}, L_n, t)$ are also uniformly integrable, and hence converge in expectation since they converge in distribution (again by Proposition 3.12, \cite{kallenberg2006foundations}).

\medskip

Finally, the bound on the probability distribution for the first swap location in a random sorting network (Theorem \ref{T:dist-k1}) and time stationarity (Theorem \ref{T:time-stat}) allows us to conclude that for any $n, i_n$ and $t$, that
$$
\expt W(U_n^{i_n}, t) \le 3t,
$$
which in turn proves the bound on the expectation of $C^+(U_n^{i_n}, L_n, t)$ via \eqref{E:C-W}.
\end{proof}

This next lemma is similar to Lemma \ref{L:L1-swap-rate}, but deals with particle speeds rather than upcrossing rates.
For a swap function $V$ we define
$$
S(V, t) = \frac{V(0, t)}{t}.
$$
This is the average speed of particle $0$ in the interval $[0, t]$. 

\begin{lemma}
\label{L:L1-flux}
Let $\{U_n^{i} : i \in \{1, \dots n\}, n \in \nat\}$ be the array of locally scaled random sorting networks defined in Section \ref{S:local}, and let $U$ be the local limit. Let $I$ be a uniform random variable on $(-1, 1)$, independent of everything else.
For each $n$, define $I_n = \ceil{n(I + 1)/2}$, let 
$$
R_n = \sqrt{1 - \lf[\frac{2I_n}n - 1\rg]^2}, \qquad \mathand \qquad R = \sqrt{1 - I^2}.
$$
Then the following statements hold.
\begin{enumerate}[label=(\roman*)]
\item For any $t \in (0, \infty)$, we have that $R_nD_\mu(S(U_n^{I_n}, R_nt)) \xrightarrow[n \to \infty]{d} RD_\mu(S(U, R t))$.
\item For any $t \in (0, \infty)$, we have that $\expt R_n D_\mu(S(U_n^{I_n}, R_nt)) \xrightarrow[n \to \infty]{} \expt R D_\mu \lf(S(U, R t) \rg)$.
\item $\expt R D_\mu \lf(S(U, R t) \rg) \to 2$ as $t \to \infty$.
\end{enumerate}
\end{lemma}

Note that the processes $\{U^n_n : n \in \mathbb{N}\}$ were not defined in Section \ref{S:local}. These processes plays no real role in the above lemma as $P(I_n = n) \to 0$ as $n \to \infty$, so we simply set $U^n_n = U^{n-1}_n$.

\begin{proof}
Fix $t \in (0, \infty)$. If $V_n$ is a sequence of swap functions converging to a swap function $V$ with no swap at time $t$, then $D_\mu(S(V_n, t)) \to D_\mu(S(V, t))$. Now condition on $I$. This fixes $R_n, I_n,$ and $R$. For any fixed time $t \in (0, \infty)$, the process $U$ has no swap at time $t$ almost surely (Theorem \ref{T:local} (iv)). Therefore for any bounded continuous test function $f$, we have that
$$
\expt \bigg[f\Big(R_nD_\mu(S(U_n^{I_n}, R_nt))\Big) \; \Big | \; I \bigg] \to \expt \bigg[f\Big(RD_\mu(S(U,Rt))\Big) \; \Big| \; I \bigg].
$$
Taking expectations proves the distributional convergence in (i).
Now, by Lemma \ref{L:convex}, we have that $D_\mu(s) \le |s| + \pi$ for any $s \in \real$. Recalling that $Q(V, t)$ is the number of swaps made by particle $0$ in the interval $[0, t]$ in a swap process $V$, this implies that
\begin{equation}
\label{E:Q-bd}
R_n\lf(D_\mu(S_{I_n}(U_n^{I_n}, R_nt))\rg)\le R_n(|S_{I_n}(U_n^{I_n}, R_nt)| + \pi) \le \frac{Q(U_n^{I_n}, R_nt)}{t} + \pi.
\end{equation}

Now, we similarly have that $Q(U_n^{I_n}, R_nt) \cvgd Q(U, Rt)$, again since $U$ has no swap at time $t$ almost surely. Moreover, 
$$
\expt Q(U_n^{I_n}, R_nt) = \frac{2\floor{nt}}n,
$$
since this expectation is simply the expected number of swaps made by a uniformly random particle in a sorting network after $2\floor{nt}$ steps. By Lemma \ref{L:swaps-2}, we also have that
$$
\expt Q(U, Rt) = \frac{4t}{\pi} \int_{-1}^1 \sqrt{1 - x^2}dx = 2t,
$$
so $\expt Q(U_n^{I_n}, R_nt) \to \expt Q(U, Rt)$. Again, by Proposition 3.12 from \cite{kallenberg2006foundations}, this implies that the random variables $\{Q(U_n^{I_n}, R_nt), n \in \mathbb N\}$ are uniformly integrable, and hence so are the random variables $\{R_n D_\mu(S(U_n^{I_n}, R_nt)), n \in \mathbb N\}$. Since these random variables converge in distribution to $R D_\mu(S(U, R t))$, they must also converge in expectation, proving (ii).

\medskip

Now we prove (iii). First, Theorem \ref{T:local-speeds} and the Lipschitz property of $D_\mu$ imply that 
\begin{equation}
\label{E:D-mu-lim}
R D_\mu(S(U, Rt)) \xrightarrow[t \to \infty]{} R D_\mu(S(0)) \quad \as,
\end{equation}
where $S(0)$ is the speed of particle $0$ in $U$. Analogously to \eqref{E:Q-bd}, we also have that
$$
R D_\mu(S(U, R t)) \le \frac{Q(U, t)} {t} + \pi.
$$
By Theorem \ref{T:particle-swaps}, the right hand side above is uniformly integrable over $t \in [1, \infty)$, and hence so is the left hand side. Therefore since $RD_\mu(S(U, R t))$ has an almost sure limit by \eqref{E:D-mu-lim}, it also converges in expectation. Finally, Theorem \ref{T:particle-swaps} implies that
\begin{equation*}
\expt R D_\mu(S(0))=\expt R  \expt D_\mu(S(0)) = 2. \qedhere
\end{equation*}
\end{proof}

To prove Theorem \ref{T:particle-crossings}, we also need two auxiliary results. The first is a simple lemma about uniform convergence of functions. This will be used in the proof of Theorem \ref{T:particle-crossings} (i).

\begin{lemma}
\label{L:mesh}
Let $f_n:[0, 1] \to [-1, 1]$ be a sequence of functions converging uniformly to a continuous function $f$, and let $h:[0, 1] \to [-1, 1]$ be any continuous function. Let 
$$
\{\Pi_n = \{t_{n, 0} =0< t_{n, 1} < \dots < t_{n, m(n)}=1\}\}_{n \in \nat}
$$
be a sequence of partitions of $[0, 1]$ such that
$$
\text{mesh}(\Pi_n) = \max_{i \in \{0, 1, \dots, m(n) - 1\}} |t_{n, i + 1} - t_{n, i}| \to 0 \qquad \mathas n \to \infty.
$$
Let $[a, b] \sset [0, 1]$. If there exist times $s < t \in [a, b]$ such that $f(s) < h(s)$ and $f(t) > h(t)$, then for all large enough $n$, there exists a time $t_{n, i} \in [a, b]$ such that $f_n(t_{n, i}) \le h(t_{n, i})$ and $f_n(t_{n, i + 1}) > h(t_{n, i + 1})$.
\end{lemma}

\begin{proof}
By the continuity of $f$ and $h$, there exists an $\ep > 0$ and disjoint intervals $[s, s_+]$ and $[t_-, t]$ such that $f(r) < h(r) - \ep$ for all $r \in[s, s_+]$ and $f(r) > h(r) + \ep$ for all $r \in [t_-, t]$. Therefore by uniform convergence, for all large enough $n$ we have that $f_n(r) < h(r) - \ep/2$ for all $r \in[s, s_+]$ and $f(r) > h(r) + \ep/2$ for all $r \in [t_-, t]$. Now, for large enough $n$ we also have that 
$$
\text{mesh}(\Pi_n) < \min(s_+ - s, t - t_-).
$$
Therefore for such $n$, there must exists $i_1 < i_2 \in \Pi_n \cap [a, b]$ such that  $f_n(t_{n, i_1}) < h(t_{n, i_1})$ and $f_n(t_{n, i_2}) > h(t_{n, i_2})$. Thus for some $i \in \{i_1, \dots i_2 - 1\}$, we must have that $f_n(t_{n, i}) \le h(t_{n, i})$ and $f_n(t_{n, i + 1}) > h(t_{n, i + 1})$.
\end{proof}

To prove part (iii), we also need that $J(\cdot)$ is lower semicontinuous.

\begin{prop}
\label{P:lsc-E}
Let $h_n \in \Lip$ be a sequence converging uniformly to $h$. Then for any set $A \sset [0, 1]$, we have that
$$
J(h; A) \le \liminf_{n \to \infty} J(h_n ; A).
$$
\end{prop}

\begin{proof}
Note that $J(h; A)$ is of the form $\int_A G(h(t), h'(t))dt$, where $G$ is a continuous, positive function, such that for any fixed value of $a$, $G(a, \cdot)$ is convex. This follows from the convexity of $D_\mu$ (Lemma \ref{L:convex}). Functionals of this form are lower semicontinuous in the uniform norm by Theorem 1.6 from \cite{struwe1996variational}.
\end{proof}

\begin{proof}[Proof of Theorem \ref{T:particle-crossings}.]
\noindent \textbf{Proof of (i):} \qquad Note first that by the symmetry of $Y$ (Proposition \ref{P:Y-symmetries}), that
$$
\prob^-_Y(h; [a, b]) = \prob^+_{-Y}(-h; [a, b]) = \prob^+_{Y}(-h; [a, b])
$$
for any path $h \in \Lip$ and any interval $[a, b]$. Also, $J^+(-h) = J^-(h)$ by the symmetry of $\mu$ (Theorem \ref{T:local-speeds}). Therefore the assertion that
$
\prob^-_Y(h; [a, b]) \le J^-(h)
$
is equivalent to the assertion that $\prob^+_Y(-h; [a, b]) \le J^+(-h)$, and so to prove Theorem \ref{T:particle-crossings} (i), it suffices to prove that $\prob^+_Y(h; [a, b]) \le J^+(h)$ for every path $h \in \Lip$.

\medskip
We first prove this for $h \in \Lip$ with range in the open interval $(-1, 1)$, since there is a technical difficulty around the definition of local speed when $|h| = 1$. We first give a somewhat informal explanation of the proof in this case.
Let $t \in \Z \cap (0, \infty)$, and for $s \in [0, 1]$, define 
$$
s_{n, t} = \frac{2t}{n-1} \floor{\frac{(n-1)s}{2t}}, \quad \mathand \quad s^+_{n, t} = \min\lf(s_{n, t} + \frac{2t}{n-1}, 1\rg).
$$
Note that we will also use this notation with $s = a$ and $s = b$.
Let $X$ be a uniform random variable on $[0, 1]$, independent of the sequence $\{Y_n\}$, and let
$$
A^+_{n, t, s} = \lf\{ Y_n(s_{n, t}) < h(s_{n, t}) + \frac{2X}n \;\; \mathand \;\; Y_n(s^+_{n, t}) \ge h(s^+_{n, t}) + \frac{2X}n \rg\}.
$$
For $t$ fixed and $n$ large, the event where $Y_n$ upcrosses $h$ in the interval $[a, b]$ is approximately equal to the union $\bigcup_{s \in [a, b]} A^+_{n, t, s}$, in a sense that can be made rigorous using Lemma \ref{L:mesh}. Therefore by taking $n \to \infty$, we can relate $\prob_Y^+(h; [a, b])$ to $\lim_{n \to \infty} \prob A^+_{n, t, s}$. This limiting probability is connected to counting upcrossings in the local limit. Taking $t \to \infty$ will relate $\lim_{n \to \infty} \prob A^+_{n, t, s}$ to particle flux. The technical tool needed to take both the limit in $n$ and in $t$ is Lemma \ref{L:L1-swap-rate}.

\medskip

The details are as follows. We first find a better way to calculate $\prob A^+_{n, t, s}$. Observe that when $s < b_{n, t}$, then $s_{n, t}^+ \le b_{n, t}$, and so $s^+_{n, t} =  s_{n, t} + 2t/(n-1)$. Therefore for such $s$, time stationarity of random sorting networks (Theorem \ref{T:time-stat}) implies that $A^+_{n, t, s}$ has the same probability as the event
$$
\lf\{ Y_n(0) < h(s_{n, t}) + \frac{2X}n \;\; \mathand \;\; Y_n\lf(\frac{2t}{n-1}\rg) \ge h(s^+_{n, t}) + \frac{2X}n \rg\}.
 $$
 Here have used that $t \in \Z$ to apply time stationarity. 
We can then express the upcrossing probability $\prob A^+_{n, t, s}$ in terms of the expected number of upcrossings in the local scaling. For $u \in [-1, 1]$, define
$$
[u]_n = \floor{\frac{n(u + 1)}2}, \quad \;\; \{u\}_n = \frac{n(u + 1)}2 - \floor{\frac{n(u + 1)}2}, \quad \;\; g_n(u) = \sqrt{1 - \lf(\frac{2[u]_n}n - 1\rg)^2}.
$$
Then for $s < b_{n, t}$, we have that
$$
n \prob(A^+_{n, t, s}) = \expt C^+\lf(U_n^{[h(s_{n, t})]_n}, L_{n, s}, g_n(h(s_{n, t})) t \rg),
$$
where
$$
L_{n, s}(r) = \frac{n\lf(h(s^+_{n, t}) - h(s_{n, t})\rg)}{2tg_n(h(s_{n, t}))} r +  \{h(s_{n, t})\}_n +  X.
$$
Now, let 
$$
L_{s}(r) = \frac{h'(s)}{\sqrt{1 -h^2(s)}}r + X.
$$
Since $h$ is Lipschitz and hence differentiable almost everywhere, Lemma \ref{L:L1-swap-rate} (ii) implies that for almost every $s \in [0, b)$, that
\begin{equation}
\label{E:n-PA}
\lim_{n \to \infty} n\prob(A^+_{n, t, s}) = \lim_{n \to \infty} \expt C^+\lf(U_n^{[h(s_{n, t})]_n}, L_{n, s}, g_n(h(s_{n, t})) t \rg) =\expt C^+(U, L_s, \sqrt{1 - h^2(s)}t).
\end{equation}
Here we have used that the range of $h$ is in $(-1, 1)$ to ensure convergence. Moreover, there exists a constant $c$ such that 
\begin{equation}
\label{E:finite-bd}
n\prob(A^+_{n, t, s}) \le ct
\end{equation}
for all $n \in \nat, s < b_{n, t}$.  This follows from Lemma \ref{L:L1-swap-rate} (iii) and the fact that $h$ is Lipschitz. Now let $Z_{n, t}$ be the number of times $r$ of the form $s_{n, t}$ in the interval $[a_{n, t}, b_{n, t})$ such that the upcrossing event $A^+_{n, t, r}$ occurs. We have that
$$
Z_{n, t} = \sum_{r \in [a_{n, t}, b_{n, t}), r = r_{n, t}} \indic (A^+_{n, t, r}) = \int_{a_{n, t}}^{b_{n, t}} \frac{(n-1)\indic (A^+_{n, t, s})}{2t} ds.
$$
Therefore by the bounded convergence theorem, we have that
\begin{equation}
\label{E:cvg-ab-flux}
\expt Z_{n, t} = \int_{a_{n, t}}^{b_{n, t}} \frac{(n-1)\prob (A^+_{n, t, s})}{2t} ds \xrightarrow[n \to \infty]{} \frac{1}{2} \int_{a}^b \frac{\expt C^+(U, L_s, \sqrt{1 - h^2(s)}t)}t ds.
\end{equation}
Next, let $Y$ be any distributional subsequential limit of $Y_n$, and consider a subsequence $n_i$ and a coupling of $Y_{n_i}, X, Y$ where $Y_{n_i} - 2 X/n_i \to Y$ almost surely. On the event where this almost sure convergence holds, Lemma \ref{L:mesh}
implies that
$$
\{\exists s < t \in [a, b] \text{ such that } Y(s) < h(s) \mathand Y(t) > h(t)\} \sset \limsup_{i \to \infty} \{Z_{n_i, t} \ge 1\}.
$$
 Therefore
$$
P^+_Y(h; [a, b]) \le \liminf_{n \to \infty} \prob(Z_{n_i, t} \ge 1) \le \lim_{n \to \infty} \expt Z_{n, t}.
$$
The integrand on the right hand side of \eqref{E:cvg-ab-flux} is bounded above by $c$ for all $t \in \Z \cap (0, \infty)$ by \eqref{E:n-PA} and \eqref{E:finite-bd}. Therefore by Theorem \ref{T:line-rate} and the bounded convergence theorem, the right hand side of \eqref{E:cvg-ab-flux} converges to 
$$
\frac{1}{2} \int_{a}^b D^+_\mu\lf(\frac{h'(s)}{\sqrt{1 - h^2(s)}}\rg) \sqrt{1 - h^2(s)} ds = J^+(h; \; [a, b]) \qquad \mathas \;\; t \to \infty.
$$
This proves Theorem \ref{T:particle-crossings} (i) for $h \in \Lip$ with range in $(-1, 1)$. Now we extend this to all $h \in \Lip$. Define $h_m = h \vee(-1 + 1/m) \wedge (1 - 1/m)$, and let 
 $$
 A_m = \{t \in [0, 1] : h(t) \ge 1- 1/m \text{ or } h(t) \le -1 + 1/m\}.
 $$
 Letting $\scrL$ be Lebesgue measure on $[0, 1]$, we have that
\begin{align*}
J^+(h_m ; A_m \cap [a, b]) &= D^+_\mu(0)\sqrt{1 - (1- 1/m)^2}\scrL(A_m \cap [a, b]), \qquad \mathand \qquad \\
J^+(h_m ; A_m^c \cap [a, b]) &= J(h ; A_m^c \cap [a, b]).
\end{align*}
The flux $J^+(h_m ; A_m) \to 0$ as $m \to \infty$, so
$$
\limsup_{m \to \infty} J^+(h_m; [a, b]) \le J^+(h; [a, b]).
$$
Moreover, $h_m$ converges uniformly to $h$, so if $Y$ upcrosses $h$, then $Y$ will eventually upcross $h_m$. Therefore
$$
\liminf_{m \to \infty} P^+_Y(h_m; [a, b]) \ge P^+_Y(h; [a, b]).
$$
Putting these two inequalities together completes the proof.

\medskip

\noindent \textbf{Proof of (ii):} \qquad 
Let $h \in \Lip^*$, and let $Y$ be a subsequential limit of $Y_n$. Therefore by Theorem \ref{T:particle-crossings} (i),
\begin{align*}
\prob \Big(\exists r < t \in [0, 1] \text{ such that either } Y(r) < h(r), \;Y(t) > h(t),\text{ or } &Y(r) > h(r), \;Y(t) < h(t) \Big) \\
&\le J^+(h) + J^-(h) = J(h).
\end{align*}
The event above holds unless $Y(0) = h(0)$ since $Y(0) = - Y(1)$ almost surely by Theorem \ref{T:bounded-speed}. Since $Y(0)$ is uniformly distributed (Theorem \ref{T:bounded-speed}), $Y(0) \ne h(0)$ with probability one, so the left hand side above is equal to $1$.
\medskip

\noindent \textbf{Proof of (iii):} \qquad
Let $Y$ be any subsequential limit of $Y_n$. Fix $t \in (0, \infty) \cap \Z$, and define 
$$
t_{n, j} = \frac{2jt}{n-1} \quad \text{for  } j \in \{0, \ddd, \floor{(n-1)/(2t)}\} \quad \text{ and let } t_{n, \floor{\frac{n-1}{2t}} + 1} = 1.
$$
Define $Y_{n, t}$ so that 
$
Y_{n, t}\lf(t_{n, j}\rg) = Y_{n}\lf(t_{n, j}\rg)$ for $j \in \{0, \ddd, \floor{(n-1)/(2t)} + 1\}$, and so that $Y_{n, t}$ is linear at times in between. By time stationarity of random sorting networks (Theorem \ref{T:time-stat}), we can write
\begin{equation}
\label{E:JY-ineq}
\expt J(Y_{n, t}) = \sum_{j=0}^\floor{\frac{n-1}{2t}} \expt J(Y_{n, t} ;  \; [t_{n, j}, t_{n, j+1}] ) \le \lf(\floor{\frac{n-1}{2t}} + 1\rg) \expt J(Y_{n, t} ;  \; [0, t_{n, 1}]). 
\end{equation}
Here we have used that $t \in \Z$ to apply time stationarity of sorting networks. The final term in the sum above may be slightly smaller than the previous terms since the length of interval may be less than $2t/(n-1)$; this gives rise to the inequality.

\medskip

Now we have that
\begin{align}
\nonumber
J(Y_{n, t} ;  \; [0, t_{n, 1}]) &= \frac{1}{2} \int_0^{t_{n, 1}} \int \lf|Y_{n, t}'(0)- \sqrt{1 - Y_{n, t}^2(r)}x\rg|d\mu(x) dr \\
\label{E:frac-en}
&=  \frac{1}{2} \int_0^{t_{n, 1}} \int \lf|Y_{n, t}'(0)- \sqrt{1 - Y_{n, t}^2(0)}x\rg|d\mu(x) dr + \ep_n
\end{align}
for some error term $\ep_n$. In the first equality we have used that $Y'_{n, t}(r) = Y'_{n, t}(0)$ for all $r \in [0, t_{n, 1}]$ by piecewise linearity of $Y'_{n, t}$. Now if $|Y_{n, t}'(0)| \ge \pi$, then since $\mu$ is symmetric and supported in $[-\pi, \pi]$ (Theorem \ref{T:local-speeds}), the two inner integrals are the same. Therefore $\ep_n = 0$ in this case. Also, when 
$
|Y_{n, t}'(0)| < \pi 
$
and $x \in [-\pi, \pi]$, the difference of the integrands is bounded above by
$$
\lf|\sqrt{1 - Y_{n, t}^2(0)}x - \sqrt{1 - Y_{n, t}^2(r)}x\rg| \le |x|\sqrt{1 - \lf(1 - \frac{2\pi t}{n-1}\rg)^2} \le \pi\sqrt{\frac{4\pi t}{n-1}}.
$$
Hence $|\ep_n| \le k(t/n)^{3/2}$ for some constant $k$. Now, letting $I_n = n(Y_{n, t}(0) + 1)/2$, and using the notation of Lemma \ref{L:L1-flux}, \eqref{E:frac-en} is equal to
$$
\frac{t}{n-1}R_nD_\mu\lf(S(U_n^{I_n}, R_n t)\rg) + \ep_n.
$$
Therefore by Lemma \ref{L:L1-flux} (ii) and the bound on $\ep_n$, \eqref{E:JY-ineq} implies that
\begin{align*}
\expt J(Y_{n, t}) \le \lf(\floor{\frac{n - 1}{2t}} + 1 \rg)\lf[\frac{t}{n-1} \expt  R_nD_\mu\lf(S(U_n^{I_n}, R_n t)\rg)  + \expt \ep_n \rg]\xrightarrow[n \to \infty]{} \frac{1}2 \expt R D_\mu \lf(S(U, Rt) \rg).
\end{align*}
Letting $t \to \infty$, Lemma \ref{L:L1-flux} (iii) then implies that
\begin{equation}
\label{E:flux-lim}
\lim_{t \to \infty} \lim_{n \to \infty} \expt J(Y_{n, t}) \le 1.
\end{equation}
Since subsequential limits of $Y_n$ are Lipschitz by Theorem \ref{T:bounded-speed}, $Y$ is a subsequential limit of $Y_n$ if and only if $Y$ is a subsequential limit of $Y_{n, t}$. Therefore, by Fatou's lemma and the lower semicontinuity of $J(\cdot)$ (Proposition \ref{P:lsc-E}), 
\begin{equation*}
\lim_{n \to \infty} \expt J(Y_{n, t}) \ge \expt J(Y)
\end{equation*}
for all $t$, so $\expt J(Y) \le 1$ by \eqref{E:flux-lim}. Moreover, $J(Y) \ge 1$ almost surely by Theorem \ref{T:particle-crossings} (ii) and $Y(0) = -Y(1)$ almost surely by Theorem \ref{T:bounded-speed} (ii), so $J(Y) = 1$ almost surely.
\end{proof}

\section{Characterization of minimal flux paths}
\label{S:unique}

In this section, we show that any subsequential limit $Y$ of $Y_n$ is uniquely determined by its initial position, maximum absolute value, and whether it is initially increasing or decreasing. By Theorem \ref{T:bounded-speed} (ii) and Theorem \ref{T:particle-crossings} (iii), it is enough to characterize paths $h \in \Lip^*$ with minimal flux $J(h) = 1$. 

\medskip

Let $\Lip_0$ be the set of path $h \in \Lip$ with $h(0) = h(1) = 0$. We can first recognize that to characterize minimal flux paths $h \in \Lip^*$, it is enough to characterize minimal flux paths $h \in \Lip_0$. This is a simple consequence of the following simple fact.

\begin{lemma}
\label{L:path-shift}
Let $h \in \Lip^*$ and $t_0 = \inf \{t : h(t) = 0\}$. Define the path
\begin{equation*}
\tilde h(t) = \begin{cases}
&h(t + t_0), \qquad \quad \quad t \le 1- t_0 \\
&-h(t + t_0 - 1), \qquad t > 1 - t_0.
\end{cases}
\end{equation*}
Then $J(h) = J(\tilde h)$. In particular, every path $h \in \Lip^*$ with flux $J(h) = 1$ can be shifted to a path $\tilde h \in \Lip_0$ with $J(\tilde h) = 1$.
\end{lemma}

We build up to the following characterization of minimal flux paths in $h \in \Lip_0$. Define the maximum height $m(h)$ of a continuous path $h:[0, 1] \to [-1, 1]$ by
$$
m(h) = \max \{ |h(t)| : t \in [0, 1] \}.
$$ 
Recall also the definition of $g(y)$-Lipschitz from Section \ref{S:local}.
\begin{theorem}
\label{T:path-set}
For every $m \in [0, 1]$, there exists exactly one $\pi\sqrt{1- y^2}$-Lipschitz path $h_m \in \Lip_0$ such that $J(h_m) = 1$, $m(h_m) = m$, and $h_m \ge 0$. The paths $h_m$ satisfy the following properties.
\smallskip
\begin{enumerate}[nosep,label=(\roman*)]
  \item $h_m(1/2) = m$, and $h_m(t) = h_m(1 - t)$ for $t \in [0, 1/2]$.
  \item $h_m$ is strictly increasing on the interval $[0, 1/2]$ for $m > 0$.
  \item For any $\ell \in [0, m]$, we have that $h_{\ell}(s) \le h_{m}(s)$ for all $s \in [0, 1]$.
\end{enumerate}
\smallskip
Now define $h_{-m} = -h_m$. If $h \in \Lip_0$ is a $\pi\sqrt{1 -h^2}$-Lipschitz path with $J_m(h) = 1$ and $m(h) = m$, then either $h = h_m$ or $h = h_{-m}$.
  \end{theorem}
  
Theorem \ref{T:path-set} will be proven as Proposition \ref{P:monotone-traj}, Proposition \ref{P:unique-path}, Corollary \ref{C:sym}, Lemma \ref{L:no-cross}, and Proposition \ref{P:exist-m}. When thinking about the proofs in this section, it may be useful to keep in mind that our end goal is to show that $h_m(t) = m \sin(\pi t)$ for all $m \in [-1, 1]$.

\medskip

We also state an analogue of Theorem \ref{T:path-set} for minimal flux paths $h \in \Lip^*$. First, for a path $h \in \Lip$, define
$$
S(h) = \sup_{t \in [0, 1]} h(t) \qquad \mathand \qquad I(h) = \inf_{t \in [0, 1]} h(t).
$$
\begin{theorem}
\label{T:unique-2}
Fix $a \in [-1, 1]$ and $k \in [|a|, 1]$. Then the following hold:

\begin{enumerate}[nosep,label=(\roman*)]
  \item There is exactly one $\pi\sqrt{1 - y^2}$-Lipschitz path $h_{a, k} \in \Lip^*$ such that $h_{a, k}(0) = - h_{a, k}(1) = a$ and $m(h) = S(h) = k$.
 \item There is a exactly one $\pi\sqrt{1 - y^2}$-Lipschitz path $h_{a, -k} \in \Lip^*$ such that $h_{a, -k}(0) = - h_{a, -k}(1) = a$ and $-m(h) = I(h) = - k$. We have that $h_{a, k}(t) = -h_{a, -k}(1 - t)$ for all $t$.
  \item If $a \neq 0$, there is a unique time $t \in (0, 1)$ such that $h_{a, k}(t) = 0$. Moreover, the path
$$
g(s) = 
\begin{cases}
&h(t + s), \qquad s \le 1 - t, \\
&-h(t + s - 1), \qquad s > 1 - t 
\end{cases}
$$
is equal to $h_k$ if $a < 0$ and $h_{-k}$ if $a > 0$.
  \item For $k_1 \le k_2 \in [-1, -|a|] \cup [|a|, 1]$, we have that $h_{a, k_1}(t) \le h_{a, k_2}(t)$ for all $t \in [0, 1]$. 
  \item $h_{a, a} = h_{a, -a}$.
\end{enumerate}
\end{theorem}

All parts of this theorem follow from applying Lemma \ref{L:path-shift} to Theorem \ref{T:path-set}, except part (iv). This will be proven in Lemma \ref{L:no-cross}.

\subsection{Basic bounds on $D_\mu$}
\label{SS:basic-bounds}

In order to work with the functional $J(\cdot)$, we first prove a few basic bounds on $D_\mu$.

\begin{lemma}
\label{L:max-heights}
Let $Y$ be a subsequential limit of $Y_n$. For all $a \in [0, 1]$, we have that
$$
\prob(m(Y) > a) \leq \frac{D_\mu(0)\sqrt{1- a^2}}{2}.
$$
\end{lemma}

\begin{proof}
Since $Y(0) = -Y(1)$ and $Y(0)$ is uniformly distributed (Theorem \ref{T:bounded-speed}), the left hand side of this inequality can be written as
\begin{align*}
&\prob\bigg(Y(0) < a \mathand Y(t) > a \text{ for some } t > 0, 
\mathor Y(0) > - a \mathand Y(t) < - a \text{ for some } t > 0 \bigg).
\end{align*}
By Theorem \ref{T:particle-crossings}, this is bounded above by $J^+(g_a) + J^-(g_{-a})$, where $g_a$ is the path of constant height $a$. Finally, 
\begin{equation*}
J^+(g_a) + J^-(g_{-a}) = \frac{D_\mu(0)\sqrt{1- a^2}}2. \qedhere
\end{equation*}
\end{proof}
\begin{lemma}
\label{L:expect-2}
$D_\mu(0) = 2.$ That is, if $X$ is a random variable with distribution $\mu$, then $\expt|X| = 2$.
\end{lemma}

\begin{proof}  Let $Y$ be a subsequential limit of $Y_n$, and let $\al = D_\mu(0)$.  By Lemma \ref{L:max-heights}, we have that
$$
1 = \prob(m(Y) > 0) \leq \frac{\al}2,
$$
so $\al \geq 2$. The equality above follows since $m(Y) \ge |Y(0)|$ and $Y(0)$ is uniformly distributed on $[-1, 1]$ (Theorem \ref{T:bounded-speed}). Now, for every height $a$ such that 
$$
a > \sqrt{1 - \lf(\frac{2}{\al}\rg)^2},
$$
Lemma \ref{L:max-heights} guarantees that $m(Y) \le a$ with positive probability. Therefore by Theorem \ref{T:particle-crossings} (iii), for any $\ep > 0$ there is a $\pi\sqrt{1-y^2}$-Lipschitz path $h_\ep \in \Lip^*$ with $J(h_\ep) = 1$ such that
\begin{equation}
\label{E:m-hep}
m(h_\ep) \leq m_\ep := \sqrt{1 - \lf(\frac{2}{\al} - \ep\rg)^2}.
\end{equation}
Using that $D_\mu$ is minimized at $0$ (Lemma \ref{L:convex}), we have the bound
\begin{equation}
\label{E:flux-he}
\begin{split}
1 = J(h_\ep) &\ge \frac{\al}2 \int_0^1\sqrt{1-h_\ep^2(t)}dt.
\end{split}
\end{equation}
The above integrand is always bounded below by $2/\al - \ep$. Also, since $h_\ep(0) = -h_\ep(1)$ and $h_\ep$ is $\pi$-Lipschitz, the amount of time that $h_\ep$ spends in the interval $[-m_\ep/2, m_\ep/2]$ is at least $m_\ep/(2\pi)$. Therefore
$$
\frac{2}{\al} \ge \int_0^1\sqrt{1-h_\ep^2(t)}dt \ge \lf(\frac{2}\al - \ep\rg)\lf(1 - \frac{m_\ep}{2\pi}\rg) + \frac{m_\ep}{2\pi}\sqrt{1 - \frac{m_\ep^2}{4}}.
$$
Letting $\ep \to 0$, we get an inequality in $\al$ which implies that $\al = 2$.
\end{proof}

Combining Lemmas \ref{L:max-heights} and \ref{L:expect-2} shows that $\prob(m(Y) > a) \le \sqrt{1 - a^2}$. This is optimal, as can be seen by a quick calculation involving the Archimedean path. Later on in the paper we will prove the opposite inequality, which is much more involved.

\begin{lemma}
\label{L:crude-deviation-bd} There is a sequence $u_n \to 0$ such that 
$$
D_\mu(u_n) \leq \frac{2}{\cos(u_n/2)}.
$$
\end{lemma}

\begin{proof}
For any subsequential limit $Y$ of $Y_n$ and any $\ep > 0$, Lemmas \ref{L:max-heights} and \ref{L:expect-2} imply that
$$
\prob(m(Y) \le \ep) \geq 1 - \sqrt{1-\ep^2} > 0.
$$
Also, $Y(0)$ is uniformly distributed by Theorem \ref{T:bounded-speed}, so $\prob(m(Y) = 0) = 0$. Therefore by Theorem \ref{T:particle-crossings} (iii), we can find a sequence of positive numbers $\al_n \to 0$ and a sequence of paths $h_n \in \Lip^*$ with  $m(h_n) = \al_n$ and $J(h_n) = 1$. 
\medskip

Let $s_n$ be the local speed of $h_n$. The total variation of each of the paths $h_n$ is at least $2\al_n$, and hence the average absolute local speed of each $h_n$ is at least $2 \arcsin(\al_n)$.  The convexity and symmetry of $D_\mu$ (Lemma \ref{L:convex}) then gives the following bound.
\begin{align*}
J(h_n) &\geq \sqrt{1- \al_n^2}\int_0^1  D_\mu(s_n(t)) dt \\
&\geq \sqrt {1-\al_n^2} D_\mu(2 \arcsin(\al_n)).
\end{align*}
Letting $u_n = 2\arcsin(\al_n)$ and rearranging completes the proof.
\end{proof}

\subsection{Monotonicity and uniqueness for minimal flux paths}
\label{SS:monotone}
In this subsection, we prove that minimal flux paths $h \in \Lip_0$ with a particular maximum height $m(h)$ are unique up to sign, and that they satisfy a particular monotonicity relation.

\medskip

We start with two simple lemmas. The first lemma shows that minimal flux paths minimize flux on every subinterval of $[0, 1]$.
  
\begin{lemma}
\label{L:min-path}
Let $h \in \Lip^*$ be a path with $J(h) = 1$. For any interval $[a, b] \sset [0, 1]$ and any path $g \in \Lip$ with $g(a) = h(a)$ and $g(b) = h(b)$, we have that
$$
J(h; [a, b]) \le J(g; [a, b]).
$$
Moreover, if $f \in \Lip^*$ is another path with $J(f) = 1$, $f(a) = h(a)$, and $f(b) = h(b)$, then we can form a new path $p \in \Lip^*$ with $J(p) = 1$ by letting
\begin{equation*}
p(t) = \begin{cases}
&f(t) \qquad t \leq t_1, t \ge t_2 \\
&h(t) \qquad t \in [t_1, t_2].
\end{cases}
\end{equation*}
\end{lemma}

\begin{proof}
If there is some $g$ with $J(h; [a, b]) > J(g; [a, b])$, then we can make a new path $p \in \Lip^*$ which is equal to $g$ on $[a, b]$ and equal to $h$ on $[a, b]^c$. This path $p$ will have $J(p) < 1$, contradicting Theorem \ref{T:particle-crossings} (ii). The second part of the lemma is a consequence of the fact that 
\[
J(p) = J(p ; [a, b]) + J(p ; [a, b]^c). \qquad \qquad \qedhere
\]
\end{proof}

The next lemma uses the bounds established in Section \ref{SS:basic-bounds} to eliminate plateaus in minimal flux paths. 

\begin{lemma}
\label{L:no-plat}
For any height $\al \in (0, 1)$ and any interval $[t_1, t_2] \sset [0, 1]$, we have that 
$$
\inf_{h \in \Lip} \big\{J(h; [t_1, t_2]) : h(t_1) = \al, h(t_2) = \al \big\} < (t_2-t_1)\sqrt{1-\al^2}.
$$
\end{lemma}

\begin{proof}
Without loss of generality, we may assume that  $t_1 = 0$. Let $\{u_n\}_{n \in \nat}$ be as in Lemma \ref{L:crude-deviation-bd}, and define a sequence of paths $h_n \in \Lip$ by letting

\begin{equation*}
h_n(t) = \begin{cases} 
\;\sin (\arcsin(\al) + u_n t), \qquad &t \leq t_2/2 \\
\;h_n(t_2 - t), \qquad \qquad \quad \;\;\;\;\; &t \in [t_2/2, t_2] \\
\;\al, \qquad \qquad \qquad \quad \quad \; &t > t_2.
\end{cases} 
\end{equation*}
Lemma \ref{L:crude-deviation-bd} gives the following bound on the flux of $h_n$ on the interval $[0, t_2]$.
\begin{align*}
J(h_n; [0, t_2]) &\leq \frac{2}{\cos(u_n/2)} \int_0^{t_2/2} \cos(\arcsin(\al) + u_n t)dt \\
&= t_2 \sqrt{1-\al^2} - \frac{t_2^2 \al u_n}{4} + O(u_n^2).
\end{align*}
For small enough $u_n$,  this calculation shows that $J(h_n) < t_2 \sqrt{1-\al^2}$.
\end{proof}

We can now prove that minimal flux paths $h \in \Lip_0$ are unimodal.

\begin{prop}
\label{P:monotone-traj}
Let $h \in \Lip_0$ be such that $J(h) = 1$, and $m(h) \in (0, 1)$. Then $|h(1/2)| = m(h)$. Moreover, if $h(1/2) = m(h)$, then $h$ is strictly increasing on $[0, 1/2]$ and strictly decreasing on $[1/2, 1]$. If $-h(1/2) = m(h)$, then $h$ is strictly decreasing on $[0, 1/2]$ and strictly increasing on $[1/2, 1]$. 
\end{prop}

\begin{proof}
First consider the case where $h \ge 0$. Let $t_m$ be any time when $h(t_m) = m(h)$. Suppose that there exist times $s_1 < s_2 \in [0, t_m]$ such that $h(s_1) = h(s_2)$, and $h(t) \le h(s_1)$ on the interval $[s_1, s_2]$. Define a new path
$$
r(t) = 
\begin{cases}
h(t), \qquad &t \le s_1, t \ge t_m\\
h(t + (s_2 - s_1)), \qquad &t \in (s_1, t_m - (s_2 -s_1)]\\
m(h) \qquad &t \in [t_m -(s_2 - s_1), t_m].
\end{cases}
$$
The path $r$ replaces the segment of $h$ in the interval $[s_1, s_2]$ with a plateau at height $t_m$ at a later time. This operation cannot increase flux since $D_\mu$ is minimized at $0$ (Lemma \ref{L:convex}), so $r$ must be a minimal flux path in $\Lip^*$. However 
$$
J(r; [t_m -(s_2 - s_1), t_m]) = (s_2 - s_1)\sqrt{1 - m^2(h)},
$$
which contradicts Lemmas \ref{L:min-path} and \ref{L:no-plat}. Therefore there is no interval $[s_1, s_2] \sset [0, t_m]$ where $h(s_1) = h(s_2)$ and $h(t) \le h(s_1)$ for $t \in [s_1, s_2]$, so $h$ must be strictly increasing on $[0, t_m]$. Similarly, $h$ is strictly decreasing on $[t_m, 1]$. The point $t_m$ is the unique time when $h$ achieves its maximum.

\medskip

We now show that $t_m = 1/2$. Without loss of generality, assume that $t_m \le 1/2$. Define a new path $g(t) = h(1 - t)$. The path $g \in \Lip_0$ also satisfies $J(g) = 1$ and $m(g) = m$. We have that $g(1/2) = h(1/2)$, and so by Lemma \ref{L:min-path} we can create a path $p \in \Lip^*$ with $J(p) = 1$ and $m(p) = m$ by letting
\begin{equation*}
p(t) =\begin{cases}
&h(t), \qquad t \leq 1/2,\\
&g(t), \qquad t \geq 1/2.
\end{cases}
\end{equation*}
The path $p(t)$ achieves its maximum height at both $t_m$ and $1- t_m$, so we must have that $t_m = 1 - t_m$, and hence $t_m = 1/2$.

\medskip

Now if $h$ is not a strictly non-negative path, then we can create a non-negative path $p(t) = |h(t)|$. By the symmetry of $\mu$ (Theorem \ref{T:local-speeds}), the path $p$ must again have minimal flux, so since $p \ge 0$, by the above argument $p$ is strictly increasing on $[0, 1/2]$ and strictly decreasing on $[1/2, 1]$. Therefore either $h = p$ or $h = -p$, completing the proof.
\end{proof}

We can now use Proposition \ref{P:monotone-traj} to prove uniqueness of minimal flux paths with a given maximum height.
\begin{prop}
\label{P:unique-path}
For every $m \in [0, 1]$, there is at most one $\pi\sqrt{1- y^2}$-Lipschitz path $h \in \Lip_0$ with $J(h) = 1$, $m(h) = m$, and $h \ge 0$. If such a path $h$ exists, then the only other $\pi\sqrt{1- y^2}$-Lipschitz path $g \in \Lip_0$ with $J(g) = 1$ and $m(g) = m$ is $g = -h$.
\end{prop}

\begin{proof}
First observe that the only $\pi\sqrt{1- y^2}$-Lipschitz paths in $\Lip_0$ with $m(h) = 1$ are $h = \pm \sin(\pi t)$. Similarly, the only path $h \in \Lip$ with $m(h) = 0$ is $h = 0$. This proves the proposition for $m \in \{0, 1\}$. Now we assume $m \in (0, 1)$.

\medskip

Suppose that $h, g \in \Lip_0$ are two distinct non-negative paths with $J(h) = J(g) = 1$ and $m(h) = m(g) = m$. By Proposition \ref{P:monotone-traj}, $h(1/2) = g(1/2) = m$. Without loss of generality, there exists a time $t_2 \in [0, 1/2)$ such that $g(t_2) < h(t_2)$. Since $g(t_2) \ge 0, h(0) = 0$ and $h$ is continuous we can find a time $t_1 < t_2$ such that $h(t_1) = g(t_2)$. Define the shifted path
\begin{equation*}
g^*(t) = \begin{cases}
\;\;g(t + (t_2 - t_1)), \qquad &t \leq 1 - (t_2 - t_1) \\
\;\;- g(t + (t_2 - t_1) - 1), \qquad& t \geq 1 - (t_2 - t_1).
\end{cases}
\end{equation*}
By Proposition \ref{P:monotone-traj}, we have that 
$$
m = g^*(1/2 - (t_2 - t_1)) > h(1/2 - (t_2 - t_1)) \quad \mathand \quad g^*(1/2) < h(1/2) = m.
$$
Therefore there is some time $t_3 \in (1/2 - (t_2 - t_1), 1/2)$ such that $g^*(t_3) = h(t_3)$.  Moreover, $g^* \in \Lip^*$, $J(g^*)  = 1$, and $g^*(t_1) = h(t_1)$. Therefore by Lemma \ref{L:min-path} we can create a path $p \in \Lip_0$ with $J(p) = 1$ by letting
\begin{equation*}
p(t) = \begin{cases} 
\;\;h(t), \qquad& t \leq t_1 \mathor t \geq t_3 \\
\;\;g^*(t), \qquad& t \in (t_1, t_3).
\end{cases}
\end{equation*}
This new path $p$ is a non-negative minimal flux path in $\Lip_0$ which does not uniquely achieve its maximum value at $1/2$, contradicting Proposition \ref{P:monotone-traj}.

\medskip

Now, if $g \in \Lip_0$ is another minimal flux path with $J(g) = 1$ and $m(g) = m$, then $|g|$ is a non-negative minimal flux path in $\Lip_0$, so $|g| = h$. Since $h(t) > 0$ for $t \in (0, 1)$ by Proposition \ref{P:monotone-traj}, either $g = h$ or $g = -h$.
\end{proof}

Proposition \ref{P:unique-path} gives us the following easy corollary.

\begin{corollary} 
\label{C:sym}
Any $\pi\sqrt{1-h^2}$-Lipschitz path $h \in \Lip_0$ with $J(h) = 1$ has $h(t) = h(1-t)$ for all $t \in [0, 1]$.
\end{corollary}
\begin{proof}
Without loss of generality, assume that $h \ge 0$. 
Both $h(t)$ and $h(1-t)$ are non-negative minimal flux paths satisfying all the conditions of Proposition \ref{P:unique-path} and with the same maximum height, so by that proposition, $h(t) = h(1-t)$ for all $t$.
\end{proof}

\begin{remark}
\label{R:shift}
Note that the uniqueness proofs in this section automatically give the uniqueness of the paths $h_{a, k}$ in Theorem \ref{T:unique-2}. The fact that uniqueness immediately carries over to shifted paths follows from the strict monotonicity in Proposition \ref{P:monotone-traj}.
\end{remark}

\subsection{Existence of minimal flux paths}
\label{SS:unique}

We now show that for every $m \in [0, 1]$, there exists a minimal flux path in $\Lip_0$ with maximum height $m$. This is a consequence of the following proposition, which shows that the inequality in Theorem \ref{T:particle-crossings} is an equality for minimal flux paths. Recall the definition of the upcrossing probability $P^+_Y(h; [a, b])$ and the downcrossing probability $P^-_Y(h; [a, b])$ from Section \ref{S:flux-intro}.

\begin{prop}
\label{P:flux-upcross}
Let $h \in \Lip^*$ be a path with $J(h) = 1$. Let $Y$ be a subsequential limit of $Y_n$. Then for any interval $[a, b]$, we have that
$$
P^+_Y(h; [a, b])  = J^+(h; [a, b]) \quad \mathand \quad P^-_Y(h; [a, b]) = J^-(h; [a, b]).
$$
\end{prop}

\begin{proof}
We will only prove the first equality, as the second one follows by the symmetry of $Y$ and $\mu$ (Theorem \ref{T:local-speeds}, Proposition \ref{P:Y-symmetries}). We have that
\begin{equation*}
J(h) = J^+(h) +J^-(h) = 1 \le P^+_Y(h; [0, 1]) + P^-_Y(h; [0, 1]).
\end{equation*}
Here the inequality uses that $Y(0) = -Y(1)$ almost surely and $Y(0)$ is uniformly distributed (Theorem \ref{T:bounded-speed}) and a union bound. Therefore by Theorem \ref{T:particle-crossings} (i), we have that $P^+_Y(h; [0, 1]) = J^+(h)$.
Now,
\begin{equation}
\label{E:big-J+}
\begin{split}
P^+_Y(h; [0, 1])& \le
P^+_Y(h; [0, a]) + P^+_Y(h; [a, b]) + P^+_Y(h; [b, 1])\\
 &\le J^+(h; [a, b]^c) + P^+_Y(h; [a, b]) \\
 &\le  J^+(h; [a, b]^c) + J^+(h; [a, b]) = J^+(h).
\end{split}
\end{equation}
To see the first inequality above, note that the union of the three events on the right hand side contains the event on the left hand side minus the set $\{Y(a) = h(a) \mathor Y(b) = h(b)\}$. This set has probability $0$ since $Y(t)$ is uniformly distributed for all $t$ (Theorem \ref{T:bounded-speed}). The second and third inequalities follow from Theorem \ref{T:particle-crossings} (i). 

\medskip

Since $P^+_Y(h; [0, 1]) = J^+(h)$, all the inequalities in \eqref{E:big-J+} must in fact be equalities, proving the proposition.
\end{proof}

To apply the above proposition in order to prove the existence of minimal flux paths at every height, we need three lemmas. The first lemma shows that any two minimal flux paths in $\Lip^*$ cannot cross each other more than once. This lemma also proves part (iii) of Theorem \ref{T:path-set} and part (iv) of Theorem \ref{T:unique-2}.

\begin{lemma}
\label{L:no-cross}
Suppose that $h, g \in \Lip^*$ are non-negative $\pi\sqrt{1-y^2}$-Lipschitz paths with $J(h) = J(g) = 1$. Then there cannot exist times $t_0 < t_1 < t_2 \in [0, 1]$ such that $h(t_0) < g(t_0)$, $h(t_1) > g(t_1)$, and $h(t_2) = g(t_2)$.
In particular, if $g(0) = h(0)$, then either $h \le g$ or $g \le h$.
\end{lemma}

\begin{proof}
Suppose there exist times $t_0 < t_1 < t_2 \in [0, 1]$ such that $h(t_0) < g(t_0)$, $h(t_1) > g(t_1)$, and $h(t_2) = g(t_2)$. First, by possibly shifting the paths as in Lemma \ref{L:path-shift}, we may assume that $t_2 = 1$. Hence $g(1) = h(1) = -g(0) = -h(0)$. Now let $s \in (t_0, t_1)$ be such that $h(s) = g(s)$. Letting $a = g(0)$, Remark \ref{R:shift} implies that $h = h_{a, k_1}$ and $g = h_{a, k_2}$ for some $k_1, k_2 \in [-1, -|a|] \cup [|a|, 1]$. Here the paths $h_{a, k}$ are as in Theorem \ref{T:unique-2}. Without loss of generality, assume that $k_2 > k_1$ and that $k_2 \ge a$; the other cases follow by symmetric arguments.

\medskip

Define $p = \max(h, g)$. By Lemma \ref{L:min-path}, $J(p) = 1$. Moreover, $p(0) = -p(1) = a$, and since $k_2 > k_1$ and $k_2 \ge |a|$, we have that $m(p) = S(p) = k_2$. This contradicts the uniqueness of the path $h_{a, k_2}$ established in Remark \ref{R:shift}.
\end{proof}

The next lemma establishes a strong concavity property for minimal flux paths. 

\begin{lemma}
\label{L:concave-paths} 
For every $k \in [0, 1]$, if there exists a minimal flux, $\sqrt{1-y^2}$-Lipschitz path $h_k \in \Lip_0$ with $h_k \ge 0$ and $m(h_k) = k$, then $\arcsin(h_k)$ is concave. In other words, the local speed $s_k$ of $h_k$ is a non-increasing function of $t$.
\end{lemma}

\begin{proof} 
By the symmetry of $h_k$ (Corollary \ref{C:sym}), it suffices to prove that $\arcsin(h_k)$ is concave on $[0, 1/2]$. Suppose that there exist times $t_1 < t_2 < t_3 \le 1/2$ such that $h(t_2) < p(t_2)$, where $p$ is the path of constant local speed $\close{s} \in [0, \pi]$ with $p(t_1) = h(t_1)$ and $p(t_2) = h(t_2)$. Let 
$$
t_4 = \sup \{t \in [t_1, t_2] : p(t) \le h(t)\} \qquad \mathand \qquad t_5 = \inf \{t \in [t_2, t_3] : p(t) \le h(t)\}.
$$
Then $h(t_4) = p(t_4)$ and $h(t_5) = p(t_5)$, and for every $t \in (t_4, t_5)$, we have that $h(t) < p(t)$. By Lemma \ref{L:convex}, we can find a line $L(s) = as + b$ with $b \ge 0$ such that $L(s) \leq D_\mu(s)$ for all $s$, and such that $L(\close{s}) = D_\mu(\close{s}).$ Therefore 
\begin{align*}
J(h; [t_4, t_5]) 
& \geq \int_{t_4}^{t_5} (as(t) + b) \sqrt{1-h^2(t)} dt \\
&\ge \int_{t_4}^{t_5} ah'(t) dt + \int_{t_4}^{t_5} b \sqrt{1-p^2(t)} dt.
\end{align*}
The inequality in the second line follows since $p(t) > h(t) \ge 0$ for all $t \in (t_4, t_5)$. By the fundamental theorem of calculus, since $p(t_4) = h(t_4)$ and $p(t_5) = h(t_5)$, we have
\begin{align*}
\int_{t_4}^{t_5} ah'(t) dt + \int_{t_4}^{t_5} b \sqrt{1-p^2(t)} dt &= 
\int_{t_4}^{t_5} ap'(t) dt + \int_{t_4}^{t_5} b \sqrt{1-p^2(t)} dt \\
& = J(p; [t_4, t_5]).
\end{align*}
We can then create a non-negative $\pi\sqrt{1 - y^2}$-Lipschitz path $g \in \Lip_0$ with $m(g) = k$ by letting $g(t) = p(t)$ for $t \in [t_4, t_5]$ and $g(t) = h(t)$ otherwise. We have $J(g) \le 1$, contradicting the uniqueness of $h_k$. 
\end{proof}

The next lemma is a consequence of the symmetry and concavity of $Y$.

\begin{lemma}
\label{L:Y-sym-type}
Let $Y$ be a subsequential limit of $Y_n$, let $t \in [0, 1)$, and let $[a, b] \sset [0, 1]$. Then $Y'(t)$ exists almost surely and 
\begin{enumerate}[label=(\roman*)]
\item $\prob(Y'(t) = 0) = 0.$
\item $\prob(Y'(t) > 0 , \; Y(t) \in [a, b]) = \prob(Y'(t) < 0 , \; Y(t) \in [a, b]) = (b-a)/4.$
\end{enumerate}
\end{lemma}

\begin{proof}
By time stationarity of $Y$ (Proposition \ref{P:Y-symmetries} (i)) it suffices to consider $t = 0$. This also implies that
$$
Y'(0) \eqd Y'(U),
$$
for an independent uniform random variable $U$ on $[0, 1]$. Since $Y$ is Lipschitz and hence almost everywhere differentiable, this proves that $Y'(0)$ exists almost surely. By the concavity and strict monotonicity of minimal flux paths (Lemma \ref{L:concave-paths} and Proposition \ref{P:monotone-traj}), we have that
\begin{equation}
\label{E:prob-Y-0}
\prob (Y'(U) = 0) = \prob( m(Y) = |Y(U)|) = 0.
\end{equation}
Define $Z(t) = -Y(1- t)$. Since $Y(0) = -Y(1)$ (Theorem \ref{T:bounded-speed}), we have that
$
Z(0) = Y(0)
$
and $Z'(0) = -Y'(0)$. By Proposition \ref{P:Y-symmetries}, $Y \eqd Z$, so 
$$
\prob(Y'(0) > 0 , Y(0) \in [a, b]) = \prob(Y'(0) < 0 , Y(0) \in [a, b]).
$$
Putting this together with \eqref{E:prob-Y-0} and the fact that $\prob(Y(0) \in [a, b]) = (b-a)/2$ completes the proof.
\end{proof}

We are now ready to prove the existence of minimal flux paths at every height.

\begin{prop}
\label{P:exist-m}
For every $m \in [0, 1]$, there exists a $\pi\sqrt{1-y^2}$-Lipschitz path $h_m \in \Lip_0$ such that $m(h_m) = m$, $J(h_m) = 1$, and $h \ge 0$.
\end{prop}

\begin{proof} First observe that we can set $h_0(t) = 0$ and $h_1(t) = \sin (\pi t)$. Both of these paths satisfy all of the above properties. For this, we use that $D_\mu(\pm \pi) = \pi$ and $D_\mu(0) = 2$ (Lemma \ref{L:convex} and Lemma \ref{L:expect-2}).

\medskip

Let $A \sset [0, 1]$ be the set of all $m$ such there is a $\pi\sqrt{1-y^2}$-Lipschitz path $h_m \in \Lip_0$ with $h_m \ge 0$, $J(h_m) = 1$ and $m(h_m) = m$.  Let $m \in \close{A}$, and let $m_k \in A$ be a sequence converging to $m$. 

\medskip

By the Arzel\`a-Ascoli Theorem, there exists a $\pi\sqrt{1-h^2}$-Lipschitz subsequential limit $h_m \in \Lip_0$ of $h_{m_k}$ with $m(h_m) = m$ and $h_m \ge 0$. By the lower semicontinuity of $J$ (Proposition \ref{P:lsc-E}), $J(h_m) = 1$. Therefore $A$ is closed.

 \medskip
 
 Now suppose that $A \ne [0, 1]$. Then there exists an interval $(\ell, m) \sset [0, 1]$ such that $\ell, m \in A$ and $(\ell, m) \cap A = \emptyset$. Let $k = (\ell + m)/2$, and choose $t_1 < t_2 \in [0, 1/2]$ so that $
 h_m(t_1) = \ell $ and $h_m(t_2) = k.$
We will show that
 \begin{equation}
 \label{E:P=P}
P^+_Y(h_m ; [t_1, t_2]) + P^-_Y(h_m ; [t_1, t_2]) =  P^+_Y\lf(k ; [t_1, t_2]\rg) + P^-_Y\lf(k; [t_1, t_2]\rg).
 \end{equation}
 Here we use $k$ as shorthand for the line $h(t) = k$.
 
 \medskip
 
Let $g \in \Lip^*$ be any minimal flux $\pi\sqrt{1- y^2}$-Lipschitz path. We will analyze how $g$ crosses the paths $k$ and $h_m$ in the interval $[t_1, t_2]$. By the unimodality established in Proposition \ref{P:monotone-traj}, $g$ upcrosses $k$ at most once in the interval $[0, 1]$, and downcrosses $k$ at most once. Moreover, by Lemma \ref{L:no-cross}, $g$ crosses $h_m$ at most once.

\medskip

If $g(t_1) \notin [\ell, k]$, and $g(t_2) \ne k$, then the total number of crossings of the lines $h_m$ and $k$ is even. Hence, either $g$ doesn't cross either path in the interval $[t_1, t_2]$, or else $g$ first crosses $k$ going down, and then crosses $k$ going up, all while staying above $h_m$. This would force $g$ to have a positive local minimum in the interval $[t_1, t_2]$. This is forbidden by the monotonicity properties of minimal flux paths (Proposition \ref{P:monotone-traj}).

\medskip

If $g(t_1) \in (\ell, k)$ and $g'(t_1) < 0$, then again by the monotonicity properties of $g$, $g$ downcrosses $h_m$, does not upcross $h_m$, and neither upcrosses nor downcrosses $k$. 

\medskip

If $g(t_1) \in (\ell, k)$, $g'(t_1) > 0$, and $g(t_2) \ne k$, then $m(g) \ge m$ since $(\ell, m) \cap A = \emptyset$. By the monotonicity of $g$ (Proposition \ref{P:monotone-traj}), $g(t^*) = m(g)$ for some $t^* > t_1$. If $g(t_2) < h(t_2)$, then since both $g$ and $h_m$ are minimal flux paths, the restrictions on crossings imposed by Lemma \ref{L:no-cross} imply that $t^* < t_2$. In this case, $g$ both upcrosses and downcrosses $k$ in the interval $[t_1, t_2]$, and downcrosses but does not upcross $h_m$.

\medskip

 On the other hand, if $g(t_2) > h(t_2)$, then $g$ upcrosses $k$ in the interval $[t_1, t_2]$ but does not downcross $k$, and does not cross $h$ at all in this interval.

\medskip

Now let $Y$ be any subsequential limit of $Y_n$. By Theorem \ref{T:bounded-speed} and Lemma \ref{L:Y-sym-type} (i), we have that
$$
\prob \Big(Y(t_1) \notin \{\ell, k\}, Y(t_2) \ne k, Y'(t_1) \text{ exists and is not equal to $0$}\Big) = 1.
$$
Therefore by the above analysis of minimal flux $\pi\sqrt{1- y^2}$-Lipschitz paths $g \in \Lip^*$, we have that
\begin{align*}
P^+_Y(h_m ; [t_1, t_2]) - P^+_Y\lf(k ; [t_1, t_2]\rg) &= - \prob(Y'(t_1) > 0, Y(t_1) \in [\ell, k]), \quad \mathand \\
P^-_Y(h_m ; [t_1, t_2]) - P^-_Y\lf(k ; [t_1, t_2]\rg) &= \prob(Y'(t_1) < 0, Y(t_1) \in [\ell, k]).
\end{align*}
Equation \eqref{E:P=P} then follows from the symmetry established in Lemma \ref{L:Y-sym-type} (ii). Now by Proposition \ref{P:flux-upcross}, the left hand side of \eqref{E:P=P} is equal to $J(h_m ; [t_1, t_2])$, and by Theorem \ref{T:particle-crossings} (i), the right hand side of \eqref{E:P=P} is bounded above by $J\lf(k; [t_1, t_2] \rg)$.
However, since $D_\mu$ is minimized at $0$ by Lemma \ref{L:convex}, we can also easily calculate that
\begin{align*}
J(h_m ; [t_1, t_2]) &\ge \frac{D_\mu(0)}{2} \int_{t_1}^{t_2} \sqrt{1 - h^2_m(t)}dt \\
&> \frac{D_\mu(0)}2(t_2 - t_1)\sqrt{1 - {k^2}} = J\lf(k; [t_1, t_2] \rg).
\end{align*}
This is a contradiction, so $A$ must be the whole interval $[0,1]$.
\end{proof}

\section{The derivative distribution}
\label{S:path-speed}

Let
$
\prob_{a}(Y \in \cdot)
$
be any regular conditional distribution of $Y$ given that $Y(0) = a$. In this section, we use the structure of minimal flux paths to find $\prob_a(Y'(0) \in \cdot)$.

\begin{prop}
\label{P:local-global-deriv}
Let $Y$ be any subsequential limit of $Y_n$. With probability one, $Y'(0)$ exists. Moreover,
\begin{equation}
\label{E:X-Y'}
\prob_{a} \lf(\frac{Y'(0)}{\sqrt{1 - a^2}} \in \cdot\rg) = \mu,
\end{equation}
for Lebesgue-a.e. $a \in (-1, 1)$.
\end{prop}

As an immediate corollary of Proposition \ref{P:local-global-deriv}, we will be able to find the distribution of the maximum height function $m(Y)$ for any subsequential limit $Y$ of $Y_n$. This will allow us to show that the sequence $\{Y_n\}$ has a unique limit point. 

\medskip

Moreover, Proposition \ref{P:local-global-deriv} will be used later to prove an integral transform formula for the local speed distribution $\mu$, which will allow us to determine $\mu$, and then $Y$. This is done in Section \ref{S:integral-formula}.

\medskip
 
To prove Proposition \ref{P:local-global-deriv}, we first show that minimal flux paths fill space. For this lemma, we use the notation of Theorem \ref{T:unique-2}.

\begin{lemma}
\label{L:closed-int}
Let $a \in [0, 1]$, and let 
$$
M_a = \lf\{h_{a, k} : k \in  [-1, -a] \cup [a, 1] \rg\}.
$$
 Then for any $t \in [0, 1]$, we have that
 $$
\{h(t) : h \in M_a\} = [\sin(-\pi t + \arcsin(a)), \sin(\pi t + \arcsin(a))].
$$
\end{lemma}

\begin{proof}
Let $K_{a, t} = \{h(t) : h \in M_a\}.$ The set $K_{a, t}$ contains the endpoints of the above interval, since $M_a$ contains the two extremal functions $h_{a, -1}(t) = \sin(-\pi t + \arcsin(a))$ and $h_{a, 1}(t) = \sin(\pi t + \arcsin(a))$. Moreover, the ordering on minimal flux paths (Theorem \ref{T:unique-2} (iv)) implies that 
$$
K_{a, t} \sset [\sin(-\pi t + \arcsin(a)), \sin(\pi t + \arcsin(a))].
$$
Now suppose that $c \in (\sin(-\pi t + \arcsin(a)), \sin(\pi t + \arcsin(a)))$. Let 
$$
A = \{h_{a, k} \in M_a : h_{a, k}(t) \le c\} \qquad \mathand \qquad B = \{h_{a, k} \in M_a : h_{a, k}(t) > c\}.
$$
Both of these sets are non-empty.
 By Theorem \ref{T:unique-2} (iv), $M_a$ is a linearly ordered subset of the set of $\pi \sqrt{1-y^2}$-Lipschitz paths in $\Lip$ with the usual partial order on functions. Therefore the sets $A$ and $B$ have a supremum and infimum in $\Lip$ by the Arzel\`a-Ascoli Theorem. By the lower semicontinuity of flux (Proposition \ref{P:lsc-E}), $\sup A$ and $\inf B$ lie in the set $M_a$. This, combined with the monotonicity from Theorem \ref{T:unique-2} (iv), implies that
 \begin{align*}
  \sup A &= h_{a, k_A}, \quad \text{ where } \quad k_A = \sup \{ k : h_{a, k} \in A\}, \quad \mathand \\
   \inf A &= h_{a, k_B}, \quad \text{ where } \quad k_B = \inf \{ k : h_{a, k} \in B\}.
 \end{align*}
Since $A \cup B = M_a$, we either have that $k_A = k_B$, or else $k_A = -a$ and $k_B = a$. Since $h_{a, a} = h_{-a, a}$ by Theorem \ref{T:unique-2} (v), this implies that $h_{a, k_A}(t) = h_{a, k_B}(t) =  c$, so $c \in K_{a, t}$.
\end{proof}

\begin{proof}[Proof of Proposition \ref{P:local-global-deriv}] First, $Y'(0)$ exists almost surely by Lemma \ref{L:Y-sym-type}.
We now show that for every $q \in (-\pi, \pi)$ and $b < c \in (-1, 1)$, that
\begin{equation}
\label{E:deriv-integrated}
\int_{b}^c \expt_{a} (Y'(0) - \sqrt{1 - a^2}q)^+ da = D_\mu^+(q) \int_b^c \sqrt{1-a^2}da.
\end{equation}
Here $\expt_a$ is the expectation taken with respect to $\prob_a$. Lemma \ref{L:closed-int} guarantees that there exists a time $t_0>0$ such that for any $r \in [0, t_0]$ and $a \in [b, c]$, there is a minimal flux $\pi\sqrt{1-y^2}$-Lipschitz path $g \in \Lip^*$ with $g(0) = a$ and $g(r) = \sin(\arcsin(a) + qr)$. Call this path $g_{a, r}$. Letting
$$
s_{a, r}(t) = \frac{g'_{a, r}(t)}{\sqrt{1 - g^2_{a, r}(t)}}
$$
be the local speed of $g_{a, r}$, we have that
\begin{align}
\nonumber
J^+(g_{a, r} ; [0, r]) &\ge \frac{1}2 \min_{t \in [0, r]} \sqrt{1 - g^2_{a, r}(t)} \int_0^r D^+_\mu(s_{a, r}(t)) dt \\
\label{E:J-lower}
&\ge  \frac{1}2 \min_{t \in [0, r]} {\sqrt{1 - g^2_{a, r}(t)}} r D^+_\mu(q).
\end{align}
Here the second inequality follows by Jensen's inequality, using that $D^+_\mu$ is convex (Lemma \ref{L:convex}) and that the average local speed of $g_{a, r}$ is $q$. Now letting $g_a(t) = \sin(\arcsin(a) + qt)$, Lemma \ref{L:min-path} implies that
\begin{equation}
\label{E:g-J-h}
J^+(g_{a, r} ; [0,r]) \le J^+(g_a ; [0, r]) \le \frac{1}2  \max_{t \in [0, r]} {\sqrt{1 - g^2_{a}(t)}} r D^+_\mu(q).
\end{equation}
Combining this inequality with \eqref{E:J-lower} implies that
$$
\lim_{r \to 0^+} \frac{J^+(g_{a, r} ; [0, r])}{r} = \frac{D_\mu^+(q)}2 \sqrt{1 - a^2}.
$$
Moreover, \eqref{E:g-J-h} implies that $J^+(g_{a, r} ; [0, r]) \le rD^+_\mu(q)$ for all $a \in [b, c], r \in [0, t_0]$. Therefore
by the bounded convergence theorem, we have that
\begin{equation}
\label{E:want-1}
\lim_{r \to 0} \int_{b}^c \frac{J^+(g_{a, r} ; [0, r])}r da = \frac{D^+_\mu(q)}2\int_b^c \sqrt{1- a^2}da.
\end{equation}
Now recall that $P^+_Y(g_{a, r} ; [0, r])$ is the probability that $Y$ upcrosses $g_{a, r}$ in the interval $[0, r]$. Since minimal flux paths cross at most once (Lemma \ref{L:no-cross}),
$$ 
\prob(Y(0) < a, Y(r) > g_{a, r}(r))= P^+_Y(g_{a, r} ; [0, r]).
$$
Therefore by Proposition \ref{P:flux-upcross}, we can write 
\begin{align}
\nonumber
\int_{b}^c \frac{J^+(g_{a, r} ; [0, r])}r da &= \int_{b}^c \frac{\prob(Y(0) < a, Y(r) > g_{a, r}(r))}r da.
\end{align}
Then defining $P_x(a, r) = \prob_x(Y(r) > g_{a, r}(r))/r$, the right hand side above is equal to 
\begin{align}
\label{E:J-+-da}
\int_{b}^c \frac{1}2 \int_{-1}^a P_x(a, r)dxda = \frac{1}2 \int_{-1}^c \int_{x \vee b}^c  P_x(a, r) dadx.
\end{align}
Now define $S_{x, y} = \sin(\arcsin(x) + (\pi - q)y)$. Since $Y$ is $\pi\sqrt{1-h^2}$-Lipschitz almost surely (Theorem \ref{T:bounded-speed}), for almost every $x \in [-1, 1]$, we have that
$$
\prob_x(Y(t) > g_{a, r}(r)) = 0
$$
whenever $a > S_{x, r}.$ Therefore we can rewrite the right hand side of \eqref{E:J-+-da} as
\begin{equation}
\label{E:want-2}
\begin{split}
\frac{1}2 \lf[\int_b^{c} \int_{x}^{S_{x, r}} P_x(a, r) da dx - \int_{S_{c, -r}}^c \int_c^{S_{x, r}} P_x(a, r) da dx + \int_{S_{b, -r}}^b \int_b^{S_{x, r}} P_x(a, r) da dx\rg].
\end{split}
\end{equation}
In the second and third terms above, the integrand is of size $O(r^{-1})$, whereas the region of integration is of size $O(r^{2})$. Therefore these terms go to zero as $r \to 0$. Now letting
$$
f(a, x, s, r) = \frac{\sin(\arcsin(a) + sr) - x}r,
$$
and making the substitution $a \mapsto u = f(a, x, q, r)$, we can write the first term of \eqref{E:want-2} as
\begin{equation}
\label{E:b-c}
\begin{split}
\int_b^c \int_{f(x, x, q, r)}^{f(x, x, \pi, r)} \prob_x\lf(\frac{Y(r) - Y(0)}r > u \rg) \frac{\cos(\arcsin(ru + x) - qr)}{\sqrt{1 - (ru + x)^2}}du dx.
\end{split}
\end{equation}
Moreover, for every $x \in [-1, 1]$ such that $Y'(0)$ exists $\prob_x$-almost surely, and for $u \in [q\sqrt{1 - x^2}, \pi\sqrt{1 - x^2}]$ for which $u$ is not an atom of the distribution $\prob_x(Y'(0) \in \cdot)$, we have that
$$
\prob_x\lf(\frac{Y(r) - Y(0)}r > u \rg) \xrightarrow[r \to 0]{} \prob_x(Y'(0) > u).
$$
Therefore by the dominated convergence theorem, \eqref{E:b-c} converges to 
$$
\int_b^c \int_{q\sqrt{1 - x^2}}^{\pi\sqrt{1 - x^2}} \prob_x(Y'(0) > u) du dx = \int_{b}^c \expt_{x} (Y'(0) - \sqrt{1 - x^2}q)^+ dx \qquad \mathas \quad r \to 0.
$$
The equality above again uses that $|Y'(0)| \le \pi\sqrt{1- Y^2(0)}$ almost surely, since $Y$ is $\pi\sqrt{1-y^2}$-Lipschitz. 
Combining this with \eqref{E:want-1}, \eqref{E:J-+-da}, and \eqref{E:want-2} proves \eqref{E:deriv-integrated}.
Now define 
\[
d(Y) = \frac{Y'(0)}{\sqrt{1 - Y^2(0)}}.
\]
By \eqref{E:deriv-integrated}, for almost every $a \in (-1, 1)$, for every $q \in \rat \cap (\pi, \pi)$, we have that 
$$
\sqrt{1 - a^2}\expt_{a} \lf(d(Y) -  q\rg)^+ = \expt_{a} \lf(Y'(0) - \sqrt{1 - a^2} q\rg)^+ = \sqrt{1- a^2}D_\mu^+(q).
$$
Therefore by continuity in $s$ of the functions $\expt_{a} \lf(d (Y) - s\rg)^+$ and $D_\mu^+(s)$, we have that
\begin{equation}
\label{E:daY}
\expt_{a} \lf(d (Y) - s\rg)^+ = D_\mu^+(s) \qquad \mathforall s \in (-\pi, \pi).
\end{equation}
Now, $D_\mu^+(s)$ is Lipschitz (Lemma \ref{L:convex}) and hence differentiable almost everywhere, so we can differentiate both sides of the above equation to get that for almost every $s \in (-\pi, \pi)$, that
$$
\prob_a(d(Y) > s) = \mu(s, \infty).
$$
Finally, for almost every $a \in (-1, 1)$, we have that 
$$
\prob_a(d(Y) \in [-\pi, \pi]) = \mu[\pi, \pi] = 1
$$
since $Y$ is almost surely $\pi\sqrt{1 - y^2}$-Lipschitz (Theorem \ref{T:bounded-speed}).
Therefore $\prob_a(d(Y) \in \cdot) = \mu$.
\end{proof}

\subsection{The max-height distribution and the uniqueness of $Y$}

As an application of Proposition \ref{P:local-global-deriv}, we can find $\prob(m(Y) > a)$ for all $a \in [0, 1]$. This allows us to deduce the uniqueness of $Y$.

\begin{theorem}
\label{T:max-height}
Let $Y$ be any subsequential limit of $Y_n$. For all $a \in [0, 1]$, we have that 
$ \prob(m(Y) > a) = \sqrt{1 - a^2}$. That is, $m(Y) \eqd \sqrt{1- U^2}$ for a uniform random variable $U$ on $[0, 1]$.
\end{theorem}

\begin{proof}
Let $V(Y)$ be the total variation of $Y$ on the interval $[0, 1]$. We have that
\begin{align*}
\expt V(Y) &= \expt \int_0^1 |Y'(t)|dt = \expt \lf(\frac{|Y'(0)|}{\sqrt{1-Y^2(0)}}\rg) \sqrt{1 - Y^2(0)} = 2 \expt \sqrt{1 - U^2},
\end{align*}
where $U$ is a uniform random variable on $[0, 1]$. Here the second equality follows from time stationarity of $Y$ (Proposition \ref{P:Y-symmetries} (i)). The third equality follows by Proposition \ref{P:local-global-deriv}, the fact that the first moment of $\mu$ is $2$ (Lemma \ref{L:expect-2}), and the fact that $Y(0)$ is uniformly distributed (Theorem \ref{T:bounded-speed}). 

\medskip

Moreover, the characterization of minimal flux paths (Theorem \ref{T:path-set}) implies that 
$
V(Y) = 2 \expt m(Y),
$
so $\expt m(Y) = \expt \sqrt{1 - U^2}$.
Finally, the bound in Lemma \ref{L:max-heights} implies that 
$$
\prob(m(Y) > a) \le \frac{D_\mu(0)}2 \sqrt{1 - a^2} = \sqrt{1- a^2} = \prob(\sqrt{1 - U^2} > a),
$$
so the random variable $m(Y)$ is stochastically dominated by $\sqrt{1 - U^2}$. Hence $m(Y) \eqd \sqrt{1 - U^2}$.
\end{proof}

Now we can prove the existence of a unique limit $Y$ of $Y_n$. First extend the paths $h_m$ defined in Theorem \ref{T:path-set} to paths $h_m:[0, \infty) \to [-1, 1]$ by letting $h_m(t) = -h_m(t - 1)$ for all $t > 1$.

\begin{theorem}
\label{T:Y-unique}
The sequence $Y_n$ has a distributional limit $Y$ given by
$$
Y(t) = h_{\sqrt{1 - U^2}}(V + t).
$$
Here $U$ is a uniform random variable on $[0, 1]$ and $V$ is uniform on $[0, 2]$, independent of $U$.
\end{theorem}

\begin{proof}
By Theorem \ref{T:unique-2} and Theorem \ref{T:particle-crossings} (iii), any subsequential limit $Y$ of $Y_n$ is supported on the set of paths 
$$
\{h_m (\cdot + u) : m \in [0, 1], u \in [0, 2]\}, 
$$ 
so we can write 
$
Y(\cdot) = h_M(\cdot + V),
$
for a pair of random variables $(M, V)$. By Theorem \ref{T:max-height}, $M = \sqrt{1 - U^2}$, for a uniform random variable $U$ on $[0, 1]$. By the time stationarity of $Y$ (Proposition \ref{P:Y-symmetries} (i)), for any $t \in [0, 2]$, we have that
$$
h_M(\cdot + t +  V) \eqd h_M(\cdot +  V),
$$
which in turn implies that $(M, V) \eqd (M, V + t \Mod 2 )$ for any $t \in [0, 2]$. This implies that $V$ has uniform distribution, and that $M$ is independent of $V$.
\end{proof}

\section{The integral transform formula}
\label{S:integral-formula}
Let $\mu_+$ be the pushforward of the local speed distribution $\mu$ under the map $x \mapsto |x|$.
In this section we prove an integral transform formula for $\mu_+$. This transform formula allows us to identify $\mu$ as the arcsine distribution on $[-\pi, \pi]$. Once we know $\mu$, we can find the set of minimal flux paths $h_m$ introduced in Theorem \ref{T:path-set}, and then in turn identify $Y$.

\medskip

 To find the integral transform formula, we will calculate $\prob(m(Y) > k)$ in two different ways. We first give a heuristic explanation of how to do this when the local speed distribution $\mu$ has no atoms. By Theorem \ref{T:max-height}, we have that
$$
\prob(m(Y) > k) = \sqrt{1 - k^2}.
$$
We can also calculate $\prob(m(Y) > k)$ by integrating the marginal probabilities $\prob_a(m(Y) > k)$. This gives that
\begin{equation}
\label{E:heur1}
\sqrt{1 - k^2} = \frac{1}2 \int_{-1}^1 \prob_a(m(Y) > k)da = 1 - k + \frac{1}2 \int_{-k}^k \prob_a(m(Y) > k)da.
\end{equation}
Now we want to find an expression for $\prob_a(m(Y) > k)$ when $|a| < k$.  By the ordering on minimal flux paths, if $Y(0) = a$, then $m(Y) > k$ if and only if $|Y'(0)|$ is greater than some threshold value. Therefore for some $s_{a, k} \in \real$, we have that
\begin{equation}
\label{E:heur2}
\prob_a(m(Y) > k) = \prob_a (|Y'(0)| > \sqrt{1 - a^2}s_{a, k}) = \mu_+(s_{a, k}, \infty).
\end{equation}
The final equality here follows from Proposition \ref{P:local-global-deriv}. To find $s_{a, k}$, we calculate $\prob(m(Y) > k | \; m(Y) > a)$. Set 
\[
T_a = \inf \{t \in [0, 1] : |Y(t)| = a\}.
\]
Note that the time $T_a$ may not exist, in which case we set $T_a = \infty$. We should have that
$$
\prob(m(Y) > k | \; m(Y) > a) = \prob(|Y'(T_a)| > \sqrt{1 - a^2}s_{a, k} | T_a < \infty).
$$
Now, the ``amount of time" that $Y$ spends with $|Y| = a$ is inversely proportional to its speed at that location. Therefore the distribution of $|Y'(T_a)|$ should be a size-biased version of the distribution of $Y'(0)$ given that $|Y(0)| = a$. Hence,  
$$
\prob(|Y'(T_a)| > \sqrt{1 - a^2}s_{a, k}) = \hat{\mu}_+(s_{a, k}, \infty),
$$
where $\hat{\mu}_+$ is the size-biased distribution of $\mu_+$ (we define this formally in the next paragraph). Finally, we can also calculate $\prob(m(Y) > k | \; m(Y) > a)$ using Theorem \ref{T:max-height}. This gives that
$$
\hat{\mu}_+(s_{a, k}, \infty) = \prob(m(Y) > k | \; m(Y) > a) = \frac{\prob(m(Y) > k)}{\prob(m(Y) > a)} = \frac{\sqrt{1 - k^2}}{\sqrt{1 - a^2}}.
$$
We can combine this with \eqref{E:heur2} and \eqref{E:heur1} to get an integral transform formula involving the function $r_{\mu_+}(x) = S(\hat{S}^{-1}(x))$, where $S$ and $\hat{S}$ are the survival functions of $\mu_+$ and $\hat{\mu}_+$, respectively. 

\medskip

We now precisely define everything that is needed to state the integral transform formula. For a probability measure $\nu$ on $(0, \infty)$ with finite mean, define the {\bf size-biased distribution} $\hat{\nu}$ on $(0, \infty)$ by the Radon-Nikodym derivative formula
 $$
 \frac{d\hat{\nu}}{d \nu}(x) = \frac{x}{\int_0^\infty x d \nu(x)}.
 $$
In order to define the integral transform when $\mu$ has atoms, we define the \textbf{extended survival function} $S:\real \times [0,1] \to [0, 1]$ of a probability measure $\nu$ by
$$
S(x, q) = \nu(x, \infty) + (1 - q) \nu(x).
$$
$S$ is a non-increasing, continuous function in the lexicographic ordering on $\real \times [0,1]$. The function $S$ can be thought of as the survival function in the lexicographic ordering for $\nu \X \scrL$, where $\scrL$ is uniform measure on $[0, 1]$. Now define the \textbf{ size-bias ratio function} $r_\nu(x):[0, 1] \to [0, 1]$ of a probability measure $\nu$ on $(0, \infty)$ with finite mean by
\begin{equation*}
r_\nu(x) = S(\hat{S}^{-1}(x)),
\end{equation*}
where $\hat{S}$ is the extended survival function of $\hat{\nu}$ and $S$ is the extended survival function of $\nu$. Here the inverse function is given by
$$
\hat{S}^{-1}(x) = \sup \{(y, q) \in \real \times [0,1] : \hat{S}(y, q) = x \},
$$
where the supremum is taken with respect to the lexicographic ordering. When $\nu$ has no atoms, we can define $r_\nu$ in terms of the usual survival functions. 
\begin{prop}[The integral transform] 
\label{P:integral-equality} Let $\mu_+$ be the pushforward of the measure $\mu$ under the map $f(x) = |x|$. 
For every $k \in (0, 1)$, we have that
\begin{equation}
\label{E:integral-H}
 1 - k + \int_0^k r_{\mu_+} \lf(\frac{\sqrt{1 - k^2}}{\sqrt{1- x^2}}\rg) dx = \sqrt{1-k^2}.
\end{equation}
\end{prop}

To prove Proposition \ref{P:integral-equality}, we first establish that $\mu(0) = 0$. 
\begin{lemma}
\label{L:X-atom}
$\mu(0) = 0$.
\end{lemma}

\begin{proof} 
By Lemma \ref{L:Y-sym-type}, $\prob \lf(\frac{Y'(0)}{\sqrt{1 - Y^2(0)}} = 0 \rg) = 0$. By Proposition \ref{P:local-global-deriv}, this is equal to $\mu(0)$. \end{proof}

Now let $m < 1$, and let $h_m$ be as in Theorem \ref{T:path-set}. For any $a \in [0, m]$, define
$$
s_{a, m} = \frac{1}{\sqrt{1 - a^2}} \liminf_{r \to 0} \frac{|h_m(t^* + r) - h_m(t^*)|}{r},
$$
where $t^*$ is any point where $h_m(t^*) = a$. Note that $s_{a, m}$ is independent of $t^*$ by the symmetry of minimal flux paths (Theorems \ref{T:path-set} (i)). If $h_m$ is differentiable at $t^*$, then $s_{a, m} = |h_m'(t^*)|/\sqrt{1 - a^2}$. 
We now prove an integral formula relating the speeds $s_{a, m}$ to the local speed distribution.

\begin{lemma}
\label{L:bias-unbias}
Let $\nu$ be the law of $m(Y)$. For almost every $a \in [0, 1)$, for every $k \in (a, 1)$ there exists a constant $q_{a, k} \in [0, 1]$ such that
\begin{equation}
\begin{split}
\prob_a(m(Y) > k) &= \frac{1}{\sqrt{1 - a^2}} \int \indic \Big(s_{a, m} > s_{a, k} \;\; \mathor \;\; s_{a, m} = s_{a, k}, \; m > k\Big) \frac{2}{s_{a, m}} d\nu(m) \\
\label{E:s-a-k}
&= \mu_+(s_{a, k}, \infty) + q_{a, k}\mu_+(s_{a, k}).
\end{split}
\end{equation}
Moreover, $\prob_a(m(Y) > k)$ is a continuous function of $k \in (a, 1)$ for almost every $a \in [0, 1)$.
\end{lemma}

\begin{proof}
Let $h_m$ be as in Theorem \ref{T:path-set}, and let $a < b \in [0, 1)$. We first compute the amount of time that $h_m$ spends in the interval $[a, b]$. Define $h_m^{-1}:[0, m] \to [0, 1/2]$ by
$$
h_m^{-1}(x) = \inf \{t : h_m(t) = x \}.
$$
By the strict monotonicity and symmetry of $h_m$ (Theorem \ref{T:path-set} (i), (ii)), we can write
\begin{equation}
\label{E:L-h}
\scrL\{t: |h_m(t)| \in [a, b]\} = 2[h_m^{-1}(b) -  h_m^{-1}(a)],
\end{equation}
where $\scrL$ is Lebesgue measure on $[0, 1]$.
Now by the concavity of $h_m$ (Lemma \ref{L:concave-paths}), the inverse $h_m^{-1}$ is absolutely continuous with derivative $2/[s_{x, m} \sqrt{1 - x^2}]$ for almost every $x$. Therefore the left hand side of \eqref{E:L-h} is equal to 
$$
\int_a^{b} \frac{2}{s_{x, m} \sqrt{1 - x^2}}dx.
$$
Now letting $U$ be a uniform random variable on $[0, 1]$ that is independent of $Y$, for any $a < b < k \in [0, 1)$, we have that
\begin{equation*}
\begin{split}
\prob\big(m(Y) > k \mathand Y(0) \in [a, b]\big) &= \frac{1}2\prob\big(m(Y) > k \mathand |Y(U)| \in [a, b]\big) \\
&= \frac{1}2 \int_k^1 \scrL\{t: |h_m(t)| \in [a, b]\} d\nu(m) \\
&= \int_a^{b} \frac{1}{\sqrt{1 - x^2}} \int_k^1 \frac{1}{s_{x, m}} d\nu(m)dx.
\end{split}
\end{equation*}
The first equality above follows by the time stationarity and symmetry of $Y$ (Proposition \ref{P:Y-symmetries} (i) and (ii)). The second equality follows since $Y$ is supported on shifts of the minimal flux paths $h_m$ (Theorem \ref{T:unique-2}). This implies that for almost every pair $a < k \in [0, 1)$, we have
\begin{equation}
\label{E:f-amid}
\prob_a(m(Y) > k) = \frac{1}{\sqrt{1 - a^2}} \int_k^1 \frac{2}{s_{a, m}} d\nu(m).
\end{equation}
Now by the ordering on minimal flux paths (Theorem \ref{T:unique-2}(iv)), we have that
\begin{equation}
\label{E:speed-mono}
\text{if} \;\; s_{a, m(Y)} < s_{a, k}, \qquad \text{then } m(Y) < k.
\end{equation}
This allows us to rewrite \eqref{E:f-amid} to get that
\begin{equation}
\label{E:f-amid-2}
\prob_a(m(Y) > k) = \frac{1}{\sqrt{1 - a^2}} \int \indic \Big(s_{a, m} > s_{a, k} \;\; \mathor \;\; s_{a, m} = s_{a, k}, m > k\Big) \frac{2}{s_{a, m}} d\nu(m)
\end{equation}
for almost every pair $a < k \in [0, 1)$.
Moreover, for almost every $a \in [0, 1)$, the derivative $Y'(0)$ exists almost surely (Proposition \ref{P:local-global-deriv}). In particular,
$$
|Y'(0)| = \sqrt{1 - a^2}s_{a, m(Y)} \qquad \prob_a\text{-almost surely}.
$$
This uses the fact that $Y$ is supported on shifts of $h_{m(Y)}$. Also, $Y'(0)/\sqrt{1- a^2}$ has distribution $\mu$ for almost every $a$ (Proposition \ref{P:local-global-deriv}). Therefore \eqref{E:speed-mono} implies that for almost every $a \in [0, 1)$, for every $k \in (a, 1)$ there exists a constant $q_{a, k} \in [0, 1]$ such that
\begin{equation}
\label{E:Pa-mY}
\prob_a(m(Y) > k) = \mu_+(s_{a, k}, \infty) + q_{a, k}\mu_+(s_{a, k}).
\end{equation}

Now let $a$ be such that \eqref{E:Pa-mY} holds for every $k \in (a, 1)$ and \eqref{E:f-amid} and \eqref{E:f-amid-2} hold for almost every $k \in (a, 1)$. 
By concavity of minimal flux paths (Lemma \ref{L:concave-paths}), $s_{a, m} > 0$ whenever $m > a$. Therefore since $\nu$ has a Lebesgue density by Theorem \ref{T:max-height}, the right hand side of \eqref{E:f-amid} is continuous and non-increasing. Since both sides of \eqref{E:f-amid-2} are also non-increasing, and are equal to the right hand side of \eqref{E:f-amid} for almost every $k \in (a, 1)$, they must be equal for every $k \in (a, 1)$. Combining this with \eqref{E:Pa-mY} proves the lemma.
\end{proof}

\begin{lemma}
\label{L:equal-measure} 
Let $\mathfrak{s}_a$ be the law of $\prob(s_{a, m(Y)} \in \cdot \; | m(Y) > a)$, and define the measure $\bar{\mathfrak{s}}_a$ by the Radon-Nikodym formula 
$$
\frac{2}{s} d\mathfrak{s}_a(s) = d\bar{\mathfrak{s}}_a(s).
$$
Then $\bar{\mathfrak{s}}_a = \mu_+$ for almost every $a \in [0, 1)$.
In particular, for such $a$, for all $k \in (a, 1)$ we have
\begin{equation}
\label{E:q-ak}
q_{a, k}\mu_+(s_{a, k}) = \frac{2}{s_{a, k}\sqrt{1 - a^2}} \int \indic \Big(s_{a, m} = s_{a, k}, m > k\Big)d\nu(m).
\end{equation}
\end{lemma}

Implicit in the statement of Lemma \ref{L:equal-measure} is the claim that $\mathfrak{s}_a > 0$ almost surely for all $a$. This follows since $s_{a, k} > 0$ whenever $a < k$, which is a consequence of the concavity of minimal flux paths (Lemma \ref{L:concave-paths}).

\begin{proof}
Let $a$ be such that \eqref{E:s-a-k} holds for every $k \in (a, 1)$ and such that $\prob_a(m(Y) > k)$ is continuous. Further assume that
$$
\prob_a\lf( Y'(0) \text{ exists and is non-zero }, m(Y) < 1, J(Y) = 1 \rg) = 1.
$$
These conditions hold for almost every $a \in [0, 1)$ (Proposition \ref{P:local-global-deriv}, Lemma \ref{L:X-atom}). Noting that $\sqrt{1 - a^2} = \prob(m(Y) > a)$ by Theorem \ref{T:max-height}, equation \eqref{E:s-a-k} implies that for every $k \in (a, 1)$, there exists a $p_{a, k} \in [0, 1]$ such that
\begin{equation}
\label{E:comp}
\prob_a(m(Y) > k) = \bar{\mathfrak{s}}_a(s_{a, k}, \infty) + p_{a, k}\bar{\mathfrak{s}}_a(s_{a, k}) = \mu_+(s_{a, k}, \infty) + q_{a, k}\mu_+(s_{a, k}).
\end{equation}
Now, since $Y'(0) \ne 0$, $\prob_a$-almost surely, the concavity of minimal flux paths (Lemma \ref{L:concave-paths}) implies that $\prob_a(m(Y) = a) = 0$. Moreover, since $m(Y) < 1$, $\prob_a$-almost surely, we have that $\prob_a(m(Y) \in (a, 1)) = 1$, and so
\begin{equation}
\label{E:mY}
\lim_{k \to 1} \prob_a(m(Y) > k) = 0 \quad \mathand \quad \lim_{k \to a} \prob_a(m(Y) > k) = 1.
\end{equation}
Therefore, \eqref{E:comp} and the continuity of $\prob_a(m(Y) > k)$ implies that 
\begin{equation}
\label{E:full}
\mu_+ \lf( s_{a, k} : k \in (a, 1) \rg) = \bar{\mathfrak{s}}_a \lf( s_{a, k} : k \in (a, 1) \rg) = 1.
\end{equation}
Now fix $k \in (a, 1)$. Let 
$$
k^* = \inf \{\ell  \in [k, 1] :  s_{a, k} < s_{a, m} \;\; \text{ for all } \;\;m > \ell \}.
$$
Equation \eqref{E:comp} and the continuity of $\prob_a(m(Y) > k)$ then implies that
$$
\prob_a(m(Y) > k^*) = \bar{\mathfrak{s}}_a(s_{a, k}, \infty) = \mu_+(s_{a, k}, \infty).
$$
Combining this with \eqref{E:full} proves that $\mu_+ = \bar{\mathfrak{s}}_a$. Equation \eqref{E:q-ak} follows by using that $\mu_+ = \bar{\mathfrak{s}}_a$ to simplify equation \eqref{E:s-a-k}.
\end{proof}

\begin{proof}[Proof of Proposition \ref{P:integral-equality}]
Fix $a$ so that the conclusion of Lemma \ref{L:equal-measure} holds, and let $k > a$. By Theorem \ref{T:max-height} and the ordering on minimal flux paths, we can write 
\begin{align*}
\nonumber\prob(m(Y) > k | \; m(Y) > a)  &= \frac{1}{\sqrt{1 - a^2}} \lf[\int \indic(s_{a, m} > s_{a, k} ) d\nu(m)  + \int \indic(s_{a, m} = s_{a, k}, m > k) d\nu(m) \rg].
\end{align*}
By Lemma \ref{L:equal-measure}, we can rewrite the first integral above in terms of $\mathfrak{s}_a$, and then in terms of $d\mu_+$. This gives that
$$
\frac{1}{\sqrt{1 - a^2}}\int \indic(s_{a, m} > s_{a, k} ) d\nu(m) = \int \indic(s > s_{a, k}) d\mathfrak{s}_a(s) = \int \indic(s > s_{a, k})\frac{s}2d\mu_+(s).
$$
We can also rewrite the second integral using \eqref{E:q-ak}. This implies that
$$
\prob(m(Y) > k | \; m(Y) > a) = \int \indic(s > s_{a, k})\frac{s}2d\mu_+(s) + \frac{s_{a, k}}2 q_{a, k} \mu_+(s_{a, k}).
$$
By Lemma \ref{L:expect-2}, we can recognize $\frac{s}{2}$ as the Radon-Nikodym derivative of the size-biased random variable $\hat{\mu}_+$ with respect to $\mu_+$, proving that
$$
\frac{\sqrt{1 - k^2}}{\sqrt{1 - a^2}} = \prob(m(Y) > k | m(Y) > a) = \hat{\mu}_+(s_{a, k}, \infty) + q_{a, k}\hat{\mu}_+(s_{a, k})
$$
for almost every $a < k \in [0, 1)$. Now, Lemma \ref{L:bias-unbias} allows us to conclude that 
\begin{align*}
\prob_a(m(Y) > k) &= S(s_{a, k}, 1 - q_{a, k}) = S\lf(\hat{S}^{-1}\lf(\frac{\sqrt{1- k^2}}{\sqrt{1- a^2}}\rg)\rg)
\end{align*}
for almost every $a < k \in [0, 1)$. Combining this with the symmetry of $Y$ (Proposition \ref{P:Y-symmetries}) and Theorem \ref{T:max-height} implies that
\[
\sqrt{1 - k^2} = \prob(m(Y) > k) = \int_{0}^1 \prob_a(m(Y) > k) da = 1 - k + \int_0^k r_{\mu_+}\lf(\frac{\sqrt{1- k^2}}{\sqrt{1- a^2}}\rg) da. \qedhere
\]
\end{proof}

\section{The weak trajectory limit}
\label{S:transform}

In this section we show that the integral transform in Proposition \ref{P:integral-equality} determines the local speed distribution. This will allow us to immediately conclude that the weak limit of the trajectory random variables $Y_n$ is the Archimedean path. We begin with two basic lemmas about size-bias ratio functions.

\begin{lemma} 
\label{L:loc-lip}
For any probability distribution $\nu$ on $(0, \infty)$ with finite first moment $m$, the size-bias ratio function $r_\nu$ is locally Lipschitz on $[0, 1)$ and continuous on $[0, 1]$.
\end{lemma}

\begin{proof}
Let $\pi_1(x, y) = x.$ By calculus, when $x \in (0, 1)$ we have that
$$
\del_+ r_\nu(x) = \lim_{h \to 0^+} \frac{m}{\pi_1(\hat{S}^{-1}(x + h))} \qquad \mathand \qquad
\del_- r_\nu(x) = \lim_{h \to 0^+} \frac{m}{\pi_1(\hat{S}^{-1}(x - h))}.
$$
The first equation also holds at $x = 0$. As $\pi_1(\hat{S}^{-1}(y))$ is a decreasing function of $y$, and strictly positive for all $y \in [0, 1)$, this shows that $r_\nu$ is locally Lipschitz on $[0, 1)$. It is straightforward to check that $r_\nu$ is continuous at $1$.
\end{proof}

\begin{lemma}
\label{L:rX-det}
Suppose that $\nu_1$ and $\nu_2$ are probability measures on $(0, \infty)$ with the same first moment, such that $r_{\nu_1} = r_{\nu_2}$. Then $\nu_1 = \nu_2$.
\end{lemma}

\begin{proof} We prove the contrapositive.
Let $\nu_1$ and $\nu_2$ be measures with the same first moment, and suppose that $\nu_1 \ne \nu_2$. Since $\nu_1$ and $\nu_2$ have the same first moment, $\hat{\nu}_1 \ne \hat{\nu}_2$, so  for some value of $(y, q) \in (0, \infty) \X [0, 1]$, we have that $\hat{S}_{\nu_1} (y, q) \ne \hat{S}_{\nu_2} (y, q)$. Since the functions $\hat{S}_{\nu_1}(y, \cdot)$ and $\hat{S}_{\nu_2}(y, \cdot)$ are linear for any fixed value of $y$, this implies that $\hat{S}_{\nu_1} (y, r) \ne \hat{S}_{\nu_2} (y, r)$ for some $r \in \{0, 1\}$. Moreover, since 
$$
\hat{S}_{\nu_i}(y, 1) = \sup \{\hat{S}_{\nu_i}(z, 0) : z > y \}, \qquad i = 1, 2,
$$
we can conclude that $\hat{S}_{\nu_1} (z, 0) \ne \hat{S}_{\nu_2} (z, 0)$ for some $z \in (0, \infty)$. Without loss of generality, assume that $a_1 := \hat{S}_{\nu_1} (z, 0) > a_2 :=\hat{S}_{\nu_2} (z, 0)$. Therefore using the notation of the previous lemma, letting $b = (a_1 + a_2)/2$, we have that
$$
\lim_{h \to 0^+} \pi_1(\hat{S}_{\nu_1}^{-1}(b + h)) \ge z, \qquad \text{whereas} \qquad \lim_{h \to 0^+} \pi_1(\hat{S}_{\nu_1}^{-1}(b + h)) < z.
$$
Since $b \in (0, 1)$, the derivative computation in the previous lemma combined with the fact that $\nu_1$ and $\nu_2$ have the same first moment implies that $r_{\nu_1} \ne r_{\nu_2}$.
\end{proof}
Now let $\scrX$ be the space of continuous functions from $[0, 1] \to \real$ that are locally Lipschitz on $[0, 1)$. Define an integral transform $H$ on $\scrX$ by 
\begin{equation*}
H(r)(k) = \int_0^k r \lf(\frac{\sqrt{1 - k^2}}{\sqrt{1- x^2}}\rg) dx,
\end{equation*}
for $k \in [0, 1]$. By Lemma \ref{L:loc-lip}, any size-bias ratio function is in $\scrX$, so if the integral transform $H$ is injective on $\scrX$, then $H(r_{\mu_+})$ determines $r_{\mu_+}$.

\begin{lemma}
\label{L:injective-transform-lip}
The integral transform $H$ is injective on $\scrX$.
\end{lemma}

\begin{proof}
Let $r \in \scrX, r \ne 0$. We will show that $H(r) \ne 0$. Without loss of generality, we may assume that 
$$
\max_{x \in [0, 1]} r(x) = \delta > 0.
$$
Letting $u = \frac{\sqrt{1 - k^2}}{\sqrt{1- x^2}}$ and $y = \sqrt{1 -k^2}$, we have that
\begin{equation*}
H(r)(\sqrt{1-y^2}) = \int_y^1 r(u) \frac{y^2}{u^2\sqrt{u^2 - y^2}}du.
\end{equation*}
For $y \in (0, 1]$, we have that
\begin{equation}
\label{E:H-bar}
\close{H}(r)(y) := \frac{H(r)(\sqrt{1 - y^2})}{y^2} = \int_y^1 \frac{r(u)}{u^2\sqrt{u^2 - y^2}}du.
\end{equation}
It suffices to show that $\close{H}(r)(y) \ne 0$ for some $y \in (0, 1]$. 
Observe that if $r(1) > 0$, then there would exist a $\ga > 0$ such that $r(x) > 0$ for all $x > \ga$, so $\close{H}(r)(\ga) > 0$. Also, if $r(0) > 0$, then $\close{H}(r)(y) \to \infty$ as $y \to 0$. Therefore there must exist $y \in (0,1)$ be such that  $r(y) = \de$.

 Since $r$ is locally Lipschitz on $[0, 1)$, we can find $k, \ep_0 > 0$ such that $r(y + x) \ge \delta - kx$ for all $x \in [0, \ep_0]$. 
Therefore for $\ep < \ep_0$ we have
\begin{align}
\label{E:de-k-bd}
\int_y^{y+\ep}\frac{r(u)}{u^2\sqrt{u^2 - y^2}}du \ge \int_y^{y+\ep}\frac{(\delta - k \ep)}{u^2\sqrt{u^2 - y^2}}du =  \frac{(\delta - k \ep)\sqrt{\ep^2 + 2y\ep}}{y^2(y + \ep)}.
\end{align}

Now consider the difference between $\close{H}(r)(y)$ and $\close{H}(r)(y + \ep)$. We have 
\begin{align}
\nonumber \close{H}&(r)(y) - \close{H}(r)(y + \ep) \\
\nonumber &= \int_y^{y+\ep}\frac{r(u)}{u^2\sqrt{u^2 - y^2}}du + \int_{y+ \ep}^1 \frac{r(u)}{u^2\sqrt{u^2 - y^2}}du - \int_{y+ \ep}^1 \frac{r(u)}{u^2\sqrt{u^2 - (y+ \ep)^2}}du  \\
\nonumber &\ge \int_y^{y+\ep}\frac{r(u)}{u^2\sqrt{u^2 - y^2}}du + \delta \int_{y + \ep}^1 \lf[\frac{1}{u^2\sqrt{u^2 - y^2}} - \frac{1}{u^2\sqrt{u^2 - (y + \ep)^2}}\rg]du \\
\label{E:int-alm} &\ge \frac{(\de - k \ep) \sqrt{\ep^2 + 2y\ep}}{y^2(y + \ep)} + \de \lf[\frac{(y + \ep)^2\sqrt{1-y^2} -y^2\sqrt{1-(y + \ep)^2}- (y + \ep)\sqrt{\ep^2 + 2y\ep}}{(y + \ep)^2y^2}\rg] .
\end{align}
In above calculation the first inequality comes from the fact that $r(x) \le \de$ for all $x$, and the observation that
$$
\frac{1}{u^2\sqrt{u^2 - y^2}} < \frac{1}{u^2\sqrt{u^2 - (y + \ep)^2}}
$$ 
for all $u \in (y + \ep, 1)$. The second equality follows by integration and plugging in the bound in \eqref{E:de-k-bd}. Now expanding in $\ep$ about $\ep = 0$, we get that \eqref{E:int-alm} is equal to \begin{align*}
\frac{(2-y^2)\de\ep}{y^3\sqrt{1-y^2}} + O(\ep^{3/2}).
\end{align*}
This is strictly greater than 0 for small enough $\ep$, so
$
\close{H}(r)(y) - \close{H}(r)(y + \ep) \ne 0
$
for such $\ep$. Hence $\close{H}(r)(x) \ne 0$ for some $x \in [0, 1)$. 
\end{proof}

\begin{prop}
\label{P:arcsin-en} Let $\mathfrak{arc}$ be the arcsine distribution on $[-\pi, \pi]$, and let $\mathfrak{arc}_+$ be the pushforward of $\mathfrak{arc}$ under the map $x \mapsto |x|$. Then for every $k \in [0, 1]$, we have that
\begin{equation}
\label{E:integral-H-2}
 1 - k + \int_0^k r_{\mathfrak{arc}_+} \lf(\frac{\sqrt{1 - k^2}}{\sqrt{1- x^2}}\rg) dx = \sqrt{1-k^2}.
\end{equation}
\end{prop}

\begin{proof} The distribution $\mathfrak{arc}$ has density $(\pi \sqrt{\pi^2 - x^2})^{-1}$ on $[-\pi, \pi]$. From this, we can calculate that
\begin{align*}
r_{\mathfrak{arc}_+}(y) = 1 - \frac{2}{\pi}\arcsin(\sqrt{1-y^2}).
\end{align*}
We now use the connection between the arcsine distribution and the Archimedean measure $\mathfrak{Arch}$ to help evaluate the integral in \eqref{E:integral-H-2}. Let $(X_1, X_2) \sim \mathfrak{Arch}$ and let 
$$
\scrA(t) = X_1 \cos (\pi t) + X_2 \sin(\pi t)
$$
be the Archimedean path. Writing $\prob(m(\scrA) > k) = \prob(X_1^2 + X_2^2 > k^2)$ in both polar and Cartesian coordinates, we get that
$$
\int_k^1 \frac{rdr}{\sqrt{1 - r^2}} = 4 \int_{0}^1 \int_{\sqrt{k^2 - y^2} \vee 0}^{\sqrt{1 - y^2}} \frac{dxdy}{2 \pi \sqrt{1 - x^2 - y^2}}.
$$
The left hand side can easily be evaluated as $\sqrt{1 - k^2}$. The right hand side is equal to \begin{align*}
(1 - k) + \int_{0}^k \frac{2}{\pi} \lf(\frac{\pi}{2} - \arcsin\lf(\frac{\sqrt{k^2 - y^2}}{\sqrt{1- y^2}}\rg)\rg)dy
= (1- k) + \int_{0}^k r_\mathfrak{arc}\lf(\frac{\sqrt{1 - k^2}}{\sqrt{1- y^2}}\rg) dy. \qquad \qedhere
\end{align*}

\end{proof}

We can now prove Theorem \ref{T:main-2}, and in turn use that to prove Theorem \ref{T:weak-limit}.

\begin{proof}[Proof of Theorem \ref{T:main-2}]
By Proposition \ref{P:integral-equality}, Lemma \ref{L:injective-transform-lip} and Proposition \ref{P:arcsin-en}, we have that  $r_{\mu_+} = r_{\mathfrak{arc}_+}$. Moreover, the first moment of $\mathfrak{arc}_+$ is $2$. By Lemma \ref{L:expect-2}, this matches the first moment of $\mu_+$. Therefore by Lemma \ref{L:rX-det}, $\mu_+ = \mathfrak{arc}_+$. Finally, symmetry of both measures implies that $\mu = \mathfrak{arc}$.
\end{proof}

\begin{proof}[Proof of Theorem \ref{T:weak-limit}] Fix $m > 0$.
Let $g_m(t) = m \sin (\pi t)$, and let $g_m^{-1}$ be the inverse of $g_m$ on the interval $[0, 1/2]$. Let 
$$
s_m(t) = \frac{g_m'(t)}{\sqrt{1 - g_m^2(t)}}
$$
be the local speed of $g_m$. We can calculate that
\begin{align}
\nonumber J(g_m) = 2J(g_m; [0, 1/2]) &= \int_0^{1/2} D_\mu(s_m(t)) \sqrt{1 - g_m^2(t)}dt \\
\label{E:s-h}
&= \int_0^m \frac{D_\mu(s_m(g_m^{-1}(x))}{s_m(g_m^{-1}(x))}dx.
\end{align}
Here we have made the substitution $x = g_m(t)$ to go from the first to the second line. Now using Theorem \ref{T:main-2}, we can calculate
$$
D_\mu(c) = D_\mathfrak{arc}(c) = \frac{2}{\pi} \lf(c \arcsin \lf(\frac{c}{\pi}\rg) + \sqrt{\pi^2 - c^2}\rg).
$$
From here we can use that $s_m(g_m^{-1}(x)) = \frac{\pi \sqrt{m^2 - x^2}}{\sqrt{1 - x^2}}$ to compute that
\begin{equation*}
J(g_m) = \int_0^m \frac{2}{\pi} \lf[ \frac{\sqrt{1-m^2}}{\sqrt{m^2 - x^2}} + \arcsin\lf(\frac{\sqrt{m^2 - x^2}}{\sqrt{1- x^2}}\rg)\rg]dx
\end{equation*}
The first part of this integral can be easily evaluated, and the second part can be evaluated by comparing with the integral in the final line of the proof of Proposition \ref{P:arcsin-en}. Putting this all together yields that $J(g_m) = 1$, as desired. Therefore by Theorem \ref{T:Y-unique}, we can write
$$
Y(t) = \sqrt{1 - V^2}\sin(\pi t + 2 \pi U),
$$
where $U$ and $V$ are independent uniform random variables on $[0, 1]$. This is the Archimedean path!
\end{proof}

\subsection{The empirical distribution of trajectories}
\label{SS:slightly}

By a compactness argument, we can immediately prove a slightly stronger version of Theorem \ref{T:weak-limit}. This will allow us to conclude Theorem \ref{T:subnetwork}. This theorem will also be necessary for establishing the stronger limits in Sections \ref{S:strong-limit} and \ref{S:geom-limit}.

\medskip

For a Polish space $S$, let $\scrM(S)$ be the space of probability measures on $S$ with the topology of weak convergence. Note that $\scrM(S)$ is itself a Polish space.
Recall the notation $\sig_G$ introduced in Section \ref{S:intro}. For a fixed sorting network $\sig$, let $\nu_\sig \in \scrM(D)$ be uniform measure on the set $\{\sig_G(i, \cdot)\}_{i \in \{1, \dots, n\}}$. Now let $\Om_n$ be the space of all $n$-element sorting networks, and define 
$$
\nu_n = \frac{1}{\card{\Om_n}} \sum_{\sig \in \Om_n} \de(\nu_\sig).
$$
For each $n$, $\nu_n \in \scrM(\scrM(\scrD))$. We now extend Theorem \ref{T:weak-limit} to give a limit theorem for the sequence $\nu_n$. 

\begin{proof}
By Remark 2.4 from \cite{dauvergne1}, the sequence $\nu_n$ is precompact. For any subsequential limit $\nu$ of $\nu_n$, Theorem \ref{T:weak-limit} implies that $\nu$-almost every $\rho$ is supported on curves of the form $a\sin(\pi t) + b \cos(\pi t)$. Hence if $Z$ is a random path with law $\rho$, then
$$
Z(t) = X_1 \cos(\pi t) + X_2 \sin(\pi t)
$$
for some random variables $(X_1, X_2)$. Moreover, $Z(t)$ is uniform for every $t$ since each measure $\nu_n$ is supported on the set $\{\nu_\sig\}_{\sig \in \Om_n}$.  Therefore $(X_1, X_2) \sim \mathfrak{Arch}$, so $Z$ is the Archimedean path, and hence $\nu = \de_{\prob_\scrA}$.
\end{proof}

\begin{proof}[Proof of Theorem \ref{T:subnetwork}]
We use the notation introduced in the paragraphs preceding Theorem \ref{T:subnetwork}. Fix $m \in \mathbb N$. First, by Theorem \ref{T:weak-limit'}, we have
\begin{equation}
\label{E:product-version}
\frac{1}{|\Om_n|} \sum_{\sig \in \Om_n} \nu_\sig^m \to \prob_\scrA^m \qquad \mathas \quad n \to \infty.
\end{equation}
Here $\nu^m$ refers to the $m$-fold product measure $\nu \X \nu \X \dots \X \nu$ and the convergence in \eqref{E:product-version} is simply weak convergence in $\scrM(\scrD)$. Next, let $\tilde \nu_\sig^m$ denote the product measure $\nu_\sig^m$ conditioned to lie in the set 
$$
T := \{(x_1, \dots, x_m) \in \scrD^m : x_i \ne x_j \text{ for all } i, j\}
$$
Then for every $n, m$ and every $\sig \in S_n$, we have $\nu_\sig^m(T) \ge 1 - m^2/n$, and so \eqref{E:product-version} also holds with $\tilde  \nu_\sig^m$ in place of  $\nu_\sig^m$. Now, $\tau^n_m$ is equal in law to the pushforward of $\frac{1}{|\Om_n|} \sum_{\sig \in \Om_n} \tilde \nu_\sig^m$ under a certain `subnetwork map' $F$, which we describe formally in the next paragraph. Therefore the limit of $\tau^n_m$ should be the pushforward of $\prob^m_\scrA$ under $F$, which we can check to be $\tau_m$. The details are as follows.
 
 \medskip
 
Let $R_m$ be the set of  $k$-tuples $(x_1, \dots, x_m) \in \scrD^m$ where 
\begin{itemize}[nosep]
	\item For every pair $i \ne j$, there exists a unique time $t_{i, j} \in [0, 1]$ for which either $x_i(t) > x_j(t)$ for $t < t_{i, j}$ and $x_i(t) > x_j(t)$ for $t > t_{i, j}$ or else $x_i(t) > x_j(t)$ for $t > t_{i, j}$ and $x_i(t) > x_j(t)$ for $t < t_{i, j}$. 
	Moreover, $t_{i, j} \ne 0$.
	\item For all $(i, j) \ne (i', j')$ we have $t_{i, j} \ne t_{i', j'}$.
\end{itemize}
The set $R_m$ should be thought of as a generalization of the space of all $m$-element sorting networks.
The subnetwork map $F$ takes an element $x \in R_m$ and maps it to an $m$-element sorting network, as follows: for $t \in [0, 1]$, let $\sig_t$ be the relative order of $(x_1(t), \dots, x_m(t))$ to $(x_1(0), \dots, x_m(0))$. This is well-defined outside of the points $t_{i, j}$; we define the map $t \mapsto \sig_t$ to be right continuous at these points. Letting  $(t_0, t_1, \dots, t_{m \choose 2})$ be the order statistics of the points $0, t_{i, j}, i \ne j \in \{1, \dots, m\}$, setting $F(x) = (\sig_{t_0}, \sig_{t_1}, \dots, \sig_{t_{m \choose 2}})$ gives an $m$-element sorting network. 

\medskip

The function $F$ is continuous on the subset of $R^*_m$ of $R_m$ consisting of continuous paths. Moreover, $\prob^m_\scrA(R^*_m) = 1$, so by the continuous mapping theorem, $\tau^n_m$ converges to the pushforward of $\scrP^m_\scrA$ under $F$. Since $\scrP^m_\scrA$ is supported on sine curves of period $2$, this is simply the geometric sorting network $\tau_m$.
\end{proof}

\begin{section}{The strong sine curve limit}
\label{S:strong-limit}
Theorem \ref{T:weak-limit} shows that for any $\ep > 0$, with high probability $(1- \ep)n$ particle trajectories in a random sorting network are close to sine curves. In this section, we extend this result to all particle trajectories, thus proving Theorem \ref{T:sine-curves}. By combining Theorem \ref{T:sine-curves} and Theorem \ref{T:weak-limit}, we also prove Theorems \ref{T:matrices} and \ref{T:unif-rotation}.

\medskip

The idea behind the proof is as follows. By Theorem \ref{T:weak-limit}, we know that most trajectories in a typical large-$n$ sorting network are close to sine curves. Since any two trajectories must cross exactly once, this restricts the type of behaviour that the remaining trajectories can have. Specifically, this forces all other trajectories to be either sine curves themselves, or to spend a lot of time at the edge of the sorting network. We can eliminate this second case by using the octagon bound from \cite{angel2007random}. To state this bound, let $A_{n, \ga}$ be the event where
\begin{align*}
\lf|\sig^n_G(i, t) - \sig^n_G(i, 0)\rg| < 2\sqrt{2t - t^2} + \ga \;\; &\mathand \;\; \lf|\sig^n_G(i, t) - \sig^n_G(i, 1)\rg| < 2\sqrt{1 - t^2} + \ga \\
&\text{for all } \;\; t \in [0, 1], i \in \{1, \ddd, n\}.
\end{align*}

\begin{theorem}[Octagon bound, \cite{angel2007random}]
\label{T:octagon}
For any $\ga > 0$, we have that 
$$
\lim_{n \to \infty} \prob(A_{n, \ga}) = 1.
$$
\end{theorem}

Now define 
$$
L^n_{i}(\de) = \scrL \{ t: |\sig^n_G(i, t)| \ge 1 - \de\},
$$
where $\scrL$ is Lebesgue measure on $[0, 1]$. This is the amount of time that particle $i$ spends within $\de$ of the edge in the random sorting network $\sig^n$.
\begin{lemma}
\label{L:edge-time}
For every $\ep > 0$, we have that
$$
\lim_{n \to \infty} \prob \lf( \max_{i \in [1, n] } L^n_i(\ep^2/16) > \ep \rg) =0.
$$
\end{lemma}

\begin{proof}
Fix $\ep > 0, n \in \nat$, and let $i$ be such that $|\sig_G(i, 0)| \ge 1 - \ep^2/16$. On the event $A_{n, \ep^4/64}$, a simple calculation shows that
$$
-1 + \ep^2/16 < \sig_G(i, t) < 1 - \ep^2/16 \qquad \mathforall t \in [\ep/2, 1 - \ep/2].
$$
Therefore by Theorem \ref{T:octagon}, we have that
\begin{equation}
\label{E:time-shift}
\lim_{n \to \infty} \prob \bigg( \max \lf\{L^n_i(\ep^2/16)  : \; i \in [1, n], \;|\sig^n_G(i, 0)| \ge 1 - \ep^2/16 \rg\} > \ep \bigg) = 0.
\end{equation}
By time stationarity of random sorting networks (Theorem \ref{T:time-stat}), for each $n$ the above probability is greater than or equal to 
$$
\ep \prob \lf( \max_{i \in [1, n] } L^n_i(\ep^2/16) > \ep \rg).
$$
Therefore \eqref{E:time-shift} implies the lemma.
\end{proof}

Recall that $\nu_\sig$ is uniform measure on the trajectories of a sorting network $\sig$, and that $\prob_\scrA$ is the law of the Archimedean path.
To prove Theorem \ref{T:sine-curves} from here, we must show that if $\nu_\sig$ is close to $\prob_\scrA$ in the weak topology, then for any particle $i$, either $\sig_G(i, \cdot)$  spends a lot of time at the edge of the sorting network, or else $\sig_G(i, \cdot)$ is close to a sine curve. 

\medskip

Recall that $\scrD$ is the closure of all sorting network trajectories in the uniform norm. To metrize weak convergence on the space of probability measures on $\scrD$, we use the L\'evy-Prokhorov metric $d_{LP}$. For a set $A \sset \scrD$, define 
\begin{equation*}
A^\ep = \{ f \in \scrD : ||f - g||_u < \ep \;\; \text{for some} \; g \in A \}.
\end{equation*}
Here and throughout the remaining proofs, $||\cdot||_u$ is the uniform norm. For two probability measures $\nu_1, \nu_2$ on $\scrD$, define
$$
d_{LP}(\nu_1, \nu_2) = \inf \big\{ \ep > 0 : \nu_1(A) \le \nu_2(A^\ep) + \ep, \nu_2(A) \le \nu_1(A^\ep) + \ep \text{  for all Borel sets } A \sset \scrD \big\}.
$$
We now prove two lemmas characterizing particle behaviour when $\nu_\sig$ is close to ${\prob_\scrA}$. The first gives conditions under which a particle spends time close to the edge of a sorting network. Let $a_{n, i} = 2i/n - 1$, and let
$$
h_{a_{n, i}, 1}(t) = \sin(\arcsin(a_{n, i}) + \pi t) \qquad \mathand \qquad h_{a_{n, i}, -1}(t) = \sin(\arcsin(a_{n, i}) - \pi t).
$$
This agrees with the notation $h_{a, k}$ from Theorem \ref{T:unique-2}. Recalling the definition of the Archimedean measure $\mathfrak{Arch}$ from Section \ref{S:intro}, for $\ep \in (0, 1]$ define
$$
L(\ep) =\mathfrak{Arch}( r \ge 1 - \ep, \theta \in [2\pi x, 2\pi (x + \ep)]) =\ep\sqrt{2\ep - \ep^2},
$$
where $(r, \theta)$ are polar coordinates. Note that $L(\ep)$ is independent of $x$ by the rotational symmetry of $\mathfrak{Arch}$.

\begin{lemma}
\label{L:edge-cond}
Let $\ga \in (0, 2]$ and let $\ep \in (0, \ga^2/100)$. 
Suppose that for a fixed $n$-element sorting network $\sig$ and a particle $i \in \{1, \ddd, n \}$, that $d_{LP}(\nu_\sig, {\prob_\scrA}) < L(\ep)/2$, and either
$$
\max_{t \in [0, 1]} [\sig_G(i, t)  - h_{a_{n,i}, 1}(t)]\ge \ga \quad \text{or} \quad \max_{t \in [0, 1]} [h_{a_{n,i}, -1}(t) - \sig_G(i, t)] \ge \ga.
$$
Then 
$$
\scrL\lf\{ t \in [0, 1] : |\sig_G(i, t)| \ge 1- 3\ep \rg\} \ge \ga/6.
$$
\end{lemma}

\begin{proof}
Without loss of generality, we can assume that there exists an $s \in [0, \arccos(a_{n, i})/\pi]$ such that
\begin{equation}
\label{E:sig-ga}
\sig_G(i, s)  - h_{a_{n,i}, 1}(s) \ge \ga.
\end{equation}
This is case where particle $i$ passes $h_{a_{n,i}, 1}$ while $h_{a_{n,i}, 1}$ is increasing.
The other cases (where particle $i$ passes $h_{a_{n,i}, 1}$ while $h_{a_{n,i}, 1}$ is decreasing, and the two cases where particle $i$ passes $h_{a_{n,i}, -1}$) follow by symmetric arguments. We can only be in this case if $a_{n, i} \le 1 - \ga$, so $\arccos(a_{n, i}) > \ga$.

\medskip

Since $d_{LP}(\nu_\sig, {\prob_\scrA}) < L(\ep)/2$, for every $x \in [0, 1]$, there exists a particle $j(x)$ such that $$
||\sig_G(j(x), t) - r\sin(\pi t + 2 \pi (x + \beta)) ||_u < L(\ep)/2
$$
for some $r \ge 1 - \ep$ and $\beta \in [0, \ep]$. Now since
$$
||r\sin(\pi t + 2 \pi (x + \beta)) - \sin(\pi t + 2 \pi x)||_u \le 2 \ep,
$$
this implies that 
\begin{equation}
\label{E:sig-jx}
||\sig_G(j(x), t) -  \sin(\pi t + 2 \pi x)||_u < L(\ep)/2 + 2 \ep
\end{equation}
for all $x \in [0, 1]$. Define $\de = \ga - \frac{L(\ep)}2 - 2 \ep$. Note that $\de$ is positive. For all $\al \in [0, \de]$, the inequality \eqref{E:sig-ga} implies that
\begin{align}
\label{E:cond-1}
\sig_G(i, s)- \sin(\arcsin(a_{n, i}) + \al + \pi s) > \ga - \de = L(\ep)/2 + 2 \ep.
\end{align}
Now define 
$j_{\al} = j((\arcsin(a_{n, i}) + \al)/(2\pi)).$ Combining \eqref{E:cond-1} with \eqref{E:sig-jx}, we have that
\begin{align}
\nonumber
\sig_G(i, s) > \sig_G\lf(j_\al, s\rg) \qquad \text{for } \al \in [0, \de].
\end{align}
 Moreover, if $1 - \cos(\al) \ge L(\ep)/2 + 2 \ep$ and $\al \le \arccos(a_{n, i})$, then 
 $$
 \sig_G(i, 0) = a_{n, i} \le  \sin(\arcsin(a_{n, i}) + \al) - (L(\ep)/2 + 2 \ep),
 $$
 and so 
 $$
 \sig_G(i, 0) < \sig_G\lf(j_\al, 0\rg).
 $$
 Now, noting that $\de < \ga < \arccos(a_{n, i})$, for every 
 $$\al \in [\arccos(1 - L(\ep)/2 - 2 \ep), \arccos(a_{n, i})] \cap [0, \de] = [\arccos(1 - L(\ep)/2 - 2 \ep), \de],
 $$
 the particles $i$ and $j_\al$ must cross during the interval $[0, s]$. Therefore
\begin{equation}
\label{E:sig-G}
\sig_G(i, t) > \sig_G\lf(j_\al, t\rg)\qquad \mathforall t > s, \;\; \al \in [\arccos(1 - L(\ep)/2 - 2 \ep), \de].
\end{equation}
 We now show that this forces the particle $i$ to spend a large amount of time close to the edge of the sorting network.  For all $\al \in [0, \de]$, there must be some time $t_\al \in [0, 1]$ such that 
$$
\sin(\pi t_\al +\arcsin(a_{n, i}) + \al) = 1.
$$
The time  $t_\al \notin [0, s]$ since for every $\al \in [0, \de]$, we have that
$$
||\sin(\pi t +\arcsin(a_{n, i}) + \al) - h_{a_{n,i}, 1}(t)||_u < \ga, \quad \mathand \quad h_{a_{n,i}, 1}(t) \le h_{a_{n,i}, 1}(s) \le 1 - \ga
$$
for all $t \in [0, s]$.
We have used \eqref{E:sig-ga} to get the second inequality on the right side above. Therefore for all $\al \in [\arccos(1 - L(\ep)/2 - 2 \ep), \de]$, \eqref{E:sig-G} implies that
$$
\sig_G(i, t_\al) > \sig(j_\al, t_\al) \ge 1 - L(\ep)/2 - 2\ep \ge 1 - 3\ep.
$$ 
 Using the fact that $\arccos(1 - x) \le 2\sqrt{x}$ for $x \in [0, \pi/2)$, we have that
 \[
\scrL \big\{t_\al : \al \in [\arccos(1 - L(\ep)/2 - 2 \ep), \de] \big\} \ge \frac{\de - 2\sqrt{L(\ep)/2 + 2 \ep}}{\pi} \ge \frac{\ga - 3\ep - 2\sqrt{3\ep}}{\pi} \ge \frac{\ga}{6}.
\]
Here the final bound follows from the fact that $\ep < \ga^2/100$.
\end{proof}

The second lemma shows that if a curve $\sig_G(i, \cdot)$ stays close to the region between the curves $h_{a_{n,i}, -1}$ and $h_{a_{n,i}, 1}$, then it must be close to a sine curve. 

\begin{lemma}
\label{L:sine-close}
Let $\sig$ be an $n$-element sorting network, $i \in \{1, \ddd, n \}$, and $\ga \in (0, 1)$. Suppose that
$$
\max_{t \in [0, 1]} [\sig_G(i, t)  - h_{a_{n,i}, 1}(t)] < \ga \quad \text{and} \quad \max_{t \in [0, 1]} [h_{a_{n,i}, -1}(t) - \sig_G(i, t)] < \ga,
$$
and that $d_{LP}(\nu_\sig, {\prob_\scrA}) < \frac{\ga^4}{128}$. Then there exist constants $a \in [0, 1]$ and $\theta \in [0, 2\pi]$ such that 
$$
||\sig_G(i, t) - a\sin(\pi t + \theta) ||_u < 2\ga + 2/n.
$$
\end{lemma}

\begin{proof} 
Suppose first that for some $s \in [0,1]$, that 
\begin{equation}
\label{E:sig-GG}
\sig_G(i, s) \in [h_{a_{n,i}, -1}(s) + \ga, h_{a_{n,i}, 1}(s) - \ga]. 
\end{equation}
Observe that $s \in [\arcsin(\ga)/\pi, 1- \arcsin(\ga)/\pi]$, since otherwise the above interval 
is necessarily empty. By \eqref{E:sig-GG}, we can find a point $(x_0, y_0) \in B(0, 1 - \ga)$ such that 
$$
x_0 = \sig_G(i, 0) \qquad \mathand \qquad x_0\cos(\pi s) + y_0 \sin(\pi s) = \sig_G(i, s).
$$
Since $s \le 1- \arcsin(\ga)/\pi$, we can find another point 
$
(x_1, y_1) \in
B((x_0, y_0), \ga) \sset B(0, 1)
$
such that 
 \begin{equation}
 \label{E:x1-x0}
x_1 > x_0 + \ga^2/3 \quad \mathand \quad x_1\cos(\pi s) + y_1\sin(\pi s) > x_0\cos(\pi s) + y_0\sin(\pi s) + \ga^2/3.
 \end{equation}
Now observe that
$$
\ga^4/64\le \inf \lf\{ \mathfrak{Arch}(B(x, \ga^2/6)) : x \in B(0, 1) \rg\}.
$$
Therefore since $d_{LP}(\nu_\sig, \prob_{\prob_\scrA}) < \ga^4/128$, there must be a point $(x_2, y_2) \in  B(0, 1) \cap B((x_1, y_1),\ga^2/6)$ and a particle $j \in \{1, \dots, n\}$ such that 
\begin{equation}
\label{E:G-jj}
\begin{split}
||\sig_G(j, t) - (x_1\cos(\pi t) &+ y_1\sin(\pi t))||_u \\
&\le ||\sig_G(j, t) - (x_2\cos(\pi t) + y_2\sin(\pi t))||_u + \ga^2/6 < \ga^2/3.
\end{split}
\end{equation}
By \eqref{E:x1-x0}, this implies that 
$$
\sig_G(j, 0) > \sig_G(i, 0) \qquad \mathand \qquad \sig_G(j, s) > \sig_G(i, s),
$$
and hence $\sig_G(j, t) > \sig_G(i, t)$ for all $t \in [0, s]$. Combining this with \eqref{E:G-jj} and the fact that $(x_1, y_1) \in B((x_0, y_0), \ga)$ implies that
$$
x_0\sin(\pi t) + y_0 \cos(\pi t) - \sig_G(i, t)  \le \ga^2/3 + \ga < 2\ga \qquad \mathforall t \in [0, s].
$$
Symmetric arguments give the same upper bound on this difference over the interval $[s, 1]$, and on the difference $\sig_G(i, t) - x_0\sin(\pi t) + y_0 \cos(\pi t)$ over the interval $[0, 1]$. This implies that
$$
||\sig_G(i, t) - [x_0\sin(\pi t) + y_0 \cos(\pi t)]||_u \le 2\ga.
$$
 Now suppose that there does not exist an $s \in [0, 1]$ such that $\sig_G(i, s) \in [h_{a_{n,i}, -1}(s) + \ga, h_{a_{n,i}, 1}(s) - \ga]$. Then either
$$
||\sig_G(i, \cdot) - h_{a_{n,i}, 1}(\cdot)||_u < 2\ga + \frac{2}n \qquad \mathor \qquad ||\sig_G(i, \cdot) - h_{a_{n,i}, -1}(\cdot)||_u < 2\ga + \frac{2}n.
$$
This can be seen by observing that the difference $h_{a_{n,i}, 1} - h_{a_{n,i}, -1}$ is unimodal and non-negative, so the path $\sig_G(i, \cdot)$ must stay close of one of those paths for its entire trajectory. The additive factor of $2/n$ comes from the fact that the path $\sig_G(i, \cdot)$ can make jumps of that size.

\medskip
As all the functions $h_{a_{n,i}, \pm 1}$, and $x_0\sin(\pi t) + y_0 \cos(\pi t)$ are of the form $a\sin(\pi t + \theta)$, for some $(a, \theta) \in [0, 1] \X[0, 2\pi]$, this proves the lemma.
\end{proof}

\begin{proof}[Proof of Theorem \ref{T:sine-curves}] Fix $\ga \in (0, 1)$, and recall from Section \ref{SS:slightly} that $\nu_n$ is uniform measure on the set $\{\nu_\sig\}_{\sig \in \Om_n}$. For small enough $\ep > 0$, Lemma \ref{L:edge-time} and Theorem \ref{T:weak-limit'} imply that 
$$
\prob \lf ( \max_{i \in [1, n]} \scrL\lf\{t: |\sig^n_G(i, t)| \ge 1 -3\ep \rg\} < \frac{\ga}{6}, \;\; d_{LP}(\nu_n, \prob_\scrA) < \frac{\ga^4}{128} \wedge \frac{L(\ep)}2 \rg) \to 1 \qquad \mathas n \to \infty.
$$
Combining Lemmas \ref{L:edge-cond} and \ref{L:sine-close}, this implies that there exist random variables $A_{n, i, \ga} \in [0, 1]$ and $\Theta_{n, i, \ga} \in [0, 2\pi]$ such that 
$$
P_{n, \ga} := \prob \lf(\max_{i \in [1, n]} ||\sig^n_G(i, t) - A_{n, i, \ga} \sin(\pi t + \Theta_{n, i, \ga})||_u < 2\ga + \frac{2}n \rg) \to 1 \qquad \mathas n \to \infty.
$$
As constructed, the random variables $A_{n, i, \ga}$  and $\Theta_{n, i, \ga}$ depend on $\ga$. To remove this dependence, let $\ga_n \to 0$ be a sequence such that $P_{n, \ga_n} \to 1$ as $n \to \infty$. Let $A_{n, i} = A_{n, i, \ga_n}$ and $\Theta_{n, i} = \Theta_{n, i, \ga_n}$, 
and define
\begin{equation}
\label{E:Bnga}
B_{n, \ga} = \lf\{ \max_{i \in [1, n]} ||\sig^n_G(i, t)- A_{n, i} \sin(\pi t + \Theta_{n, i})||_u < \ga \rg\},
\end{equation}
Then for any $\ga > 0$, we have that $\prob(B_{n, \ga}) \to 1$ as $n \to \infty$.
\end{proof}

We can also prove Theorem \ref{T:matrices} and Theorem \ref{T:unif-rotation}. For these we use the notation $B_{n, \ga}$ from \eqref{E:Bnga}.

\begin{proof}[Proof of Theorem \ref{T:matrices}] Fix $t \in [0, 1]$.
Since $\nu_{n} \to \de_{\prob_\scrA}$ (Theorem \ref{T:weak-limit'}), the random measure $\rho_t^n$ converges in probability to the law of $(X, X \cos (\pi t )+ Y \sin (\pi t))$, where $(X, Y) \sim \mathfrak{Arch}$. This law is simply $\mathfrak{Arch}_t$. 

\medskip

Now for a set $A \sset [-1, 1]^2$, let 
$$
A^\ga = \{x \in [-1, 1]^2: d(x, A) < \ga \}.
$$
On the event $B_{n, \ga}$, the support of the measure $\rho_t^n$ is contained in the set 
$
\supp(\mathfrak{Arch}_t)^\ga.
$
Moreover, since $\mathfrak{Arch}_t$ has a Lebesgue density that is bounded below, the weak convergence of $\rho^n_t$ to $\rho$ implies that with high probability,
$
\supp(\mathfrak{Arch}_t) \sset \supp(\rho_n^t)^\ga
$
for any $\ga > 0$. Therefore since $\prob(B_{n, \ga}) \to 1$ as $n \to \infty$ by Theorem \ref{T:sine-curves}, we have that
\[
\lim_{n \to \infty} d_H(\supp(\rho^n_t), \supp(\mathfrak{Arch}_t)) = 0 \qquad \text{ in probability. }  \quad \qedhere
\]
\end{proof}

\begin{proof}[Proof of Theorem \ref{T:unif-rotation}]
First observe that for any $a \in [0, 1]$, $t \in [0, 1/2]$, and $\theta \in [0, 2\pi]$, that
$$
e^{\pi i t} \lf[a\sin(\pi t + \theta) + ia\sin(\pi t + \pi/2 + \theta) \rg] = a \sin(\theta).
$$
Therefore on the event $B_{n, \ga}$, we have that
$$
\max_{j \in [1, n]} \max_{s, t \in [0, 1]} |Z^n_j(t) - Z^n_j(s)| \le 2\ga.
$$
Since $\prob(B_{n, \ga}) \to 1$ as $n \to \infty$ for any $\ga > 0$ by Theorem \ref{T:sine-curves}, this proves the theorem.
\end{proof}

\end{section}
\section{The geometric limit}
\label{S:geom-limit}

In this section, we use Theorem \ref{T:sine-curves} and Theorem \ref{T:weak-limit} to prove Theorem \ref{T:geom-limit}. Recall from Section \ref{S:intro} that $d_\infty(f, g)$ is the uniform norm between two $\real^n$-valued functions $f$ and $g$, where the pointwise distance is the $L^\infty$-distance. Recall also that $\bar{\sig}$ is the embedding of a sorting network $\sig$ into the $(n-2)$-dimensional sphere $\mathbb{S}^{n-2} \sset \real^n$.

\medskip

By Theorem \ref{T:sine-curves} and a change of variables,  there exist random vectors $\mathbf{X^n} = (X^n_1, \ddd, X^n_n)$ and $\mathbf{V^n} = (V^n_1, \ddd, V^n_n)$ such that $d_\infty(F_n, \close{\sig}^n)/n \to 0$ in probability, where 
$$
F_n(t) = \mathbf{X^n} \cos(\pi t) + \mathbf{V^n} \sin(\pi t) + \mathbf{c}, \qquad \text{where} \qquad 
\mathbf{c} = \lf( \frac{n + 1}2, \ddd, \frac{n + 1}2\rg).
$$
Moreover, we can assume that $X^n_i = i - (n + 1)/2$ and that $V^n_i =  \sig^n(i, N/2) -  (n + 1)/2$, as these changes only shift the curve  $\mathbf{X^n} \cos(\pi t) + \mathbf{V^n} \sin(\pi t)$ by $d_\infty$-distance $o(n)$ in probability. 

\medskip

The point $\mathbf{c}$ is the center of $\mathbb{S}^{n-2}$. It remains to show that we can shift the curve $F_n$ by $d_\infty$-distance $o(n)$ to obtain a great circle in $\mathbb{S}^{n-2}$. For this we need the following lemma.

\begin{lemma}
\label{L:dot-prod-small}
Let the vectors $\mathbf{X^n}$ and $\mathbf{V^n}$ be as above. Then 
$$
\lim_{n \to \infty} \frac{\lf\langle\mathbf{X^n} , \mathbf{V^n} \rg\rangle}{n^3} = 0 \qquad \text{in probability.}
$$
\end{lemma}

\begin{proof}
Let $I_n$ be a uniform random variable on $\{1, \dots, n\}$, independent of $\mathbf{X^n}$ and $\mathbf{V^n}$, and define
$$
(\tilde{X}^n, \tilde{V}^n) =  \lf(\frac{2X^n_{I_n}}n, \frac{2V^n_{I_n} }n\rg).
$$
We have that 
$$
(\tilde{X}^n, \tilde{V}^n) \eqd (Y_n(0) - 1/n, Y_n(1/2) - 1/n),
$$
where $Y_n$ is the trajectory random variable of $\sig^n$.  Therefore by Theorem \ref{T:weak-limit}, 
$$
(\tilde{X}^n, \tilde{V}^n) \cvgd (X, V) \sim \mathfrak{Arch}.
$$
By the bounded convergence theorem, this implies that $\expt \tilde{X}^n\tilde{V}^n \to \expt XV$ in probability. Observing that $n^3 \expt \tilde{X}^n\tilde{V}^n = 4\lf\langle\mathbf{X^n} , \mathbf{V^n} \rg\rangle$ and that $\expt XV = 0$ completes the proof.
\end{proof}

\begin{proof}[Proof of Theorem \ref{T:geom-limit}.] Fix $n \in \nat$. For ease of bounding errors, we assume that $n \ge 27$. Let $\sig$ be a fixed $n$-element sorting network. Our goal is to perturb the vector $\mathbf{V^n}$ to a new vector $\mathbf{W^n}$ so that the path
\begin{equation}
\label{E:C-nep}
C_{n}(t) := \mathbf{X^n} \cos(\pi t) + \mathbf{W^n} \sin(\pi t) + \mathbf{c}
\end{equation}
follows a great circle on $\mathbb{S}^{n-2}$, and so that $d_\infty(C_{n}, F_n)$ is small. In order to do this, we just need to find $\mathbf{W^n}$ such that
$$
\sum_{i=1}^n W^n_i = 0, \;\lf\langle\mathbf{X^n} , \mathbf{W^n} \rg\rangle = 0, \quad \mathand \quad ||W^n||_2^2 = (n^3 - n)/12.
$$
The first of these conditions guarantees that $C_n$ lies in the correct hyperplane $\mathbb{L}_n \sset \real^n$, and the second and third conditions guarantee that $C_n$ traces out a great circle on the sphere $\mathbb{S}^{n-2}$ within that hyperplane. We first perturb $\mathbf{V^n}$ to a vector $\mathbf{Z^n}$ satisfying the first two properties. Define the random variable
$$
A_n = \frac{\lf|\lf\langle\mathbf{X^n} , \mathbf{V^n} \rg\rangle\rg|}{n^3}.
$$
Without loss of generality, assume that $\lf\langle\mathbf{X^n} , \mathbf{V^n} \rg\rangle > 0$. 

\medskip

Choose any $i \ge 3(n + 1)/4$. Let $k > 0$, and decrease the value of $V^n_i$ by $kn$ and increase the value of $V^n_{n + 1 - i}$ by $kn$. Call this new vector $\mathbf{V^n_1}$. Note that $X^n_i \ge n/4$, and that $X^n_i = -X^n_{n + 1 - i}$. Therefore $\lf\langle\mathbf{X^n} , \mathbf{V^n_1} \rg\rangle \le\lf\langle\mathbf{X^n} , \mathbf{V^n} \rg\rangle - kn^2/2$.
  
 \medskip
 
By iterating the above procedure to repeatedly lower the dot product $\lf\langle\mathbf{X^n} , \mathbf{V^n} \rg\rangle$, we can obtain a vector $\mathbf{K_n} = (K_1n, \ddd, K_nn)$ with the following properties (here is where we require that $n \ge 27$).
\smallskip
\begin{enumerate}[nosep,label=(\roman*)]
\item For all $i$, $|K_i| \le 9A_n$, $- K_i = K_{n + 1 - i}$, and $K_i = 0$ if $i \in \lf(\frac{n+1}4, \frac{3(n+1)}4\rg)$.
\item $\lf\langle\mathbf{X^n} , \mathbf{V^n + K^n} \rg\rangle = 0.$
\end{enumerate}
\smallskip

Let $\mathbf{Z^n} = \mathbf{V^n + K^n}$. Observe that $\sum_{i=1}^n Z^n_i = 0$ since both $\sum_{i=1}^n K_i = 0$ and $\sum_{i=1}^n V^n_i = 0$. The fact that $\sum_{i=1}^n V^n_i = 0$ follows since $\mathbf{V^n} + \mathbf{c}$ is a permutation of the vector $(1, 2, \ddd, n)$.
For any $t$, we have that 
\begin{align*}
||F_n(t) - (\mathbf{X^n} \cos(\pi t) + \mathbf{Z^n} \sin(\pi t) + \mathbf{c})||_\infty = \max_{i \in [1, n]} |V^n_i \sin(\pi t) - Z^n_i \sin(\pi t)| \le 9 A_n n.
\end{align*}
We can now define $\mathbf{W^n} = M_n\mathbf{Z^n}$, where $M_n$ is a random constant chosen so that $||\mathbf{W^n}||^2_2 = (n^3 - n)/12$. Using the definition of $C_n$ in \eqref{E:C-nep}, for every $t \in [0, 1]$, we have that
 \begin{equation}
 \label{E:fC}
 \begin{split}
||F_n(t) - C_n(t)||_\infty &\le 9A_n n +  \max_{i \in [1, n]} |W^n_i \sin(\pi t) - Z^n_i \sin(\pi t)| \\
&< 9 A_n n + |M_n - 1|\lf[\frac{n + 1}2 + 9 A_n n\rg].
\end{split}
\end{equation}
In the last inequality, we have used that $|Z^n_i - V^n_i| \le 9 A_n n$ to get that $|Z_{n, i}| \le (n+1)/2 + 9A_n n$. The inequality \eqref{E:fC} implies that
\begin{equation}
\label{E:d-inf-comp}
\frac{d_\infty(C_n, \close{\sig}^n)}n \le 9 A_n + |M_n - 1|\lf[\frac{n + 1}{2n} + 9 A_n \rg] + \frac{d_\infty(F_n, \close{\sig}^n)}n.
\end{equation}
Now again using that $|Z^n_i - V^n_i| \le 9 A_n n$, we have that  
\begin{align*}
\big|\;||\mathbf{Z^n}||^2_2 - ||\mathbf{V^n}||^2_2 \; \big| \le \sum_{i=1}^n |Z^n_i - V^n_i||Z^n_i + V^n_i| \le 9A_n n^2(n + 1 + 9 A_n n).
\end{align*}
Therefore since $||\mathbf{V^n}||^2_2 = (n^3 - n)/12$, we have that
$$
M_n \in  \lf[\sqrt{1 - 217A_n}, \sqrt{1 + 217A_n} \rg]
$$
whenever $A_n < 1 /217$. Lemma \ref{L:dot-prod-small} implies that $A_n \to 0$ as $n \to \infty$ in probability, and hence $M_n \to 1$ in probability. Therefore since $d_\infty(F_n, \close{\sig}^n)/n \to 0$ in probability as $n \to \infty$, \eqref{E:d-inf-comp} implies that $d_\infty(C_n, \close{\sig}^n)/n$ does as well.
\end{proof}

\section*{Acknowledgements} Thank you to B\'alint Vir\'ag for many fruitful discussions about the problem, and for many constructive comments about previous drafts. Thank you to Laurent Miclo, Maxim Arnold, and Boris Khesin for pointing out the connections between random sorting networks and fluid dynamics. Thank you to Zachary Hamaker, Svante Linusson and Robin Sulzgruber for telling me about related families of sorting processes that also appear to converge to the Archimedean limit.

\bibliographystyle{alpha}
\bibliography{LLRSNbib}
\end{document}